\documentclass[12pt,a4paper]{amsart}

\usepackage{latexsym}
\usepackage{amssymb}
\usepackage{amsmath}
\usepackage{amsthm}
\usepackage{mathrsfs}
\usepackage{lunadiagrams-KM-4}
\usepackage{color}

\makeatletter
\def\@settitle{\begin{flushleft}%
  \baselineskip14\p@\relax
    \normalfont\Large\bf%<- NEW
  \@title
  \end{flushleft}%
}
\def\section{\@startsection{section}{1}%
  \z@{.7\linespacing\@plus\linespacing}{.5\linespacing}%
  {\normalfont\bf}}

\def\@setauthors{%
  \begingroup
  \def\thanks{\protect\thanks@warning}%
  \trivlist
  \footnotesize \@topsep30\p@\relax
  \advance\@topsep by -\baselineskip
  \item\relax
  \author@andify\authors
  \def\\{\protect\linebreak}%
  \larger\sc\authors%
  \ifx\@empty\contribs
  \else
    ,\penalty-3 \space \@setcontribs
    \@closetoccontribs
  \fi
  \endtrivlist
  \endgroup
}

\makeatother

\usepackage{exscale}
\usepackage[english]{babel}
\usepackage{verbatim}
\usepackage{fullpage}
\usepackage{fix-cm}
\usepackage{enumitem}
\usepackage[colorlinks, breaklinks, allcolors=blue]{hyperref}

\swapnumbers
\theoremstyle{plain}
\newtheorem{theorem}{Theorem}[section]
\newtheorem{lemma}[theorem]{Lemma}
\newtheorem{proposition}[theorem]{Proposition}
\newtheorem{corollary}[theorem]{Corollary}
\newtheorem{conjecture}[theorem]{Conjecture}

\theoremstyle{definition}
\newtheorem{definition}[theorem]{Definition}
\newtheorem{remark}[theorem]{Remark}
\newtheorem{example}[theorem]{Example}

\newcommand{\CC}{\mathbb{C}}
\newcommand{\ZZ}{\mathbb{Z}}
\newcommand{\QQ}{\mathbb{Q}}
\newcommand{\RR}{\mathbb{R}}
\newcommand{\PP}{\mathbb{P}}

\newcommand{\Ga}{\mathbb{G}_\text a}
\newcommand{\Gm}{\mathbb{G}_\text m}

\newcommand{\fu}{\mathfrak{u}}
\newcommand{\fp}{\mathfrak{p}}
\newcommand{\fn}{\mathfrak{n}}
\newcommand{\fii}{\mathfrak{i}}

\DeclareMathOperator{\Hom}{Hom}
\DeclareMathOperator{\codim}{codim}

\DeclareMathOperator{\Span}{span}

\DeclareMathOperator{\Diag}{diag}
\DeclareMathOperator{\Exp}{Exp}
\DeclareMathOperator{\Ad}{Ad}
\DeclareMathOperator{\Divf}{div}
\DeclareMathOperator{\Ord}{ord}
\DeclareMathOperator{\Rank}{rk}
\newcommand{\Chars}{{\mathcal X}}
\newcommand{\Div}{{\mathcal D}}
\newcommand{\Fan}{{\mathcal F}}

\newcommand{\Col}{\Delta}
\newcommand{\Colemb}{{\mathcal D}}
\newcommand{\Cone}{{\mathcal C}}
\newcommand{\A}{{\mathbf A}}
\newcommand{\SL}{\mathrm{SL}}
\newcommand{\GL}{\mathrm{GL}}
\newcommand{\Sp}{\mathrm{Sp}}

\newcommand{\kmG}{{\mathcal G}}
\newcommand{\kmB}{{\mathcal B}}
\newcommand{\kmT}{{\mathcal T}}
\newcommand{\kmP}{{\mathcal P}}
\newcommand{\kmQ}{{\mathcal Q}}

\newcommand{\kmU}{{\mathcal U}}
\newcommand{\kmH}{{\mathcal H}}
\newcommand{\kmK}{{\mathcal K}}

\usepackage{parskip}

\usepackage{microtype}

\begin{document}
\title{Spherical subgroups of Kac-Moody groups and transitive actions on spherical varieties}
\author{Guido Pezzini}
\address{Dipartimento di Matematica ``G.\ Castelnuovo'', Universit\`a degli Studi di Roma ``La Sapienza'', Piazzale A.\ Moro 5, 00185 Roma, Italy}
\email{pezzini@mat.uniroma1.it}

\begin{abstract}
We define and study {\em spherical subgroups of finite type} of a Kac-Moody group. In analogy with the standard theory of spherical varieties, we introduce a combinatorial object associated with such a subgroup, its {\em homogeneous spherical datum}, and we prove that it satisfies the same axioms as in the finite-dimensional case. Our main tool is a study of varieties that are spherical under the action of a connected reductive group $L$, and come equipped with a transitive action of a group containing $L$ as a Levi subgroup.
\end{abstract}

\maketitle

\section{Introduction}
Let $G$ be a connected reductive complex algebraic group. A subgroup $H$ of $G$ is {\em spherical} if a Borel subgroup of $G$ has a dense orbit on $G/H$; in this case $G/H$ is called a {\em spherical} homogeneous space. Over the last 30 years the theory of spherical subgroups has been widely developed, as has the general theory of {\em spherical varieties}, i.e.\ normal $G$-varieties where a Borel subgroup of $G$ has a dense orbit.  This theory has strong interactions with the representation theory of reductive groups, with the geometry of flag varieties and symmetric spaces, and also with other themes such as Hamiltonian actions of compact Lie groups on symplectic manifolds.

The present paper introduces a new research project, aimed at bringing this theory to the infinite-dimensional setting of Kac-Moody groups. The first task is to choose a generalization of the definition of spherical subgroups, which in the classical case can be given in several different ways. Here we proceed as follows: we say that a subgroup $\kmH$ of a Kac-Moody group $\kmG$ is {\em spherical of finite type} if is contained in a parabolic subgroup $\kmP\subseteq\kmG$ of negative sign such that $\kmP$ has a finite-dimensional Levi subgroup $L$, the group $\kmP$ acts on $\kmP/\kmH$ via a finite-dimensional quotient $P$, and $\kmP/\kmH$ is a (finite-dimensional) spherical variety under the action of $L$ (see Sections~\ref{s:KM} and \ref{s:KMsph} and Definition~\ref{def:spherical} for details).

These assumptions are particularly effective because the intermediate subgroup $\kmP$ provides a direct connection with the standard theory of spherical varieties. Our point of view is that the ``symmetries'' of $\kmG/\kmH$ given by the action of $\kmG$ are essentially already encoded in the finite-dimensional variety $\kmP/\kmH$ viewed as a spherical $L$-variety equipped with the additional automorphisms given by the unipotent radical of $P$.

Our main results are the definition of a combinatorial object called the {\em homogeneous spherical datum} of $\kmG/\kmH$ (see Definition~\ref{def:hsd}), generalizing the one given in \cite{Lu01} in the finite-dimensional setting, and the proof that this datum satisfies the same combinatorial properties as in the finite-dimensional case (see Theorem~\ref{thm:consistent}). We recall that these objects classify finite-dimensional spherical homogeneous spaces (see the papers \cite{Lu01,Lo09,Bra09,BP14,BP16}, and also \cite{Cu09} for a different approach).

We also show that the homogeneous spherical datum is independent of the choice of $\kmP$, and invariant under conjugation of $\kmH$ in $\kmG$ (see Theorems~\ref{thm:indep} and \ref{thm:conjugation}). These properties are not as obvious as in the classical theory.

We point out that a more encompassing study of the ``big'' homogeneous space $\kmG/\kmH$ is possible, but goes beyond the scope of the present work. For example using the ind-variety structure of $\kmG/\kmP$ it is possible to give a natural structure of ind-variety to $\kmG/\kmH$. In addition, such subgroups $\kmH$ arise as stabilizers of lines in the restricted duals of integrable $\kmG$-modules in category $\mathcal O$, a feature that is not expected to hold for other generalizations of spherical subgroups (in particular it does not hold in general for symmetric subgroups).

This contributes to the relevance of spherical subgroups of finite type and opens the way to a conjectural close relationship between the geometry of $\kmG/\kmH$ and the representations of $\kmG$, generalizing many aspects of the finite-dimensional theory. Here we do not pursue these lines, nor the relationship with other possible generalizations of the notion of spherical subgroups. They will be the subject of a forthcoming paper.

This research project is motivated by various recent developments in the literature on infinite-dimensional groups, in particular on topics that, in the classical case, are closely related to spherical varieties.

First of all the representation theory of Kac-Moody groups and the geometry of their flag manifolds are well-established central themes. Symmetric subgroups of affine Kac-Moody groups have been classified (see e.g.\ \cite{Ro94}, \cite{HG12} and references therein), the geometry of the corresponding symmetric spaces and of certain compactifications are the subject of current investigation (see \cite{Fr11}, \cite{So12}). Research on the space of loops on a spherical variety  has led to a link between spherical varieties and Langlands duality (see \cite{GN10}). Finally, the techniques used for solving Delzant's conjecture on multiplicity free Hamiltonian manifolds (see \cite{Kn11}), where affine smooth spherical varieties have a central role, can be applied also to multiplicity free Hamiltonian actions of loop groups (see \cite{Kn16}).

We also mention that our results on $\kmP/\kmH$ intersect the study of the automorphism groups of spherical varieties, which is an ongoing research area (see \cite{LP14}, \cite{Pe12}). In particular, we expect that our results will provide new insight on the theory of {\em log-homogeneous varieties}, introduced by M.\ Brion in \cite{Br07}. For example, Corollary~\ref{cor:Ltoroidaln} is analogous to \cite[Theorem~3.2.1]{Br07}, and a paper on the relationship between log-homogeneous varieties and {\em $P$-toroidal} varieties defined here in Section~\ref{s:emb} is in preparation.

In the first part of the paper, Sections~\ref{s:invariantsfinite}--\ref{s:examplesfinite}, we develop the basic tools to study $\kmP/\kmH$. Here we work in even greater generality: we start with a connected linear algebraic group $P$ with Levi subgroup $L$, and we consider a homogeneous space $P/H$ with the only assumption that it is a spherical $L$-variety. In this generality we define some invariants associated with the action of $P$ on $P/H$, the most relevant being its {\em spherical roots} (see Definition~\ref{def:sphroots}), and we establish relationships between the combinatorics of these objects and the geometry of $P$-equivariant compactifications of $P/H$.

In the second part of the paper, Kac-Moody groups enter into the picture: we start now with $\kmH\subseteq\kmP\subseteq\kmG$ as above, and identify $\kmP/\kmH$ with a homogeneous space $P/H$ where $P$ is a finite-dimensional quotient of $\kmP$. We discuss the invariants of $P/H$ in this case and derive for them further combinatorial properties essentially from the structure of the roots of $\kmP$ and the Weyl group of $\kmG$. In Section~\ref{s:hsd} we collect the invariants we have defined and introduce the homogeneous spherical datum of $\kmG/\kmH$, showing that it satisfies the same axioms Luna had introduced in \cite{Lu01} in the finite-dimensional case. In Sections~\ref{s:conjugation} we prove invariance of the homogeneous spherical datum under conjugation of $\kmH$ in $\kmG$, in Section~\ref{s:conj} we discuss possible further developments, and finally Section~\ref{s:examples} is devoted to some examples.

\section*{Acknowledgments}
I am very grateful to Friedrich Knop for his help and support, for many suggestions and for sharing his ideas on the subject of this paper and on possible further developments. I thank Paolo Bravi, Michel Brion, Domingo Luna, and Bart Van Steirteghem  for stimulating discussions and suggestions, and a referee for corrections and useful remarks which led to several improvements (in particular they suggested the statement of part (\ref{cor:Ltoroidaln:equiv}) of Corollary~\ref{cor:Ltoroidaln}). This research has been supported by the DFG Schwerpunktprogramm 1388 -- Darstellungstheorie.

\section{Notation and basic assumptions}
In this paper all varieties and algebraic groups are defined over the field $\CC$ of complex numbers, varieties are always assumed to be irreducible, and subgroups of algebraic groups are assumed to be closed.

If $G$ is an affine algebraic group then we denote by $\Chars(G)$ the group of its characters, by $G^u$ its unipotent radical, by $G^\circ$ the connected component containing the neutral element $e\in G$. If $K$ is a subgroup of $G$ and $g\in G$, we set ${}^gK= gKg^{-1}$. If $V$ is a $G$-module then $V^{(G)}$ denotes the set of common eigenvectors of all elements of $G$ (in particular $0\notin V^{(G)}$), and if $\chi\in\Chars(G)$ then we set
\[
V^{(G)}_\chi = \{ v\in V\smallsetminus \{0\} \;|\; gv = \chi(g)v \;\forall g\in G \}
\]
Unless otherwise stated, we denote the Lie algebra of a group with the corresponding lower case fraktur letter.

If $X$ is a variety, by a {\em discrete valuation} on $\CC(X)$ we mean a map $v\colon\CC(X)\smallsetminus\{0\}\to \QQ$ such that $v(fg) = v(f)+v(g)$ for all $f,g\in \CC(X)\smallsetminus\{0\}$, such that $v(f+g)\geq \min\{v(f),v(g)\}$ whenever $f,g,f+g\in \CC(X)\smallsetminus\{0\}$, such that $v(f)=0$ if $f$ is constant, and such that the image of $v$ is a discrete (additive) subgroup of $\QQ$. If $G$ is a group and $X$ a $G$-variety, we denote by $\Div(X)^G$ the set of $G$-stable prime divisors of $X$. 

We denote by $P$ a connected affine algebraic group, we choose a Levi subgroup $L$ of $P$, a Borel subgroup $B$ of $L$ and a maximal torus $T$ of $B$. We denote by $S_P$ the simple roots of $L$ with respect to these choices, and we denote by $B_-$ the Borel subgroup of $L$ opposite to $B$ with respect to $T$.

For any character $\alpha$ of $T$ we denote by $\fp_\alpha$ the associated root space of $\fp$ if $\alpha$ is a root of $P$, and we set $\fp_\alpha=0$ otherwise. The sum
\[
\bigoplus_{n>0} \fp_{n\alpha}
\]
is a nilpotent Lie subalgebra of $\fp$, and we denote by $U_\alpha=\Exp(\bigoplus_{n>0} \fp_{n\alpha})$ the corresponding unipotent subgroup of $P$.

We come to our crucial assumption: unless otherwise stated, in the whole paper $H$ denotes a subgroup of $P$ such that $P/H$ is a {\em spherical} $L$-variety, i.e.\ it has a dense $B$-orbit.

\section{Invariants of spherical varieties}\label{s:invariantsfinite}

We begin by recalling some notions and notations of the standard theory of spherical varieties. With any $B$-variety $X$ one associates the lattice
\[
\Xi_B(X) = \left\{ \lambda\in\Chars(T) \;\middle\vert\; \CC(X)^{(B)}_\lambda \neq \varnothing \right\},
\]
whose rank is by definition the {\em rank} of $X$, and the vector space
\[
N_B(X) =  \Hom_\ZZ(\Xi_B(X),\QQ).
\]
Any discrete valuation $v$ of $\CC(X)$ can be restricted to the subset $\CC(X)^{(B)}$. If $X$ has an open $B$-orbit then this yields a well-defined element $\rho(v)$ of $N_B(X)$, by identifying $\Xi_B(X)$ with the multiplicative group $\CC(X)^{(B)}$ modulo the constant functions. If $X$ is normal then this applies in particular to the valuation associated with any $B$-stable prime divisor $D$ of $X$, and in this case we denote simply by $\rho(D)$ (or $\rho_X(D)$) the corresponding element of $N_B(X)$.

These definitions apply in particular if $X$ is a spherical $L$-variety. In this case the set $\Div(X)^B\setminus\Div(X)^L$ is called the set of {\em colors} (or {\em $L$-colors}) of $X$, and is denoted by $\Delta_L(X)$. If $K$ is the stabilizer in $L$ of a point in the open $L$-orbit of $X$, then $X$ is called an {\em embedding} of its open orbit $L/K$. Intersecting with $L/K$ gives a bijection between $\Delta_L(X)$ and $\Div(L/K)^B$.

\section{Preliminaries on the geometry of $P/H$}

The case where $H$ is not contained in any proper parabolic subgroup of $P$ will have a prominent role in the paper. We discuss in this section some consequences of this assumption.

\begin{definition}
Let $G$ be a connected reductive algebraic group, and $K$ a subgroup. Then $K$ is {\em very reductive} if it is not contained in any parabolic subgroup of $G$.
\end{definition}

Recall that a very reductive subgroup of a reductive group is reductive.

The following first lemma is known, and will be useful in later proofs. We recall that, given $G$ a reductive group and $V$ a $G$-module, if $V$ is a spherical $G$-variety then it is also called a {\em spherical module}.

\begin{lemma}\label{lemma:PmodH}
Assume that $H$ is not contained in any parabolic subgroup of $P$. Then $H^u\subseteq P^u$, and up to conjugating $H$ in $P$ we can assume that $H$ has a Levi subgroup $K$ contained in $L$. Under this assumption we have that
\begin{enumerate}
\item\label{lemma:PmodH:veryred} $K$ is a spherical very reductive subgroup of $L$, and the homogeneous space $P/H$ has a unique closed $L$-orbit isomorphic to $L/K$;
\item\label{lemma:PmodH:bundle} $P/H$ is an affine variety, $L$-equivariantly isomorphic to the vector bundle $L\times^K P^u/H^u$ via the morphism
\[
\begin{array}{lccc}
\varphi\colon & L\times^K P^u/H^u & \to & P/H \\
& [l, pH^u] & \mapsto & lpH
\end{array}
\]
where the action of $K$ on $P^u/H^u$ is induced by the conjugation action on $P^u$;
\item\label{lemma:PmodH:module} $P^u/H^u$ is a spherical module under the action of $K^\circ$.
\end{enumerate}
\end{lemma}
\begin{proof}
The image of $H$ via the quotient map $P\to P/P^u$ is not contained in any proper parabolic subgroup of $P/P^u$, otherwise $H$ would be contained in some proper parabolic subgroup of $P$. This shows that the image of $H$ is reductive, in other words $H^u$ is mapped to the trivial subgroup of $P/P^u$. This shows that $H^u\subseteq P^u$.

Let $K$ be a maximal reductive subgroup of $H$. Replacing if necessary $H$ and $K$ with conjugates in $P$, we can suppose that $K\subseteq L$. Then the map $\varphi$ is defined and is an isomorphism (see e.g.\ \cite[Lemma~3.5.10]{Lo09}). In particular $P/H$ is affine, its unique closed $L$-orbit is isomorphic to $L/K$. and the quotient $P^u/H^u$ is an affine space, spherical under the action of $K^\circ$ (for this last assertion see e.g.\ \cite[Corollary~2.2]{KVS06}).

Finally, the homogeneous space $L/K$ is spherical under the action of $L$, being an $L$-orbit of the spherical $L$-variety $P/H$. So $K$ is a spherical subgroup of $L$. It is also equal to the image of $H$ under the projection $P\to P/P^u$, if one identifies $P/P^u$ with $L$, and this implies that $K$ is very reductive in $L$.
\end{proof}

\begin{lemma}\label{lemma:PmodHdiv}
Let $H$ and $K$ satisfy the assumptions of Lemma~\ref{lemma:PmodH}, and identify $P/H$ with $L\times^K P^u/H^u$ via the map $\varphi$ of Lemma~\ref{lemma:PmodH}. Let $\pi\colon L\times^K P^u/H^u\to L/K$ be the projection, and identify $P^u/H^u$ with the fiber $\pi^{-1}(eK)$. Up to conjugating $K$ by an element of $L$ we may assume that $BK$ is open in $L$. Then for any $(B\cap K)^\circ$-stable prime divisor $E$ of $P^u/H^u$ the closure $D=\overline{BE}$ is a $B$-stable but not $P^u$-stable prime divisor of $P/H$.
\end{lemma}
\begin{proof}
Since $K$ is a spherical subgroup of $L$, up to replacing it by a conjugate in $L$ we may assume that $BK$ is open in $L$. Let $x\in P^u/H^u$: then the $B$-orbit of $\pi(x)=eK$ is isomorphic to $B/(B\cap K)$ and is dense in $L/K$.

The $B$-orbit $Bx$ is isomorphic to the homogeneous space $B/B_x$, where the stabilizer $B_x$ of $x$ in $B$ is a subgroup of $B\cap K$. Then
\[
\dim Bx= \underbrace{\dim B - \dim (B\cap K)}_{=\dim L/K} + \underbrace{\dim(B\cap K)^\circ - \dim B_x}_{=\dim (B\cap K)^\circ x},
\]
which, together with $\dim P/H=\dim L/K+\dim P^u/H^u$, shows that
\begin{equation}\label{lemma:PmodHdiv:codim}
\codim_{P/H} Bx = \codim_{P^u/H^u} (B\cap K)^\circ x.
\end{equation}

Now, the map $Y\mapsto Y\cap P^u/H^u$ induces a bijection between the set of those $B$-orbits of $P/H$ that meet $P^u/H^u$ and the set of $(B\cap K)$-orbits of $P^u/H^u$. This shows in particular that $P^u/H^u$ has finitely many $(B\cap K)$-orbits; each such orbit is a finite union of $(B\cap K)^\circ$-orbits, so $(B\cap K)^\circ$ also has finitely many orbits on $P^u/H^u$.

Hence, for any $(B\cap K)^\circ$-stable prime divisor $E$ of $P^u/H^u$, there exists a $(B\cap K)^\circ$-orbit that is dense in $E$. If $x$ lies on this orbit, then (\ref{lemma:PmodHdiv:codim}) yields that $Bx$ has codimension $1$ in $P/H$. In other words $\overline {Bx}$ is a $B$-stable prime divisor of $P/H$, contained in $\overline{BE}$. Observing that $E= \overline{(B\cap K)^\circ x}\subseteq \overline {Bx}$ we conclude $BE\subseteq \overline{Bx}$ since $\overline{Bx}$ is $B$-stable, and even $\overline {BE}=\overline{Bx}$ since $\overline{Bx}$ is closed.

It remains to show that $D=\overline{BE}$ is not $P^u$-stable. Suppose for sake of contradiction that $D$ is $P^u$-stable: then it contains the $P^u$-orbit of $x$ in $P/H$. This $P^u$-orbit is $P^u/H^u$, since $p\cdot \varphi([e, eH^u])= pH = \varphi([e, pH^u])$ for all $p\in P^u$.

On the other hand, the open $B$-orbit of $P/H$ is mapped via $\pi$ onto the open $B$-orbit $B\cdot eK$ of $L/K$. It follows that the open $B$-orbit of $P/H$ intersects $\pi^{-1}(eK) = P^u/H^u$, but then it intersects $D$ too: contradiction.
\end{proof}

\begin{corollary}\label{cor:PumovesD}
Suppose that $H$ doesn't contain $P^u$ and $H$ is not contained in any proper parabolic subgroup of $P$. Then there exists an element $D\in \Div(P/H)^B$ not stable under $P^u$.
\end{corollary}
\begin{proof}
We may assume that $H$ and a Levi subgroup $K\subseteq H$ are as in Lemma~\ref{lemma:PmodH}. Since here $H$ doesn't contain $P^u$, we have that $P^u/H^u$ is a non-zero $(B\cap K)^\circ$-module. The group $(B\cap K)^\circ$ being solvable, the quotient $P^u/H^u$ has a $(B\cap K)^\circ$-stable prime divisor $E$. Then $D=\overline{BE}$ is a $B$-stable but not $P^u$-stable prime divisor of $P/H$ by Lemma~\ref{lemma:PmodHdiv}.
\end{proof}

\begin{lemma}\label{lemma:fiberoverflag}
Suppose that $H$ is contained in a parabolic subgroup $Q$ of $P$ containing $B_-$, and denote $\pi\colon P/H\to P/Q$ the natural morphism. Denote $M$ the Levi subgroup of $Q$ contained in $L$ and containing $T$, and set $B_M=M\cap B$, which is a Borel subgroup of $M$ containing $T$. Then the fiber $Z=\pi^{-1}(eQ)$ is spherical under the action of $M$, and $\Xi_B(P/H)=\Xi_{B_M}(Z)$.
\end{lemma}
\begin{proof}
Denote by $R$ the parabolic subgroup of $L$ containing $B_-$ and such that $Q=RP^u$, and denote by $X_0$ the open $L$-orbit of $P/H$. Then $\pi$ restricts to an $L$-equivariant morphism $\pi_0\colon X_0\to L/R$, and the lemma follows from \cite[Lemma~3.5.5, part~(1)]{Lo09}.
\end{proof}

\section{Toroidal embeddings}\label{s:emb}

We recall that a spherical $L$-variety $X$ is {\em toroidal} if no $L$-color contains an $L$-orbit; to emphasize the acting group we will rather use the term {\em $L$-toroidal} in this paper. This notion, and more generally the one of equivariant embeddings, adapt quite directly to our setting.

\begin{definition}
A {\em $P$-embedding} of $P/H$ is a normal $P$-variety $X$ together with a $P$-equivariant open embedding $P/H\hookrightarrow X$. In this case we identify $P/H$ with its image in $X$. If, in addition, for all $D\in \Div(P/H)^B$ the closure of $D$ in $X$ doesn't contain any $P$-orbit, then $X$ is called a {\em $P$-toroidal embedding} of $P/H$. 
\end{definition}

\begin{lemma}\label{lemma:toroidalexist}
Let $Y_1$ and $Y_2$ be $P$-embeddings of $P/H$. Then there exist a $P$-toroidal embedding $X$ of $P/H$ and proper, $P$-equivariant maps $X\to Y_i$ extending the identity on $P/H$ for all $i$. In particular, if $Y_1$ and $Y_2$ are complete then $X$ is complete, and there exist complete $P$-toroidal embeddings of $P/H$.
\end{lemma}
\begin{proof}
The proof of \cite[Lemma~6.2]{Kn91} yields essentially the desired result. Let $\pi\colon P\to P/H$ be the quotient map, and choose $f\in\CC[P]^{(B\times H)}$ vanishing on $\pi^{-1}(D)$ for all $D\in \Div(P/H)^B$. Let us check that such $f$ exists. Since $B$ has an open orbit on $P/H$, we may assume that $BH$ is open in $P$. Denote by $E$ the sum of all $D\in \Div(P/H)^B$. The divisor $E$ is Cartier since $P/H$ is smooth, and we can fix a positive integer $n$ such that the associated line bundle $\mathcal L=\mathcal O(nE)$ on $P/H$ is $P$-linearizable (see \cite[Proposition~2.4]{KKLV89}). The divisor $nE$ is the zero set of a global section $s\in H^0(P/H,\mathcal L)$. Consider the corresponding element $f\in\CC[P]$ under the isomorphism $H^0(P/H,\mathcal L)\cong (\CC[P]^{(H)}_\chi \cup\{0\})$ (for a suitable character $\chi$ of $H$) induced by pulling back to $P$ the line bundle $\mathcal L$ along $\pi$ (see \cite[Proposition~2.4]{Ti11}): then the zero set of $f$ is $\pi^{-1}(D)$. Since $\pi^{-1}(D)$ is $B$-stable, for any $b\in B$ the quotient $bf/f\in \CC(P)$ is a regular invertible function on $P$, i.e.\ a scalar constant $c_b\in\CC$ times a character $\lambda_b$ of $P$ by \cite[Proposition~1.2]{KKV89}. Since $B$ is connected, the character $\lambda_b$ doesn't depend on $b$, so we denote it by $\lambda$. Since $1 = f(eb)/f(b) = c_e\lambda(b)$ for all $b\in B$, the character $\lambda$ is trivial, and it follows that $b\mapsto c_b$ is a character of $B$. In other words for all $b\in B$ and all $h\in H$ we have $f(bh) = f(b) \chi(h) = c_{b^{-1}} f(e) \chi(h)$. Since $BH$ is open in $P$, this shows $f\in \CC[P]^{(B\times H)}$.

Let $V$ be the $P$-submodule of $\CC[P]$ generated by $f$; this induces a $P$-equivariant morphism $\varphi\colon P/H\to \PP(V^*)$. Let $X'$ be the closure of the image.

Regard $f$ as a linear function on $V^*$, and let $Z$ be its set of zeros in $\PP(V^*)$. Since $V$ is generated by $f$ as a $P$-module, the elements $pf$ for $p$ varying in $P$ are not contained in any hyperplane of $V$. In other words they have no common zero in $\PP(V)$, so $\bigcap_{p\in P} pZ=\varnothing$. This implies that $Z$ doesn't contain any $P$-orbit. Since $\varphi(D)$ is contained in $Z$ for all $D\in \Div(P/H)^B$ in $\PP(V^*)$, we deduce that the closure of $\varphi(D)$ in $X'$ doesn't contain any $P$-orbit.

Now we identify $P/H$ with its image in $X'\times Y_1\times Y_2$ under the map $(\varphi, \Diag)$, and let $X$ be the normalization of the closure of the image. The map $\varphi$ extends to a $P$-equivariant morphism $X\to X'$, also denoted $\varphi$. The variety $X$ is a $P$-embedding of $P/H$, and the projections onto $Y_1$ and $Y_2$ are proper morphisms extending the identity on $P/H$. In addition $X$ is $P$-toroidal: if the closure of any $D\in \Div(P/H)^B$ in $X$ contained a $P$-orbit, then the closure of $\varphi(D)$ in $X'$ would contain a $P$-orbit, contradiction.

The second statement is obvious, and to show the last statement we take any two $P$-equivariant completions $Y_1$ and $Y_2$ of $P/H$ (by Chevalley's theorem $P/H$ is a $P$-stable subvariety of the projective space $\PP(V)$ of a $P$-module $V$, it is enough to take $Y_1=Y_2$ equal to the normalization of the closure of $P/H$ in  $\PP(V)$), and apply the first statement.
\end{proof}

\begin{lemma}\label{lemma:extend}
Let $K$ be a subgroup of $P$ containing $H$. There exist a $P$-toroidal embedding $X$ of $P/H$ and a proper $P$-equivariant morphism $X\to P/K$ extending the natural map $\pi\colon P/H\to P/K$.
\end{lemma}
\begin{proof}
It is enough to apply the proof of Lemma~\ref{lemma:toroidalexist} with $P/K$ instead of $Y_1\times Y_2$ and with $(\varphi,\pi)$ instead of $(\varphi,\Diag)$.
\end{proof}

The next lemma is fundamental in this paper and carries most of the informations we will need about the action of $P^u$ on $P/H$. It resembles results by Demazure (see e.g.\ \cite[Section~3.4]{Oda88}) on automorphism groups of toric varieties, and also results by Bien and Brion in \cite{BiBr96} on automorphism groups of complete toroidal spherical varieties. However, the results in \cite{BiBr96} cannot be used here directly, since in general a $P$-toroidal embedding of $P/H$ is not $L$-toroidal.

\begin{lemma}\label{lemma:demazure}
Let $X$ be a $P$-embedding of $P/H$, let $-\alpha$ be a root of $P$ and let $D\in \Div(X)^B$. Let $U'_{-\alpha}$ be a one-dimensional subgroup of $U_{-\alpha}$ stable under conjugation by $T$, and suppose that $D$ is not stable under the action of $U'_{-\alpha}$, but it is stable under the action of the commutator $(U'_{-\alpha}, B^u)$. Let $n$ be the positive integer such that $T$ acts on $U'_{-\alpha}$ via the character $-n\alpha$. Then $n\alpha\in \Xi_B(X)$ and $\langle\rho(D),n\alpha\rangle =1$. For all $E\in\Div(X)^B$ different from $D$ we have $\langle \rho(E),n\alpha\rangle \leq 0$. If in addition $U_{-\alpha}'$ is centralized by $B^u$ then any such $E$ is stable under $U'_{-\alpha}$.
\end{lemma}
\begin{proof}
First we claim that for some $v_0\in U'_{-\alpha}$ the group $K=v_0^{-1}(\ker\alpha)B^uv_0$ has an open orbit on $D$. To prove this, consider $x\in D$ and $v_0\in U'_{-\alpha}$ such that $v_0x$ is in the open $B$-orbit of $X$. Since $(\ker\alpha)B^u$ has codimension $1$ in $B$, the orbit $(\ker\alpha)B^uv_0x$ has codimension at most $1$ in $X$. Let $y$ be in this orbit, so there exist $t\in \ker\alpha$ and $b\in B^u$ such that $y = tbv_0x$. Then $y=v_0tbcx$ with $c\in(U'_{-\alpha}, B^u)$. It follows that $v_0^{-1}y\in D$, whence the claim.

Consider now the morphism $\varphi\colon U'_{-\alpha}\times D\to X$ induced by the action. It is dominant, since $D$ has codimension $1$ and not stable under $U'_{-\alpha}$. We claim that $\varphi$ is also generically injective, and for this it is enough to prove that for general $x\in D$ and $v\in U'_{-\alpha}$ the condition $vx \in D$ implies $v=e$.

Let $x$ be in the open $K$-orbit of $D$, suppose that $vx\in D$, let $b\in (\ker\alpha)B^u$ and consider the element $v(v_0^{-1}bv_0)x$. We have $v(v_0^{-1}bv_0)=cbv$ where $c\in (U'_{-\alpha}, B^u)$. It follows that $vv_0^{-1}bv_0x = cbvx\in cbD$, and the latter is equal to $D$ since $b$ and $c$ stabilize $D$. Since $Kx$ is open in $D$, we conclude that $vD=D$. The stabilizer of $D$ in $U'_{-\alpha}\cong \Ga$ is a proper subgroup, hence trivial; in other words $v=e$.

As a consequence the map $\varphi$ is birational, $U'_{-\alpha}$-equivariant, where we let this group act on $U'_{-\alpha}\times D$ by left translation on the first factor, and also $T$-equivariant, where we let $T$ act diagonally: by conjugation on $U_{-\alpha}'$ and as usual on $D$. The map $\varphi$ induces an isomorphism of an open subset of $U'_{-\alpha}\times D$ onto an open subset of $X$. By $U'_{-\alpha}$-equivariance, we may assume that the former is of the form $U_{-\alpha}'\times D_0$ for $D_0$ an open subset of $D$, and we denote by $X_0$ the image of $U_{-\alpha}'\times D_0$.

The projection of $U'_{-\alpha}\times D$ on the first factor induces a $T$-equivariant and $U'_{-\alpha}$-equivariant rational map $f_1\colon X\dashrightarrow U'_{-\alpha}$. On the other hand $v D$ is $B^u$-stable for all $v\in U'_{-\alpha}$, because $uvD = v(v^{-1}uvu^{-1})D = vD$ for all $u\in B^u$. Therefore $f_1$ is $B$-equivariant, where $B$ acts on $X$ as usual and we let $B^u$ act trivially on $U'_{-\alpha}$. Notice that at this point we have defined an action of $B$ on $U'_{-\alpha}\cong \CC$ that is simply the linear action via the character $-n\alpha$. So $f_1$ is actually an element of $\CC(X)^{(B)}_{n\alpha}$, which proves $n\alpha\in\Xi_B(X)$. Seen as a rational function on $U_{-\alpha}'\times D$, the function $f_1$ has a simple zero on $D$, which shows $\langle\rho(D),n\alpha\rangle =1$. 

Consider now a $B$-stable prime divisor $E\neq D$: we want to show that $\langle \rho(E),n\alpha\rangle \leq 0$. Suppose first that the $E$ doesn't intersect $X_0$. Then $E$ is an irreducible component of $\overline{X\smallsetminus X_0}$ therefore it is $U_{-\alpha}'$-stable. Now, any zero of the rational function $f_1$ on $X$ is mapped by $f_1$ to $e\in U'_{-\alpha}$, hence it is not $U'_{-\alpha}$-stable. This shows that $\langle \rho(E),n\alpha\rangle \leq 0$.

We are left with the case where $E$ intersects $X_0$. But then, since $\varnothing\neq E\cap X_0 \neq D_0$, the divisor $E$ is not a zero of $f_1$, so again $\langle \rho(E),n\alpha\rangle \leq 0$.

Let us show the last assertion of the lemma, so we assume that $U_{-\alpha}'$ is centralized by $B^u$, and we let $E$ be a $B$-stable prime divisor different from $D$. We claim that in this case $E$ cannot intersect $X_0$, which concludes the proof by the same argument as above, namely $E$ is an irreducible component of $\overline{X\smallsetminus X_0}$ therefore it is $U_{-\alpha}'$-stable.

To show the claim, we notice that the assumption that $U_{-\alpha}'$ is centralized by $B^u$ implies that $\varphi$ is $B$-equivariant, where we let $B$ act diagonally: by conjugation on $U_{-\alpha}'$ and as usual on $D$. Hence the rational map $f_2\colon X\dashrightarrow D$ induced by the second projection is also $B$-equivariant. The pull-back of functions along $f_2$ gives an inclusion $\Xi_B(D)\subseteq \Xi_B(X)$ with $\ZZ\alpha\cap \Xi_B(D) = \{0\}$.

The proof of the claim proceeds now by contradiction: assume that $E\cap X_0\neq\varnothing$. Then the restriction $f_2|_E\colon E\dashrightarrow D$ is defined, $B$-equivariant, and dominant. This assures that $\Xi_B(D)$ is a subgroup of $\Xi_B(E)$. Since $\varnothing\neq E\cap X_0 \neq D_0$ the function $f_1$ restricts to a non-zero and $B$-semiinvariant rational function on $E$, which implies $n\alpha\in\Xi_B(E)$.

On the other hand, by $B$-equivariance, the rational map $f_2|_E$ is a morphism if restricted to the open $B$-orbit $E'$ of $E$, and has image equal to the open $B$-orbit $D'$ of $D$. The existence of this morphism, and the fact that $E'$ and $D'$ are $B$-orbits of the same dimension, yields that $E$ and $D$ have the same rank as $B$-varieties, so $\Xi_B(D)$ is a subgroup of $\Xi_B(E)$ of finite index. This implies that a non-zero multiple of $\alpha$ is in $\Xi_B(D)$: contradiction.
\end{proof}

We use now Lemma~\ref{lemma:demazure} to clarify some basic relationships between the action of $P^u$ on $P/H$ and its $B$-stable prime divisors.

\begin{corollary}\label{cor:halfspace}
Let $X$ be a $P$-embedding of $P/H$. If $P^u$ doesn't act trivially on $X$, then there exists $\beta\in \Xi_B(X)$ that is $\leq 0$ on $\rho_X(\Div(X)^P)$ but not on $\rho_X(\Div(X)^B)$.
\end{corollary}
\begin{proof}
Suppose first that no proper parabolic subgroup of $P$ contains $H$. By Corollary~\ref{cor:PumovesD} there exists $D\in \Div(P/H)^B$ not stable under $P^u$. The Lie algebra $\fu$ of $P^u$ admits a filtration $0=\fu_0\subset \fu_1\subset \fu_2\subset\ldots$ where the subspaces are $(\Ad B)$-stable subalgebras and the quotients $\fu_i/\fu_{i-1}$ are one-dimensional for all $i>0$. The existence of such filtration can be proved by induction on the dimension of $\fu$. First, one finds a filtration of $(\Ad B)$-stable subspaces of the center $\mathfrak z(\fu)$. These subspaces are subalgebras. Then one extends the filtration to the whole $\fu$ by passing to the quotient $\fu/\mathfrak z(\fu)$ and applying induction on the dimension.

For $i\geq0$ denote by $U_i$ the subgroup $\Exp(\fu_i)$ of $P$. It is a subgroup of $P^u$ stable under conjugation by $B$, normal in $U_{i+1}$, and $B^u$ acts trivially on $U_{i+1}/U_{i}$ for all $i\geq0$. Let $i>0$ be such that $U_{i-1}$ stabilizes $D$ but $U_i$ doesn't. Let $\alpha$ be the character such that $T$ acts on $\fu_{i}/\fu_{i-1}$ with eigenvalue $-\alpha$ (so $-\alpha$ is a root of $P$), let $\mathsf x\in \fu_{i}\smallsetminus\fu_{i-1}$ be an $\Ad T$-eigenvector, and consider the subgroup $U'_{-\alpha}=\Exp(\CC\mathsf x)$ of $U_i$.

Then $U'_{-\alpha}$ is a $T$-stable subgroup of $U_{-\alpha}$, it doesn't stabilize $D$ but the commutator $(U'_{-\alpha},B^u)$ does. The last assertion is true because, for $u\in U'_{-\alpha}$ and $b\in B^u$, the commutator $(u,b)$ is in $U_{i-1}$ so it stabilizes $D$. Lemma~\ref{lemma:demazure} implies the corollary with $\beta=\alpha$.

It remains to consider the case where $H$ is contained in a proper parabolic subgroup of $P$, and we denote by $Q$ such a subgroup. We may suppose that $Q$ is minimal with this property, therefore $H$ is not contained in any proper parabolic subgroup of $Q$. We may also suppose that $Q$ contains $B_-$, the Borel subgroup of $L$ opposite to $B$ with respect to $T$. This implies that $B^uQ$ is open in $P$. We denote by $M$ the Levi subgroup of $Q$ containing $T$, and in this way $B_M=B\cap M=B\cap Q$ is a Borel subgroup of $M$.

Let $Y_0=Q/H$ be the fiber over $eQ\in P/Q$ of the natural morphism $\pi\colon P/H\to P/Q$, let $Y$ be the closure of $Y_0$ in $X$ and consider the equivariant fiber bundle $X'= P\times^Q Y$ over $P/Q$. It is a $P$-embedding of $P/H$ equipped with the $P$-equivariant morphism $\varphi\colon X'\to X$ given by $[p,y]\mapsto py$. The fiber bundle $X'$ can be constructed as a closed subset of $P/Q\times X$, in such a way that $\varphi$ is the restriction to $X'$ of the second projection $P/Q\times X\to X$. Since $P/Q$ is complete, the image of $\varphi$ is closed. Then, since $\varphi$ is birational, it is also surjective.

Up to positive multiples we have $\rho_{X'}(\Div(X')^P)\supseteq \rho_{X}(\Div(X)^P)$. Indeed, for any $Z\in \Div(X)^P$ there exists $Z'\in \Div(X')^P$ mapped onto $Z$. Inside $\CC(X)=\CC(X')$, the local ring $\mathcal O_{X,Z}$ is contained in the local ring $\mathcal O_{X',Z'}$ by pulling back functions from $X$ to $X'$. Being both discrete valuation rings, both $\mathcal O_{X,Z}$ and $\mathcal O_{X',Z'}$ are maximal subrings of $\CC(X)$, hence they are equal. The equality $\Div(X')^B \smallsetminus \Div(X')^P = \Div(X)^B\smallsetminus \Div(X)^P$ also holds and is compatible with the maps $\rho_{X'}$ and $\rho_X$, if one identifies non-$P$-stable prime divisors of $X$ and of $X'$ with their respective intersections with $P/H$. Therefore we may replace $X$ by $X'$, which enables us to assume that there is a $P$-equivariant map $X\to P/Q$ extending $\pi$. We still denote by $Y$ the fiber of this map over $eQ$.

The group $Q^u$ doesn't act trivially on $Q/H$, otherwise $P^u$ being contained in $Q^u$ would act trivially on $P/H$. The first part of the proof then yields $\beta\in\Xi_{B_M}(Y)$ that is $\leq 0$ on $\rho_Y(\Div(Y)^Q)$ but not on $\rho_Y(\Div(Y)^{B_M})$.

On the other hand $P/Q$ is a complete $L$-orbit, and the point $eQ$ is in the open $B^u$-orbit. It follows that restriction of functions from $X$ to $Y$ induces an isomorphism $\Xi_B(X)\cong \Xi_{B_M}(Y)$, that any $D\in\Div(X)^P$ intersects $Y$ in a $Q$-stable prime divisor, and that any $E\in \Div(Y)^{B_M}$ is the intersection of some $D\in \Div(X)^B$ with $Y$. This identifies $\rho_X(\Div(X)^P)$ with $\rho_Y(\Div(Y)^Q)$, and identifies $\rho_Y(\Div(Y)^{B_M})$ with a subset of $\rho_X(\Div(X)^B)$, compatibly with the above isomorphism between the lattices. The corollary follows.
\end{proof}

\section{The spherical roots}\label{s:sigma}

In this section we use the theory of embeddings of spherical homogeneous spaces as presented in \cite{Kn91}, we refer to loc.cit.\ for details.

Let us recall some basic facts of this theory. Given an spherical $L$-variety $X$, with any $L$-orbit $Y\subseteq X$ one associates the polyhedral (i.e.\ finitely generated) convex cone $\Cone_Y$ generated by $\rho(D)$ where $D$ varies in the set of all $B$-stable prime divisors of $X$ that contain $Y$. The {\em colored cone} associated with $Y$ is the couple $(\Cone_Y,\Colemb_Y)$ where $\Colemb_Y$ is the set of those elements in $\Delta_L(X)$ that contain $Y$.

The set of the colored cones of all $L$-orbits of $X$ is called the {\em fan} of $X$ and is denoted here by $\Fan_L(X)$. Moreover, two $L$-orbits $Y$, $Z$ satisfy $\overline Z\supseteq Y$ if and only if $\Cone_Z$ is a face of $\Cone_Y$ and $\Colemb_Z=\Colemb_Y\cap\rho^{-1}(\Cone_Z)$.

The set of $L$-invariant valuations on $\CC(X)$ is identified via the map $\rho$ with its image $V_L(X)$ in $N_B(X)$, and $V_L(X)$ is a polyhedral convex cone of maximal dimension called the {\em valuation cone}. If $X$ is complete and toroidal then $V_L(X)$ is equal to the union of all cones $\Cone_Y$ for $Y$ varying in the set of $L$-orbits of $X$. The dual cone of $-V_L(X)$ in $\Xi_B(X)$ is generated by a set $\Sigma_L(X)$ of primitive, linearly independent elements. They are usually called the {\em spherical roots} of $X$; in this paper we will use the denomination {\em spherical roots of the $L$-action on $X$}. The equality $V_L(X)=N_B(X)$ holds if and only if $X$ is a {\em horospherical} $L$-variety, i.e.\ a maximal unipotent subgroup of $L$ stabilizes a point in the open $L$-orbit of $X$.

Finally, any colored cone $(\Cone,\Colemb)\in \Fan_L(X)$ is uniquely determined by $\Cone$ among the elements of $\Fan_L(X)$. Indeed, let $v\in V_L(X)$ be in the relative interior of $\Cone$ (such $v$ exists by \cite[Section~4, property~CC2]{Kn91}). Then by \cite[Section~4, property~CF2]{Kn91} for any $(\Cone',\Colemb')\in \Fan_L(X)$ different from $(\Cone,\Colemb)$ the point $x$ is not in the relative interior of $\Cone'$, which implies $\Cone\neq\Cone'$.

The purpose of this section is to adapt the above notions to our case, in particular to introduce a set of spherical roots related to the action on $P/H$ of $P$ instead of $L$. We accomplish this task by defining a set $V(P/H)$ that will play the role of the valuation cone, and for this we use $P$-toroidal compactifications of $P/H$. However, several details need to be worked out to prove the two main results of this section: that what we define is indeed a convex cone (Theorem~\ref{thm:V}), and that it is independent from the chosen compactification (Theorem~\ref{thm:Vindep}).

\begin{definition}
Let $X$ be a complete $P$-toroidal embedding of $P/H$. We define
\[
V(X) = \bigcup_Y \Cone_Y
\]
where $Y$ varies in the set of $L$-orbits of $X$ such that $\overline Y$ is $P$-stable.
\end{definition}

Suppose that $X$ is as in the above definition. If $Y$ is an $L$-orbit with $P$-stable closure then $\overline Y$ contains a closed $P$-orbit $Z$. The latter is complete, therefore itself a closed $L$-orbit, hence $\Cone_Y$ is a face of the maximal cone $\Cone_Z\subseteq V(X)$.

It follows that in the definition of $V(X)$ the same set is obtained if we let $Y$ vary only in the set of closed $L$-orbits that are also $P$-orbits, because this corresponds exactly to taking only those $Y$ such that $\overline Y$ is $P$-stable and $\Cone_Y$ is maximal in the fan $\Fan_L(X)$.

We provide now some general insight on $L$-orbits $Y\subseteq X$ whose closure is $P$-stable.

\begin{lemma}\label{lemma:valcontained}
Let $X$ be a complete $P$-toroidal embedding of $P/H$. Then $V(X)$ contains $\rho(\Div(X)^P)$, and is contained in the convex cone generated by $\rho(\Div(X)^P)$, which is contained in $V_L(X)$.
\end{lemma}
\begin{proof}
Let $Y\subseteq X$ be an $L$-orbit whose closure is $P$-stable. We have already recalled from \cite{Kn91} that $\Cone_Y$ is generated as a convex cone by the elements $\rho(D)$ for all $D\in \Div(X)^B$ that contain $Y$. Any such $D$ also contains some $P$-orbit since $\overline Y$ is $P$-stable, therefore $D$ is $P$-stable, because $X$ is $P$-toroidal.

We deduce that $V(X)$ is contained in the convex cone generated by $\rho(\Div(X)^P)$. Also, notice that any $D\in \Div(X)^P$ is in particular $L$-stable, so $\rho(D)\in V_L(X)$.

It remains to show that $V(X)$ contains $\rho(\Div(X)^P)$. Any $D\in \Div(X)^P$ has a dense $L$-orbit $Y$, which has then $P$-stable closure and whose cone $\Cone_Y$ is generated by $\rho(D)$. Therefore $\rho(D)$ is in $V(X)$.
\end{proof}

\begin{proposition}\label{prop:inVstable}
Let $X$ be a complete $P$-toroidal embedding of $P/H$, let $Y\subseteq X$ be an $L$-orbit, and suppose that $\Cone_Y\subseteq V(X)$. Then $\overline Y$ is $P$-stable, and $\Colemb_Y=\varnothing$.
\end{proposition}
\begin{proof}
Since $\Cone_Y\subseteq V(X)$, there exists a closed $P$-orbit $Z$ such that $\Cone_Y$ is a face of the maximal cone $\Cone_Z$. Since $X$ is $P$-toroidal, any $B$-stable prime divisor containing $Z$ is $P$-stable, in particular $\Colemb_Z=\varnothing$. The relative interior of $\Cone_Y$ meets $V_L(X)$, so the face $\Cone_Y$  of $\Cone_Z$ is itself the cone of a colored cone $(\Cone_Y,\Colemb)$ of $\Fan_L(X)$, corresponding to an $L$-orbit $Y'$ of $X$, such that $\overline{Y'}\supseteq Z$ and $\Colemb=\Colemb_Z\cap \rho^{-1}(\Cone_Y)$, i.e.\ $\Colemb = \varnothing$. Colored cones are uniquely determined by their cones, which implies $\Colemb=\Colemb_Y$ and $Y'=Y$.

It follows that $\overline Y\supseteq Z$, hence any $B$-stable prime divisor containing $Y$ is also $P$-stable. As a consequence, the colored cone of $Y$ is equal to the colored cone of $Y'$, where the latter is defined as the open $L$-orbit of the intersection $\overline{Y'}$ of all $P$-stable prime divisors containing $Y$. This implies that $Y=Y'$ and $\overline Y=\overline{Y'}$. Since $\overline{Y'}$ is by definition $P$-stable, we conclude that $\overline Y$ is $P$-stable.
\end{proof}

The following corollary is analogous to \cite[Theorem~3.2.1]{Br07}.

\begin{corollary}\label{cor:Ltoroidaln}
Let $X$ be a complete embedding of $P/H$. Then:
\begin{enumerate}
\item\label{cor:Ltoroidaln:equiv} The embedding $X$ is $P$-toroidal if and only if there exists an $L$-stable open subset $X_0\subseteq X$ such that $X_0$ is an $L$-toroidal embedding of the open $L$-orbit of $X$, and $X_0$ intersects any $P$-orbit of $X$ in a single $L$-orbit.
\item\label{cor:Ltoroidaln:neighb} If $X$ is $P$-toroidal, then any such $X_0$ contains any $L$-orbit $Y\subseteq X$ such that $\overline Y$ is $P$-stable.
\end{enumerate}
\end{corollary}
\begin{proof}
We prove the ``only if'' implication of part (\ref{cor:Ltoroidaln:equiv}). Assume that $X$ is $P$-toroidal, and define $X_0$ to be the embedding of the open $L$-orbit of $X$ with colored fan equal to
\[
\Fan_L(X_0) = \{ (\Cone,\Colemb)\in \Fan_L(X)\;|\; \Cone\subseteq V(X)\}.
\]
Then, by \cite[Section~4]{Kn91}, the embedding $X_0$ is an open subset of $X$ and an $L$-orbit $Y$ of $X$ is in $X_0$ if and only if $\Cone_Y\subseteq V(X)$. Moreover $X_0$ is $L$-toroidal, because by Proposition~\ref{prop:inVstable} we have $\Colemb=\varnothing$ for all $(\Cone,\Colemb)\in \Fan_L(X_0)$.

Let $Z$ be a $P$-orbit. Then $Z$ is $L$-stable, therefore it contains a unique dense $L$-orbit $Y$, and $\overline Y=\overline Z$ is $P$-stable. It follows that $\Cone_Y\subseteq V(X)$, hence $Y\subseteq X_0$. Suppose $X_0\cap Z$ contains an $L$-orbit $Y'\neq Y$, then $\overline{Y'}$ is not $P$-stable since $PY'=Z$ has dimension greater than $\dim Y'$. Proposition~\ref{prop:inVstable} yields that $\Cone_{Y'}\not\subseteq V(X)$, whence $Y'\not\subseteq X_0$: contradiction. Therefore $X_0\cap Z$ is a single $L$-orbit.

Let us show the ``if'' implication of part (\ref{cor:Ltoroidaln:equiv}). Assume that such $X_0$ exists, and suppose that an element $D\in \Div(X)^B$ contains a $P$-orbit $Z$: we must show that $D$ is $P$-stable. Since $X_0$ intersects $Z$, it also intersects $D$. Now $X_0\cap Z$ is $L$-stable and contained in $X_0\cap D$, so the latter is a $B$-stable prime divisor of $X_0$ containing some $L$-orbit. Since $X_0$ is $L$-toroidal, we deduce that $D$ is $L$-stable. Call $D_0$ the open $L$-orbit of $D$; it is contained in $X_0$.

If $D$ is not $P$-stable, then the $P$-orbit $PD_0$ is dense in $X$, i.e.\ it contains the open $L$-orbit of $X$, and this orbit is also contained in $X_0$. It would follow that $X_0$ intersects $PD_0$ in more than one $L$-orbit: contradiction. Therefore $D$ is $P$-stable, so $X$ is $P$-toroidal.

Part (\ref{cor:Ltoroidaln:neighb}) follows from part (\ref{cor:Ltoroidaln:equiv}). Indeed, assume that $X$ is $P$-toroidal and let $Y$ be an $L$-orbit whose closure is $P$-stable. Then $Y$ is open in the $P$-orbit $PY$, because it is open in $\overline Y$ which contains $PY$. By part (\ref{cor:Ltoroidaln:equiv}), the intersection $X_0\cap PY$ is non-empty and it is a single $L$-orbit. Being open in $PY$, it is equal to $Y$.
\end{proof}

\begin{lemma}\label{lemma:Bstable}
Let $X$ be a normal $L$-variety, let $Y\subseteq X$ be an $L$-orbit and $Z\subseteq Y$ be a $B$-stable subvariety. Then $Z$ is contained in a $B$-stable prime divisor of $X$ not containing $Y$. 
\end{lemma}
\begin{proof}
Thanks to \cite{Su74}, we may suppose that $X\subseteq \PP(M)$ where $M$ is a finite-dimensional $L$-module. Let $X'$, $Y'$ and $Z'$ be the closures in $M$ of the cones over resp.\ $X$, $Y$ and $Z$. Choose a homogeneous function $f\in \CC[Y']^{(B)}$ vanishing on $Z'$. By \cite[Theorem~2.1]{Kn91} we can extend $f$ to a $B$-semiinvariant homogeneous regular function $F$ on $X'$. Then 
\[
D= \{F=0\}\cap X
\]
contains $Z$ but not $Y$. Any irreducible component of $D$ containing $Z$ satisfies the desired properties.
\end{proof}

\begin{proposition}\label{prop:locSigma1}
Let $X$ be a $P$-toroidal embedding of $P/H$ and $Y\subseteq X$ an $L$-orbit such that $\overline Y$ is $P$-stable.
\begin{enumerate}
\item\label{prop:locSigma1:toroidal} The closure $\overline Y$ is a $P$-toroidal embedding of an $L$-spherical homogeneous space $P/K$ for some subgroup $K\subseteq P$.
\item\label{prop:locSigma1:lattice} Restriction of rational functions from $X$ to $Y$ induces an identification between $\Xi_B(\overline Y)$ and $\Cone_{Y}^\perp(\subseteq \Xi_B(X))$, and correspondingly an identification between $N_B(\overline Y)$ and $N_B(X)/(\Span_\QQ\Cone_{Y})$.
\item\label{prop:locSigma1:fan} The set $V(Y)$ is the projection on $N_B(\overline Y)$ of the set $\bigcup_Z \Cone_Z$, where $Z$ varies in the set of $L$-orbits of $X$ such that $\overline Z$ is $P$-stable and $\Cone_Z\supseteq \Cone_Y$.
\end{enumerate}
\end{proposition}
\begin{proof}
Let us prove \ref{prop:locSigma1:toroidal}. The closure $\overline Y$ is a spherical $L$-variety equipped with an action of $P$ extending the one of $L$. It follows that $\overline Y$ contains an open ($L$-spherical) $P$-orbit $P/K$. It remains to check that $\overline Y$ is $P$-toroidal. Let $D$ be an element of $\Div(P/K)^B$ and suppose that $\overline D\subset \overline Y$ contains a $P$-orbit. By Lemma~\ref{lemma:Bstable} the closure $\overline D$ is contained in a $B$-stable prime divisor $E$ of $X$ not containing $\overline Y$. This implies that $\overline D$ is an irreducible component of $E\cap \overline Y$. Since $\overline D$ contains a $P$-orbit, then $E$ contains it too, so it is $P$-stable, therefore $\overline D$ is also $P$-stable.

Part \ref{prop:locSigma1:lattice} is a standard result for spherical $L$-varieties, see e.g.\ \cite[Proof of Theorem~7.3]{Kn91}.

To prove part \ref{prop:locSigma1:fan}, recall that for any $L$-orbit $Z\subseteq X$ the condition $Z\subseteq \overline Y$ is equivalent to $\Cone_Z\supseteq \Cone_Y$. It only remains to prove that the cone of $Z$ in the fan of $\overline Y$ is the projection of $\Cone_Z$ in $N_B(\overline Y)$, for any $L$-orbit $Z$ such that $Z\subseteq \overline Y$ and $\overline Z$ is $P$-stable. Thanks to Corollary~\ref{cor:Ltoroidaln}, $Z$ has in $X$ an $L$-stable neighborhood $X_0$ that is an $L$-toroidal embedding. Then $X_0$ contains $Y$, since $\overline Y\supseteq Z$. At this point the cone of $Z$ in the fan of $\overline Y$ is indeed as desired by the standard theory of toric varieties, thanks to \cite[Theorem~29.1]{Ti11}.
\end{proof}

We are ready to prove that $V(X)$ is a convex cone. The key argument comes from Corollary~\ref{cor:halfspace}, and the following three lemmas will enable us to apply it here.

\begin{lemma}\label{lemma:convex1}
Let $X$ be a complete $P$-toroidal embedding of $P/H$, and let $Y\subseteq X$ be an $L$-orbit whose closure is $P$-stable. If $\Cone_Y$ is contained in a cone of $\Fan_L(X)$ not contained in $V(X)$, then $P^u$ acts nontrivially on $\overline Y$.
\end{lemma}
\begin{proof}
By assumption $\Cone_Y$ is contained in a cone of the form $\Cone_Z$, where $Z$ is an $L$-orbit of $X$, such that $\Cone_Z$ is not contained in $V(X)$. Then $\overline Z$ is contained in $\overline Y$, it is $L$-stable but not $P$-stable. Hence $P^u$ doesn't act trivially on $\overline Y$.
\end{proof}

\begin{lemma}\label{lemma:convex2}
If $P^u$ acts non-trivially on a complete $P$-toroidal embedding $X$ of $P/H$ then $V(X)_{\RR} \smallsetminus\{0\}$ is path-connected, where $V(X)_\RR$ is the convex cone generated by $V(X)$ in $N_B(X)\otimes_\QQ\RR$.
\end{lemma}
\begin{proof}
Since $X$ is complete, the set $X^{P^u}$ is non-empty by Borel's fixed point theorem, and by \cite[Theorem~1]{BB73} it is connected.

It follows that any two closed $L$-orbits $Y,Y'$ of $X$ that are also $P$-orbits are ``connected'' by a sequence of $L$-orbits $Y=Y_1, Y_2,\ldots,Y_k = Y'$ such that for all $i\in\{1,\ldots,k-1\}$ the orbit $Y_i$ contains a point fixed by $P^u$ and we have either $\overline{Y_i}\supset Y_{i+1}$ or $Y_i\subset \overline{Y_{i+1}}$. For all $i$, the $L$-orbit $Y_i$ is then itself a $P$-orbit, thus has $P$-stable closure. Moreover $Y_i$ is not the open $L$-orbit of $X$, since $P^u$ doesn't act trivially on the whole $X$.

This translates to the existence, for any two maximal cones $\Cone, \Cone'$ in the union defining $V(X)$, of a sequence of cones $\Cone=\Cone_1, \Cone_2,\ldots,\Cone_k=\Cone'$ such that for all $i\in\{1,\ldots,k-1\}$ the cone $\Cone_i$ is contained in $V(X)$, and either is contained in $\Cone_{i+1}$ or contains it. They are also different from $\{0\}$, which is the cone of the open $L$-orbit of $X$, and the lemma follows.
\end{proof}

\begin{lemma}\label{lemma:convex3}
Let $V$ be a finite union of polyhedral convex cones of maximal dimension in a finite-dimensional real vector space, such that the intersection of any two of these cones is a face of both. If $V\smallsetminus\{0\}$ is path-connected and locally convex, and $V$ is not convex, then $V$ is not contained in any closed half-space.
\end{lemma}
\begin{proof}
Suppose that $V$ is contained in one of the two half-spaces defined by a linear subspace $M$ of codimension $1$. For dimension reasons $V$ is not contained in $M$, and we also have $\overline{V\smallsetminus M}=V$.

Let $\gamma\colon [0,1]\to V\smallsetminus\{0\}$ be a continuous path joining two points $\gamma(0)$ and $\gamma(1)$ both in $V\smallsetminus M$. Since $V\smallsetminus\{0\}$ is locally convex, it is elementary to show that $\gamma$ can be deformed if necessary in such a way that it connects the same two points $\gamma(0)$ and $\gamma(1)$, and $\gamma([0,1])$ is contained in $V\smallsetminus M$. Since $V\smallsetminus\{0\}$ is path-connected, we deduce that $V\smallsetminus M$ is path-connected too. It is also locally convex since it is the intersection of $V\smallsetminus \{0\}$ and an open half space, which are both locally convex.

At this point we claim that $V\smallsetminus M$ is convex; this will yield the desired contradiction because it implies that its closure $V$ is convex too.

Let $x,y\in V\smallsetminus M$, and fix $v$ in the open half-space defined by $M$ and containing $V\smallsetminus M$. We want to prove that the segment $s$ from $x$ to $y$ is contained in $V\smallsetminus M$; for this we may replace if necessary $x$ and $y$ with positive multiples, and assume they both lie on $(M+v)\cap V$. This intersection is closed and locally convex, and it is path-connected being a deformation retract of $V\smallsetminus M$. Thus $(M+v)\cap V$ is convex by the Tietze-Nakajima theorem (see e.g.\ \cite{T28}).

Then $s$ is contained in $V\smallsetminus M$, so $V\smallsetminus M$ is convex, and the proof is complete.
\end{proof}

\begin{theorem}\label{thm:V}
Let $X$ be a complete $P$-toroidal embedding of $P/H$. Then $V(X)$ is equal to the polyhedral convex cone generated by $\rho(\Div(X)^P)$, and has maximal dimension.
\end{theorem}
\begin{proof}
If $V(X)$ is a convex cone, then it has maximal dimension because $X$ contains a complete $P$-orbit, which is a complete $L$-orbit whose cone has maximal dimension and is contained in $V(X)$. At this point, thanks to Lemma~\ref{lemma:valcontained}, the theorem follows if we show that $V(X)$ is convex.

Denote by $\mathcal Y$ the set of $L$-orbits whose closure is $P$-stable, and let $Y\in \mathcal Y$. Consider the union of the projections on $N_B(\overline Y)$ of all cones $\Cone_Z$ such that $Z\in \mathcal Y$ and $\Cone_Z\supseteq \Cone_{Y}$. By Proposition~\ref{prop:locSigma1}, this union is equal to $V(\overline Y)$.

We proceed by contradiction, so we assume that $V(X)$ is not convex. In this case there exists $Y\in\mathcal Y$ such that $V(\overline Y)$ is not convex and $\Cone_Y$ has maximal dimension possible under these assumptions. We claim that $V(\overline Y)\smallsetminus\{0\}$ is locally convex, by maximality of the dimension of $\Cone_Y$.

Let us show this claim. For sake of contradiction, let $v\in N_B(\overline{Y})$ be a point of $V(\overline Y)\smallsetminus \{0\}$ where the latter is not locally convex. Then there exists $Y'\in\mathcal Y$ contained in the closure of $Y$ such that $v$ lies in the relative interior of the cone $\Cone^{\overline Y}_{Y'}$ of $Y'$ in the fan $\Fan_L(\overline Y)$. Then $V(\overline {Y'})$ is not convex, because $V(\overline Y)\smallsetminus \{0\}$ is not locally convex in $v$, and the cone $\Cone_{Y'}$ of $Y'$ in $\Fan_L(X)$ has bigger dimension than $\Cone_Y$: contradiction. This shows the claim that $V(\overline Y)\smallsetminus\{0\}$ is locally convex.

Now $V(\overline Y)$ is contained in the valuation cone $V_L(\overline Y)$ of the complete variety $\overline Y$. So any point $x$ not in $V(\overline Y)$ but in its convex hull is contained in a cone $\Cone^{\overline Y}_Z$ of $\Fan_L(\overline Y)$ for some $L$-orbit $Z$ in $\overline Y$ such that $Z\notin \mathcal Y$. Then $\Cone^{\overline Y}_Z$ is not contained in $V(\overline Y)$, and by Lemma~\ref{lemma:convex1} we have that $P^u$ doesn't act trivially on $\overline Y$. By Lemma~\ref{lemma:convex2}, the set $V(\overline Y)_\RR \smallsetminus\{0\}$ is path-connected, and by Lemma~\ref{lemma:convex3} the set $V(\overline Y)$ is not contained in any half-space. This contradicts Corollary~\ref{cor:halfspace}, therefore $V(X)$ is convex and the proof is complete.
\end{proof}

\begin{theorem}\label{thm:Vindep}
If $X$ and $Y$ are complete $P$-toroidal embeddings of $P/H$ then $V(X)=V(Y)$.
\end{theorem}
\begin{proof}
Thanks to Lemma~\ref{lemma:toroidalexist} there exists a complete $P$-toroidal embedding equipped with $P$-equivariant maps to $X$ and to $Y$. Hence we may suppose that there is a $P$-equivariant morphism $\varphi\colon X\to Y$. In this situation \cite[Theorem~5.1]{Kn91} implies that for any $L$-orbit $Z$ of $X$ the cone $\Cone_Z$ of $\Fan_L(X)$ is contained in the cone $\Cone_{\varphi(Z)}$ of $\Fan_L(Y)$.

Let $Z$ be a closed $L$-orbit of $X$. If $\Cone_Z\subseteq V(X)$ then $\overline Z=Z$ is $P$-stable by Proposition~\ref{prop:inVstable}, and $P^u$ acts trivially on $Z$. Then it acts trivially on $\varphi(Z)$ too, which is therefore also a $P$-orbit. It follows that $(\Cone_Z\subseteq)\Cone_{\varphi(Z)}\subseteq V(Y)$, and $V(X)\subseteq V(Y)$.

To show the reverse inclusion, we observe that $\varphi$ is proper and birational, so for all $D\in \Div(Y)^P$ there exists $D'\in \Div(X)^P$ such that $\varphi(D')=D$ and $\rho_X(D')=\rho_Y(D)$. This yields $\rho(\Div(X)^P)\supseteq \rho(\Div(Y)^P)$, and $V(X)\supseteq V(Y)$ by Theorem~\ref{thm:V}.
\end{proof}

Thanks to Theorems~\ref{thm:V} and~\ref{thm:Vindep}, we are now able to give the following.

\begin{definition}\label{def:sphroots}
Define $V(P/H)=V(X)$ where $X$ is any complete $P$-toroidal embedding of $P/H$. Also, define the set $\Sigma(P/H)$ of {\em spherical roots of the $P$-action on $P/H$} as the minimal set of primitive elements of $\Xi_B(P/H)$ such that
\[
V(P/H) = \{ v \in N_B(P/H) \;|\; \langle v, \sigma\rangle \leq 0 \;\forall\sigma\in \Sigma(P/H)\}.
\]
\end{definition}

\begin{remark}
We call the spherical roots of the $P$-action on $P/H$ also simply the {\em spherical roots of $P/H$.} To avoid confusion, we avoid doing the same for the spherical roots of the $L$-action on $P/H$.
\end{remark}

\begin{proposition}\label{prop:horospherical}
We have $V(P/H)=N_B(P/H)$ if and only if $\Sigma(P/H)=\varnothing$ if and only if $H$ contains a maximal unipotent subgroup of $P$.
\end{proposition}
\begin{proof}
The first equivalence is obvious. For the second, let $X$ be a complete $P$-toroidal embedding of $P/H$. If $H$ contains a maximal unipotent subgroup of $P$ then it contains $P^u$, therefore $P^u$ acts trivially on $P/H$. It follows that all $L$-orbits of $X$ are also $P$-orbits, so $X$ is also $L$-toroidal. As a consequence, we have $\Div(X)^P=\Div(X)^L$, and $V(X)$ is equal to the valuation cone $V_L(X)$ of $X$ as a spherical $L$-variety. We also have that $X$ is horospherical, because $P/H\cong (P/P^u)/(H/P^u)\cong L/K$ where $K$ contains a maximal unipotent subgroup of $L$. Then $V_L(X)=N_B(X)$.

Suppose now that $V(X)=N_B(X)$. By Corollary~\ref{cor:halfspace}, the unipotent radical $P^u$ acts trivially on $P/H$. This implies as above that all $L$-orbits of $X$ are also $P$-orbits, and $V(X)=V_L(X)$. Finally, the equality $V_L(X)=N_B(X)$ implies that $X$ is a horospherical $L$-variety, whence $H$ also contains a maximal unipotent subgroup $U$ of $L$. The product $UP^u$ is contained in $H$ and is a maximal unipotent subgroup of $P$.
\end{proof}

For any spherical variety $X$, the elements of $\Sigma_L(X)$ are linear combinations of simple roots of $L$ with non-negative rational coefficients (actually, integer in most cases). The analogue in our setting is provided by the following theorem.

\begin{theorem}\label{thm:NS}
Let $\sigma\in\Sigma(P/H)$. Then $\sigma\in\Sigma_L(P/H)$, or $-\sigma$ is a root of $P^u$.
\end{theorem}
\begin{proof}
Let $X$ be a complete $P$-toroidal embedding of $P/H$. Choose an $L$-orbit $Y$ such that $\Cone_Y$ is contained in $V(X)$, has codimension $1$ in $N_B(X)$ and is orthogonal to $\sigma$. Then $\overline Y$ is $P$-stable by Proposition~\ref{prop:inVstable}, and by Proposition~\ref{prop:locSigma1} we have $\Xi_B(\overline Y)=\ZZ\sigma$ and $V(\overline Y)$ is the half-line of $N_B(\overline Y)$ defined by $\langle -, \sigma\rangle\leq 0$.

Since $N_B(\overline Y)$ has dimension $1$, there are only two possibilities for $V_L(\overline Y)$, namely $V(\overline Y)$ and $N_B(\overline Y)$. In the first case $\Cone_Y$ is contained not only in a face of codimension $1$ of $V(X)$ but also of $V_L(X)$, therefore  $\sigma\in\Sigma_L(P/H)$. Hence we may suppose that $V_L(\overline Y)=N_B(\overline Y)$, i.e.\ $\overline Y$ is $L$-horospherical. Moreover $\overline Y$ is complete, hence there are cones in $\Fan_L(\overline Y)$ not in $V(\overline Y)$. This implies that not all closures of $L$-orbits of $\overline Y$ are $P$-stable, therefore $P^u$ doesn't act trivially on $\overline Y$.

Let $P/K$ be the open $P$-orbit of $\overline Y$, and consider the natural morphism $\pi\colon P/K\to P/KP^u$. Since $\Xi_B(\overline Y)=\ZZ\sigma$, the spherical $L$-variety $P/K$ has rank $1$, so its image $P/KP^u$ has rank $1$ or $0$. If $P/KP^u$ has rank $1$, then the map $\pi$ has finite fibers. But the fiber $\pi^{-1}(eKP^u)$ is the quotient $KP^u/K=P^u/(K\cap P^u)$, which is connected and not equal to a single point because $P^u$ doesn't act trivially on $\overline Y$: contradiction. Therefore $P/KP^u$ has rank $0$, i.e.\ it is a complete $L$-homogeneous space, and $KP^u$ is a parabolic subgroup of $P$. Set $Q=KP^u$.

Up to conjugating $Q$ in $P$ we may assume that it contains $B_-$, we denote by $M$ its Levi subgroup and by $B_M$ the Borel subgroup of $M$ as in Lemma~\ref{lemma:fiberoverflag}. We apply Lemma~\ref{lemma:fiberoverflag} to $K$ and $Q$, obtaining that the fiber $Z=\pi^{-1}(eQ)$ is a spherical $M$-variety, with $\Xi_B(P/K)=\Xi_{B_M}(Z)$.

Recall that $P^u$ as a variety is $L$-equivariantly isomorphic to its Lie algebra $\fn$ via the exponential map, which restricts to an $M$-equivariant isomorphism between the intersection $K\cap P^u$ and its own Lie algebra $\fii$. Then $Z\cong P^u/(K\cap P^u)\cong \fn/\fii$ is a spherical module for $M$. Moreover $Z$ has rank $1$, which implies that it is an irreducible $M$-module whose lattice $\Xi_{B_M}(Z)$ is generated by the highest weight $\gamma$ of the dual module $Z^*$. With this notation, the opposite $-\gamma$ is the lowest weight of $Z$, therefore a root of $P^u$, and it is also either $\sigma$ or $-\sigma$ because it generates $\Xi_B(P/K)$.

Consider the closure $\overline Z$ of $Z$ in $\overline Y$. It is connected and complete, so any function $f\in \CC[Z]^{(B\cap M)}_\gamma$ (i.e.\ a highest weight vector of the dual module $Z^*$), being non-constant, cannot be regular on $\overline Z$. Choose such $f$: we have already noticed the equality $\Xi_B(P/K)=\Xi_{B_M}(Z)$, i.e.\ the function $f$ extends to a rational function on $\overline Y$. Then $f$ has on $\overline Y$ a pole that intersects $\overline Z$ but not $Z$. Call it $D$: we have $\langle\rho(D),\gamma\rangle <0$. Now $D$ cannot intersect $P/K$: if it did, it would be either the inverse image of a $B$-stable prime divisor of $P/KP^u$, in which case it wouldn't intersect $\overline Z$, or $D$ would intersect $Z$, which is false. Hence $D$ is $P$-stable, which yields $\langle\rho(D),\sigma\rangle <0$ because $\sigma$ is the spherical root of $\overline Y$. We conclude that $\gamma=\sigma$, i.e.\ $-\sigma$ is a root of $P^u$.
\end{proof}

\section{Morphisms}\label{s:mor}

Given a spherical $L$-homogeneous space $X$, surjective equivariant morphisms with connected fibers $X\to Y$  are classified in \cite[Theorem~5.4]{Kn91} by means of {\em colored subspaces} of $N_B(X)$, i.e.\ couples $(\Cone, \Colemb)$ such that $\Colemb\subseteq \Delta_L(X)$, and $\Cone$ is a linear subspace of $N_B(X)$ generated as a convex cone by $\Colemb$ and by $\Cone\cap V_L(X)$. Precisely, with any morphism $\varphi\colon X\to Y$ one associates the colored subspace $(\Cone_\varphi, \Colemb_\varphi)$, where $\Cone_\varphi$ is the subspace $\Xi_B(Y)^\perp$ of $N_B(X)$, and $\Colemb_\varphi$ is the set of colors of $X$ mapped dominantly to $Y$. 

In this section we describe how this correspondence can be used to describe $P$-equivariant morphisms with domain $P/H$.

\begin{lemma}\label{lemma:pullback}
Let $K$ be a subgroup of $P$ containing $H$, and $\eta\colon P/H\to P/K$ the corresponding $P$-equivariant map.
Then there exist  a $P$-toroidal embedding $Z$ of $P/H$ and a set $\Div_0$ such that:
\begin{enumerate}
\item\label{lemma:pullback:1} $\Colemb_{\eta|_X}\subseteq\Div_0\subseteq \Colemb_{\eta|_X}\cup\Div(Z)^L$, where $X$ is the open $L$-orbit of $Z$,
\item\label{lemma:pullback:2} the convex cone generated by $\rho(\Div_0)$ is a vector subspace,
\item\label{lemma:pullback:3} $\Xi_B(P/K)$ is a subgroup of finite index of $\rho(\Div_0)^\perp (\subseteq \Xi_B(P/H))$.
\end{enumerate}
Moreover, the map $\eta$ has connected fibers if and only if $\Xi_B(P/K)$ is a saturated sublattice of $\Xi_B(P/H)$.
\end{lemma}
\begin{proof}
Using Lemma~\ref{lemma:extend} we extend the map $P/H\to P/K$ to a proper $P$-equivariant map $\psi\colon Z\to P/K$ where $Z$ is a $P$-toroidal embedding of $P/H$. Denote by
\[
Z\stackrel{\psi'}{\to} Z'\stackrel{\mu}{\to} P/K
\]
the Stein factorization of $\psi$.

Denote by $\Div_0\subseteq \Div(Z)^B$ the subsets of the elements mapped dominantly to $P/K$ by $\psi$. Then $ \Colemb_{\eta|_X}\subseteq \Div_0\subseteq \Colemb_{\eta|_X}\cup\Div(Z)^L$. Since $\mu$ is finite, the set $\Div_0$ is also the set of the elements in $\Div(Z)^B$ that are mapped dominantly to $Z'$.

For a $B$-semiinvariant rational function on $Z$, being a pull-back along $\psi'$ is equivalent to being constant on general fibers of $\psi'$, which is also equivalent to having no poles along the $B$-stable prime divisors of $Z$ mapped dominantly to $Z'$. This shows that $\Xi_B(Z')$ is the subset of the elements of $\Xi_B(Z)$ that are $\geq0$ on $\rho(\Div_0)$. Since $\Xi_B(Z')$ is also subgroup of $\Xi_B(Z)$, then all elements of $\rho(\Div_0)$ are actually $0$ on it, whence the convex cone generated by $\rho(\Div_0)$ is a vector subspace, and $\Xi_B(Z')$ is a saturated sublattice of $\Xi_B(Z)$.

The map $\mu$ is finite, therefore $\Xi_B(P/K)$ is a subgroup of $\Xi_B(Z')$ of finite index, and this finishes the proof of statements~(\ref{lemma:pullback:1}),~(\ref{lemma:pullback:2}), and~(\ref{lemma:pullback:3}). Also since $\mu$ is finite, the variety $Z'$ is a homogeneous space $P/K'$ where $H\subseteq K'\subseteq K$ and $K'/K$ is finite. This implies that $\Xi_B(P/K)=\Xi_B(P/K')$ if and only if $K=K'$, proving the last assertion of the lemma.
\end{proof}

\begin{definition}\label{def:Pequiv}
A colored subspace $(\Cone, \Colemb)$ of $N_B(P/H)$ is called {\em $P$-equivariant} if $\Cone$ is generated as a convex cone by $\rho(\Colemb)$, a subset of $\rho(\Div(P/H)^L)$, and finitely many elements of the form $\rho(v)$ for $v$ a $P$-invariant valuation of $\CC(P/H)$.
\end{definition}

\begin{remark}\label{rem:Pequiv}
Let $(\Cone, \Colemb)$ be a colored subspace of $N_B(P/H)$. If it is generated as a convex cone by $\rho(\Colemb)$, a subset of $\rho(\Div(P/H)^L)$, and a face of $V(P/H)$, then it is $P$-equivariant.
\end{remark}

We prepare for the proof of the main result of this section, Theorem~\ref{thm:morphisms}, with two lemmas essentially taken from \cite[Section~2]{Kn91}.

\begin{lemma}[see {\cite[Corollary~2.5]{Kn91}}]\label{lemma:val1}
Any $P$-invariant discrete valuation of $\CC(P/H)$ can be lifted to a $P$-invariant discrete valuation of $\CC(P)$.
\end{lemma}
\begin{proof}
The proofs of \cite[Lemma~2.4 and Corollary~2.5]{Kn91} hold verbatim with the non-reductive group $P$ in place of the reductive group $G$ of loc.cit., and yield the desired result.
\end{proof}

\begin{lemma}[see {\cite[Theorem~2.6 and Section~6]{Kn91}}]\label{lemma:val2}
Let $v$ be a $P$-invariant discrete valuation of $\CC(P)$, let $s$ be a positive integer, let $f_0\in \CC[P]$ and $g_1,\ldots,g_s\in \CC(P)$ be such that $f_i=f_0g_i\in\CC[P]$ for all $i\in\{1,\ldots,s\}$. Denote by $M_i$ the $P$-submodule of $\CC[P]$ generated by $f_i$, and denote by $M_1\cdots M_s$ the linear span of all $s$-fold products where for all $i$ the $i$-th factor is in $M_i$. Then for all $h\in M_1\cdots M_s$ we have
\[
v(h/f_0^s)\geq v(g_1)+\ldots+ v(g_s).	
\] 
\end{lemma}
\begin{proof}
Similarly to the proof of \cite[Theorem~2.6]{Kn91}, we observe that the valuation $v$ on $M_i$ is at least $v(f_i)$ by $P$-invariance, so it is at least $v(f_1)+\ldots+v(f_s)$ on $s$-fold products where the $i$-th factor is in $M_i$. This yields the inequality
\[
v(h)\geq v(f_1)+\ldots+ v(f_s)
\]
which is equivalent to the desired one.
\end{proof}

\begin{theorem}\label{thm:morphisms}
Let $X$ be the open $L$-orbit of $P/H$, let $Y$ be an $L$-homogeneous space, and $\varphi\colon X\to Y$ be an $L$-equivariant map with connected fibers. Then $(\Cone_\varphi,\Colemb_\varphi)$ is $P$-equivariant if and only if there exists a subgroup $K$ of $P$ containing $H$ such that $P/K$ is an $L$-embedding of $Y$ and the corresponding map $P/H\to P/K$ extends $\varphi$. If this is the case, a subgroup $K$ with these properties is unique.
\end{theorem}
\begin{proof}
Suppose that $\varphi\colon X\to Y$ extends to a $P$-equivariant map $P/H\to P/K$. Then $(\Cone_\varphi,\Colemb_\varphi)$ is $P$-equivariant thanks to Lemma~\ref{lemma:pullback}.

We prove the converse statement, so suppose that $(\Cone_\varphi, \Colemb_\varphi)$ is $P$-equivariant. First we construct a map from $P/H$ onto a $P$-homogeneous space $P/K$ associated with $\Cone_\varphi$ as in the proof of \cite[Theorem~5.4]{Kn91} (see \cite[proof of Theorem~3.3]{Kn91} for more details), then we check it extends $\varphi$.

The dual cone $\Lambda$ of $\Cone_\varphi$ in $\Xi_B(P/H)$ (i.e.\ the set of all elements that are non-negative on $\Cone_\varphi$) is actually a subgroup and is equal to $\Cone_\varphi^\perp$; by Gordan's Lemma $\Lambda$ is a finitely generated monoid. Let $\chi_1,\ldots,\chi_n$ be generators of $\Lambda$ and for all $i\in\{1,\ldots,n\}$ choose $g_i\in\CC(P/H)^{(B)}_{\chi_i}$. Let $D_0$ be the union of all elements of $\Div(P/H)^B$ not contained in $\Colemb_\varphi$ and not contained in $\Div(P/H)^L\cap\rho^{-1}(\Cone_\varphi)$.

For all $i$ the poles of $g_i$ are contained in $D_0$. The pull-back of $g_i$ in $\CC(P)$, also denoted $g_i$, is $H$-invariant and $B$-semiinvariant. Choose a non-zero function $f_0\in\CC[P]^{(B\times H)}$ that vanishes exactly on the inverse image of $D_0$ in $P$ (for its existence, see the proof of Lemma~\ref{lemma:toroidalexist}), with enough multiplicity so that $f_i=f_0g_i$ is a regular function on $P$.

The $P$-module $M$ generated by $f_0,f_1,\ldots,f_n$ in $\CC[P]$ defines naturally a $P$-equivariant map $\psi\colon P/H\to\PP(M^*)$. The image is a $P$-orbit, and we denote the stabilizer of $\psi(eH)$ by $K$. Notice that this orbit, isomorphic to $P/K$, is dense in $P\cdot X_0$ where $X_0 = \overline{\psi(P/H)}\setminus\{f_0=0\}$.

We claim now that
\begin{enumerate}
\item\label{proof:functions} $\Xi_B(P/K)=\Lambda$;
\item\label{proof:fibers} $\psi$ has connected fibers;
\item\label{proof:colors} $\Colemb_\varphi$ is the set of $L$-colors of $P/H$ mapped dominantly by $\psi$ to $P/K$.
\end{enumerate}
From these three properties we conclude that the $L$-equivariant map $\psi|_X$ has colored cone equal to $(\Cone_\varphi,\Colemb_\varphi)$, hence $\psi|_X=\varphi$.

Let us show (\ref{proof:functions}). It follows if we prove that $\mathcal M = \CC[X_0]^{(B)}$, where
\[
\mathcal M = \left\{f\in\CC(P/H)^{(B)}_\lambda\;\middle\vert\; \lambda\in\Lambda\right\}.
\]

The inclusion $\mathcal M \subseteq \CC[X_0]^{(B)}$ stems from the fact that the functions $g_1,\ldots,g_n$ correspond to regular functions on the principal open subset $\PP(M^*)_{f_0}$, so they restrict to regular functions on $X_0$. To show the other inclusion, we show that $\langle c, \chi\rangle \geq 0$ for any $c$ belonging to a given set of generators of $\Cone_\varphi$ as in Definition~\ref{def:Pequiv} and any $B$-character $\chi\neq0$ such that $\CC[X_0]^{(B)}_\chi$ is non-empty. Denote by $f$ an element of $\CC[X_0]^{(B)}_\chi$.

If $c=\rho(D)$ where $D\in \Colemb_\varphi\cup\Div(P/H)^{L}$ then $\langle c, \chi\rangle$, which is equal to $\Ord_D(f)$, is non-negative because the poles of $f$ as a rational function on $P/H$ are contained in $D_0$. Otherwise $c$ is of the form $c = \rho(v)$ where $v$ is a $P$-invariant discrete valuation of $\CC(P/H)$, so $\langle c,\chi\rangle = v(f)$.

We lift $v$ to a $P$-invariant valuation of $\CC(P)$ thanks to Lemma~\ref{lemma:val1}. For all $i\in\{0,\ldots,n\}$ choose $p_{i,1},\ldots,p_{i,m_i}\in P$ in such a way that the elements
\[
p_{0,1}f_0,p_{0,2}f_0,\ldots,p_{n,m_n-1}f_n,p_{n,m_n}f_n
\]
form a basis of $M$. We observe now that $X_0$ is a closed subvariety of the principal open subset $\PP(M^*)_{f_0}$. Then $f$ is the restriction to $X_0$ of a regular function on $\PP(M^*)_{f_0}$, and any such function is a polynomial in $\frac{p_{0,1}f_0}{f_0},\ldots,\frac{p_{n,m_n}f_n}{f_0}$. In other words $f$ is a linear combination of products of the form
\begin{equation}\label{eq:monomial}
\left(\frac{p_{0,1}f_0}{f_0}\right)^{a_{0,1}}\cdots\left(\frac{p_{n,m_n}f_n}{f_0}\right)^{a_{n,m_n}}
\end{equation}
for non-negative integers $a_{0,1},\ldots,a_{n,m_n}$.

Since $v$ is a discrete valuation and we want to prove that $v(f)\geq 0$, it is enough to prove this inequality assuming that $f$ is given by a single product as in (\ref{eq:monomial}). Moreover, for all $j$ we have $v(p_{0,j}f_0/f_0)=0$ by $P$-invariance of $v$, so we can also assume that
\[
f=\left(\frac{p_{1,1}f_1}{f_0}\right)^{a_{1,1}}\cdots\left(\frac{p_{n,m_n}f_n}{f_0}\right)^{a_{n,m_n}}.
\]
Denote $s=a_{1,1}+\ldots+a_{n,m_n}$. Up to reindexing $\chi_1,\ldots,\chi_n$ and introducing repetitions if necessary, we may assume that $s\leq n$ and that $ff_0^s$ is in $M_1\cdots M_s$ where $M_j$ is the $P$-submodule of $\CC[P]$ generated by $f_j$ for all $j\in\{1,\ldots,s\}$. Lemma~\ref{lemma:val2} applied to $v$ and $h = ff_0^s$ yields
\[
v(f) \geq v(g_1)+\ldots+v(g_s) = \langle c,\chi_1\rangle+\ldots\langle c,\chi_s\rangle \geq 0,
\]
concluding the proof of part (\ref{proof:functions}).

Part (\ref{proof:fibers}) stems from part (\ref{proof:functions}) and Lemma~\ref{lemma:pullback}, let us show part (\ref{proof:colors}). Any $L$-color of $P/H$ not in $\Colemb_\varphi$ is mapped to the proper subset $\{f_0=0\}$ of $P/K$, so it is not mapped dominantly to $P/K$. Suppose now that an $L$-color $D\in \Colemb_\varphi$ is not mapped dominantly to $P/K$. Then $f_0$ is non-zero on the inverse image of $D$ in $P$, so $\psi(D)$ intersects $X_0$ in a $B$-stable prime divisor $E$. By part (\ref{proof:functions}), and since $\rho(\Colemb_\varphi)\subseteq\Cone_\varphi$, no non-zero $B$-semiinvariant rational function of $X_0$ vanishes on $E$. This is absurd since $X_0$ is affine, and part (\ref{proof:colors}) is shown.

It remains to show uniqueness of $K$. Suppose that $\varphi$ has extensions $\psi_i\colon P/H\to P/K_i$ given by subgroups $K_i\supseteq H$ for $i\in \{1,2\}$. We may suppose that the $L$-orbit of $eH$ in $P/H$ is dense. Denote by $I$ the stabilizer of $eH$ in $L$, and by $J$ the stabilizer in $L$ of $\varphi(eH)$. Then $I=H\cap L$ and $J = K_i\cap L$ for any $i$. The two maps $\psi_i$ both factor as the composition of the natural maps $P/H \to P/(K_1\cap K_2)\to P/K_i$. The restriction to the open $L$-orbits, i.e.\ the map $\varphi$, is then the composition
\[
L/I \to L/J' \to L/J,
\]
where $J'=(K_1\cap K_2)\cap L$. From this equality follows $J'\supseteq J$, which yields $J'=J$. In other words the maps $P/(K_1\cap K_2)\to P/K_i$ are bijective if restricted to the respective open $L$-orbits. Since $P$ is connected, this is only possible if $K_1=K_2$.
\end{proof}

We may restate more concisely the above theorem in the following way.

\begin{corollary}
The bijection $\varphi\mapsto(\Cone_\varphi,\Colemb_\varphi)$ of \cite[Theorem~5.4]{Kn91} induces a bijection between the set of surjective $P$-equivariant maps with domain $P/H$ and connected fibers and the set of $P$-equivariant colored subspaces of $N_B(P/H)$.
\end{corollary}

Let $K\supseteq H$ be a subgroup of $P$ and $\varphi\colon P/H\to P/K$ the induced map, denote by $\varphi^*\colon \Xi_B(P/K)\hookrightarrow \Xi_B(P/H)$ the homomorphism induced by the pull-back of functions, and by $\varphi_*\colon N_B(P/H) \twoheadrightarrow N_B(P/K)$ the dual homomorphism. If $\varphi$ has connected fibers then $\varphi_*$ can be identified with the quotient map $N_B(P/H)\twoheadrightarrow N_B(P/H)/\Cone_\varphi$.

\begin{proposition}\label{prop:imageV}
Let $\varphi\colon P/H\to P/K$ be as above. Then $\varphi_*(V(P/H))=V(P/K)$.
\end{proposition}
\begin{proof}
Let $Y$ be a $P$-toroidal complete embedding of $P/K$. Thanks to Lemma~\ref{lemma:extend}, we can find a complete $P$-toroidal embedding $X$ of $P/H$ such that the map $\varphi$ extends to a $P$-equivariant map $\varphi\colon X\to Y$. Let $Z$ be a closed $P$-orbit of $X$, so that $\Cone_Z$ is a maximal cone contained in $V(X)$. Then $\varphi(Z)$ is a complete $P$-orbit of $Y$, and $\Cone_{\varphi(Z)}$ is contained in $V(Y)$. On the other hand $\Cone_{\varphi(Z)}\supseteq \varphi^*(\Cone_Z)$, which implies $\varphi^*(V(X))\subseteq V(Y)$.

Let now $Z$ be a $P$-stable prime divisor of $Y$. The inverse image $\varphi^{-1}(Z)$ contains a $P$-stable prime divisor $Z'$ of $X$ mapped dominantly to $Z$. Then $\QQ_{\geq 0}\varphi^*(\rho_X(Z'))=\QQ_{\geq 0}\rho_Y(Z)$, which shows $\varphi^*(V(X))\supseteq V(Y)$.
\end{proof}

We expect that it is possible to develop a theory of $P$-equivariant embeddings of $P/H$ based on the same ideas of \cite{Kn91}, as already suggested by the proof of Theorem~\ref{thm:morphisms}. In particular, one can show as in \cite[Lemma~6.1]{Kn91} that $V(P/H)$ is the image in $N_B(P/H)$ of the set of $P$-invariant discrete valuations of $\CC(P/H)$, in analogy with the fact that for a complete toroidal spherical $L$-variety $X$ the convex cone generated by $\rho(\Div(X)^L)$ is the valuation cone $V_L(X)$. This goes beyond the scope of the present work.

\section{Localization via one-parameter subgroups}\label{s:loclambda}

In this section we study {\em localizations} of a $P$-toroidal embedding $X\supseteq P/H$, i.e.\ certain subvarieties of $X$ that are invariant (and spherical) under the action of Levi subgroups of parabolic subgroups of $L$.

Our goal is essentially to adapt some of the results of \cite[Section~4]{Kn14}, taking into account the additional action of $P^u$ on $X$. Several proofs here follow the guidelines of those in \cite{Kn14} and \cite{Lu97} with only minimal modifications. We start recalling some well-known results on the action of one-parameter subgroups on normal, complete varieties, and we refer to loc.cit.\ for more details.

Let $X$ be a complete normal $\Gm$-variety. Then the set of fixed points $X^{\Gm}$ is closed, thus complete, and there is a unique connected component $Y$ of $X^{\Gm}$ such that
\[
\lim_{t\to0} t\cdot x \in Y
\]
for a general point $x\in X$. Denote by $\pi_\lambda(x)$ the above limit and by $Z$ the subset of $X$ such that $\pi_\lambda(x)\in Y$. Then $Z$ is open in $X$, the map $\pi\colon Z\to Y$ is a categorical quotient by $\Gm$ and its general fibers are irreducible. In particular $Y$ is irreducible and normal, and is characterized by the property that, for any $x\in X$, we have $\lim_{t\to \infty} t\cdot x\in Y$ if and only if $x\in Y$.

If $\lambda\colon \Gm\to T$ is a one-parameter subgroup of $T$, then one denotes by $L^\lambda$ the centralizer of $\lambda(\Gm)$ in $L$ and by $L_\lambda$ the parabolic subgroup of $L$ of those $l\in L$ such that $\lim_{t\to0} \lambda(t)l\lambda(t)^{-1}$ exists. Then $L^\lambda$ is a Levi subgroup of $L_\lambda$, and the map $l\mapsto \pi_L(l) = \lambda(t)l\lambda(t)^{-1}$ is the quotient map $L_\lambda\to L_\lambda/L_\lambda^u\cong L^\lambda$.

\begin{definition}\label{def:loclambda}
Let $X$ be a complete normal $T$-variety and $\lambda$ a one-parameter subgroup of $T$. We define $X_\lambda=Z$ and $X^\lambda=Y$, where $Z$ and $Y$ are defined as above considering $X$ as a $\Gm$-variety via the action of $\lambda(\Gm)\subseteq T$. The variety $X^\lambda$ is called the {\em localization of $X$ with respect to $\lambda$}.
\end{definition}

\begin{definition}
A one-parameter subgroup $\lambda$ of $T$ is {\em $P$-dominant} if $\langle\lambda,\alpha\rangle \leq0$ for all roots $\alpha$ of $P^u$ and of $(B_-)^u$. If this is the case we denote by $P^\lambda$ the centralizer in $P$ of $\lambda(\Gm)$, and put
\[
P_\lambda = \{ p\in P \;|\; \lim_{t\to0} \lambda(t)p\lambda(t)^{-1} \text{ exists} \}.
\]
For all $p\in P_\lambda$ we also set $\pi_P(p)=\lim_{t\to0} \lambda(t)p\lambda(t)^{-1}$.
\end{definition}

\begin{lemma}\label{lemma:locP}
Let $\lambda$ be a $P$-dominant one-parameter subgroup of $T$. The groups $P_\lambda$ and $P^\lambda$ have the same Levi subgroup $L^\lambda$; we have $(P^\lambda)^u = P^\lambda\cap P^u$ and $(P_\lambda)^u = (P^\lambda)^u(L_\lambda)^u$. The intersection $B^\lambda=L^\lambda\cap B$ is a Borel subgroup of $L^\lambda$ and is equal to $\pi_P(B)$; the intersection $B^\lambda_-=L^\lambda\cap B_-$ is the Borel subgroup of $L^\lambda$ opposite to $B^\lambda$ with respect to $T$. Finally, a root $\alpha$ of $P^u$ (resp.\ $B_-$) is not a root of $(P^\lambda)^u$ (resp.\ $B^\lambda_-$) if and only if $\langle\lambda,\alpha\rangle <0$.
\end{lemma}
\begin{proof}
We deduce all statements from the fact that, for any $p\in P$, the limit
\[
\lim_{t\to0} \lambda(t)p\lambda(t)^{-1}
\]
exists in $P$ if and only if $\lim_{t\to0} \lambda(t)l\lambda(t)^{-1}$ exists in $L$ and $\lim_{t\to0} \lambda(t)u\lambda(t)^{-1}$ exists in $P^u$, where $p=ul$ with $u\in P^u$ and $l\in L$. The ``if'' part of this claim is obvious. The ``only if'' part holds because $P\cong P^u\times L$ as varieties equipped with the action of $\lambda(\Gm)$ by conjugation.

At this point we have that $p\in P_\lambda$ if and only if $l\in L_\lambda$ and $u\in (P^u)_\lambda$, so $P_\lambda = L_\lambda \ltimes (P^u)_\lambda = (L^\lambda \ltimes (L_\lambda)^u)\ltimes (P^u)_\lambda$. This shows $(P_\lambda)^u = (P^u)_\lambda(L_\lambda)^u$, which is similar, but not yet equal, to one of the claims of the lemma. In the same way one checks that $p$ is $\lambda(\Gm)$-stable if and only if $u$ and $l$ are, so $P^\lambda = L^\lambda \ltimes (P^u)^\lambda$. This proves that $L^\lambda$ is a Levi subgroup of $P_\lambda$ and of $P^\lambda$, and also that $(P^\lambda)^u = (P^u)^\lambda$. A consequence of the last equality is the desired formula $(P^\lambda)^u = P^\lambda\cap P^u$.

We observe that $P^u$ is an affine space where $\lambda(\Gm)$ acts linearly via the characters $\langle \lambda,\alpha\rangle$ for $\alpha$ varying in the roots of $P^u$. Since $\langle\lambda,\alpha\rangle \leq 0$ for all such $\alpha$ by hypothesis, the limit $\lim_{t\to0} \lambda(t)u\lambda(t)^{-1}$ exists if and only if $u$ is fixed by $\lambda(\Gm)$. This shows $(P^u)_\lambda = (P^u)^\lambda$, and with it the desired equality $(P_\lambda)^u = (P^\lambda)^u(L_\lambda)^u$. The other claims of the lemma are standard.
\end{proof}

\begin{definition}
Let $\lambda$ be a one-parameter subgroup of $T$. We denote by $\lambda^r\in N_B(P/H)$ the corresponding element induced by restricting to $\Xi_B(P/H)$ the natural pairing of $\lambda$ with $\Chars(B)$.
\end{definition}

\begin{lemma}\label{lemma:lambdaPdominant}
If a one-parameter subgroup $\lambda$ of $T$ is $P$-dominant then $-\lambda^r\in V(P/H)$.
\end{lemma}
\begin{proof}
This stems from Theorem~\ref{thm:NS}, and the fact that all elements of $\Sigma_L(P/H)$ are linear combinations of simple roots of $L$ with non-negative coefficients.
\end{proof}

\begin{lemma}\label{lemma:loclambda}
Let $X$ be a $P$-toroidal embedding of $P/H$ and $\lambda$ a $P$-dominant one-parameter subgroup of $T$. Then $X_\lambda$ is $P_\lambda$-stable and $X^\lambda$ is $P^\lambda$-stable and spherical as an $L^\lambda$-variety. Moreover, the map $\pi_\lambda\colon X_\lambda\to X^\lambda$ is $P_\lambda$-equivariant, if we let $P_\lambda$ act on $X^\lambda$ via the quotient $\pi_P\colon P_\lambda \to P^\lambda$.
\end{lemma}
\begin{proof}
The centralizer $P^\lambda$ of $\lambda(\Gm)$ acts on $X^{\lambda(\Gm)}$; since $P^\lambda$ is connected it preserves the connected components hence also $X^\lambda$.

The other statements follow, thanks to the equality
\[
\lim_{t\to0} \left(\lambda(t) \cdot px\right) = \lim_{t\to0} \left((\lambda(t)p\lambda(t)^{-1}) \cdot \lambda(t)x\right) = \pi_P(p)\cdot \pi_\lambda(x)
\]
for all $x\in X_\lambda$, $p\in P_\lambda$, except for the fact that $X^\lambda$ is spherical. To prove it, we observe that the open subset $X_\lambda$ of $X$ intersects the open $B$-orbit of $X$; for any $x$ in this intersection the orbit $B^\lambda\pi_\lambda(x)$ is open in $X^\lambda$.
\end{proof}

\begin{lemma}\label{lemma:negfixes}
Let $X$ and $\lambda$ be as in Lemma~\ref{lemma:loclambda}, and let $\alpha$ be a root of $P$ with $\langle\lambda, \alpha\rangle <0$. Then $U_\alpha$ acts trivially on $X^\lambda$.
\end{lemma}
\begin{proof}
This is shown as in \cite[Proof of Proposition~1.3]{Lu97}. For all $x\in X^\lambda$ we have $x = \lambda(t)x$, and for all $u\in U_\alpha$ we have $\lim_{t\to 0} \lambda(t)^{-1}u\lambda(t) = e$. Then
\[
\lim_{t\to\infty} \lambda(t) ux = \lim_{t\to 0} \lambda(t)^{-1} ux = \lim_{t\to 0} \left(\lambda(t)^{-1}u\lambda(t)\right) \cdot x = x \in X^\lambda.
\]
This shows that $ux\in X^\lambda$, whence $\lambda(t)ux = ux$ for all $t$, so $ux=x$.
\end{proof}

Thanks to Lemma~\ref{lemma:loclambda}, the localization $X^\lambda$ has an open $B^\lambda$-orbit, therefore also an open $P^\lambda$-orbit.

\begin{proposition}\label{prop:loclambda}
Let $X$, $\lambda$ be as in Lemma~\ref{lemma:loclambda} and $P^\lambda/K$ be the open $P^\lambda$-orbit of $X^\lambda$. For all $D\in \Div(P^\lambda/K)^{B^\lambda}$ there is a unique $D^*\in \Div(P/H)^B$ such that $\pi_\lambda(D^*\cap X_\lambda)=D$. The map $D\mapsto D^*$ is a bijection between $\Div(P^\lambda/K)^{B^\lambda}$ and the set of those elements of $\Div(P/H)^B$ that are not stable under $P^\lambda$. Moreover, for any root $\alpha$ of $P^\lambda$, an element $D\in\Div(P^\lambda/K)^{B^\lambda}$ is not stable under $U_\alpha\subseteq P^\lambda$ if and only if the same holds for $D^*$.
\end{proposition}
\begin{proof}
We have that $E=\pi_\lambda^{-1}(D)$ is irreducible, because $\pi_\lambda$ is $P^\lambda$-equivariant, its general fibers are irreducible, and $P^\lambda/K$ is $P^\lambda$-homogeneous and dense in $X^\lambda$. Moreover $E$ cannot intersect the open $B$-orbit of $X$, otherwise $D$ would intersect the open $B^\lambda$-orbit of $X^\lambda$. Hence the closure $\overline E$ of $E$ in $X$ is an element of $\Div(X)^B$. Since $D$ is not $P^\lambda$-stable then $\overline E$ is not $P^\lambda$-stable. This also implies that $\overline E$ is not $P$-stable, therefore the intersection $D^*=\overline E\cap P/H$ is non-empty and obviously an element of $\Div(P/H)^B$.

The map $D\mapsto D^*$ is injective, so it remains only to prove that its image is as claimed. Let $F\in \Div(P/H)^B$ be not $P^\lambda$-stable. Then $F$ intersects $X_\lambda$, otherwise it would be $P_\lambda$-stable. But $F$ is anyway $\lambda(\Gm)$-stable because $\lambda(\Gm)\subseteq B$, therefore $D_0=\pi_\lambda(F\cap X_\lambda)$ is not $P^\lambda$-stable. It follows that $D=D_0\cap P^\lambda/K$ is non-empty and a $B^\lambda$-stable prime divisor of $P^\lambda/K$ satisfying $D^*=F$.

The last assertion stems from the fact that $\pi_\lambda\colon X_\lambda\to X^\lambda$ is in particular $P^\lambda$-equivariant.
\end{proof}

\begin{lemma}\label{lemma:loclambdaclosed}
Let $X$ and $\lambda$ be as above, and let $Z$ be a closed $P^\lambda$-orbit of $X^\lambda$. Then $Z$ contains a point fixed by $P^uB_-$, and $PZ$ is a closed $P$-orbit of $X$ with $PZ\cap X^\lambda = Z$.
\end{lemma}
\begin{proof}
Since $Z$ is a complete $P^\lambda$-orbit then it contains a point $z$ fixed by $(P^\lambda)^uB^\lambda_-$ by Borel's fixed point theorem. By Lemma~\ref{lemma:locP} and Lemma~\ref{lemma:negfixes} we have that $z$ is fixed also by $P^u$ and by $B_-$, whence $Pz=PZ$ is a closed $P$-orbit of $X$. Finally $z\in (PZ)^\lambda$, hence $PZ\cap X^\lambda = (Pz)^\lambda=P^\lambda z =Z$.
\end{proof}

\begin{definition}
Let $X$ and $\lambda$ be as above. We denote by $\Cone(\lambda)$ the cone of $\Fan_L(X)$ such that the relative interior of $\Cone(\lambda)$ contains $-\lambda^r$.
\end{definition}

\begin{theorem}\label{thm:loclambda}
Under the assumptions and notations of Proposition~\ref{prop:loclambda}:
\begin{enumerate}
\item\label{thm:loclambda:toroidal} the variety $X^\lambda$ is a $P^\lambda$-toroidal complete embedding of $P^\lambda/K$;
\item\label{thm:loclambda:Lequiv} the variety $X^\lambda$ is $L$-equivariantly birational to the localization with respect to $\lambda$ of an $L$-toroidal complete embedding of the open $L$-orbit of $P/H$;
\item\label{thm:loclambda:comb} the lattice $\Xi_{B^\lambda}(X^\lambda)$ is equal to $\Cone(\lambda)^\perp(\subseteq \Xi_B(X))$, we have $\Sigma(X^\lambda) = \Sigma(X)\cap \Cone(\lambda)^\perp=\Sigma(X)\cap \lambda^\perp$, and for all $D\in\Div(P^\lambda/K)^{B^\lambda}$ the element $\rho_{X^\lambda}(D)$ is the restriction of $\rho_X(D^*)$ on $\Xi_{B^\lambda}(P^\lambda/K)$.
\end{enumerate}
\end{theorem}
\begin{proof}
We have already observed that $X^\lambda$ is normal and complete, with open $P^\lambda$-orbit $P^\lambda/K$. Let us show it is $P^\lambda$-toroidal, so suppose that the closure in $X^\lambda$ of some $D\in\Div(P^\lambda/K)^{B^\lambda}$ contains a $P^\lambda$-orbit $Z$. We may suppose that $Z$ is closed, and by Lemma~\ref{lemma:loclambdaclosed} we have that $Z$ contains a point fixed by $P^uB_-$, so $PZ$, which is a closed $P$-orbit, is also equal to $\overline{BZ}$. It follows that the closure of $D^*$ in $X$ contains the $P$-orbit $PZ$, contradicting the fact that $X$ is toroidal. This shows part (\ref{thm:loclambda:toroidal}).

In view of showing parts~(\ref{thm:loclambda:Lequiv}) and~(\ref{thm:loclambda:comb}), we observe that $\lambda$ is dominant as a weight of the maximal torus $T$ of $L$. We proceed similarly as the proof of Corollary~\ref{cor:Ltoroidaln}, so consider the fan of colored cones
\[
\Fan = \left\{ (\Cone\cap V_L(X),\varnothing) \right\}
\]
where $\Cone$ varies in the set of cones of $\Fan_L(X)$. It is the fan of an $L$-toroidal embedding $\widetilde X$ of the open $L$-orbit $X_0$ of $X$, such that the identity on $X_0$ extends to an $L$-equivariant map $\widetilde X\to X$. Moreover, if a cone of $\Fan_L(X)$ is contained in $V_L(X)$ then it is also a cone of $\Fan=\Fan_L(\widetilde X)$.

By Lemma~\ref{lemma:lambdaPdominant} the cone $\Cone(\lambda)$ is contained in $V(X)$, therefore it is contained in $V_L(X)$. Thanks to the observation above, the cone $\Cone(\lambda)$ is also a cone of $\Fan_L(\widetilde X)$. Let $Y$ (resp.\ $\widetilde Y$) be the $L$-orbit of $X$ (resp.\ $\widetilde X$) with cone $\Cone(\lambda)$. We have $\Cone_Y=\Cone_{\widetilde Y}$, and by Proposition~\ref{prop:inVstable} also $\Colemb_{Y}=\Colemb_{\widetilde Y}(=\varnothing)$. Then by \cite[Sections~4 and~5]{Kn91} we have that $\widetilde Y$ is mapped to $Y$, and there exist $L$-equivariant neighborhoods $\widetilde X_{\widetilde Y}$ of $\widetilde Y$ in $\widetilde X$ and $X_Y$ of $Y$ in $X$ such that $\widetilde X_{\widetilde Y}$ is mapped $L$-isomorphically onto $X_Y$. It follows that the map $\widetilde Y\to Y$ is an isomorphism.

Consider now the localization $\widetilde X^\lambda$ of $\widetilde X$ with respect to $\lambda$. Thanks to \cite[Proposition~4.8]{Kn14} we have that an open subset of $\widetilde X^\lambda$ is contained in $\widetilde Y$. Therefore $\widetilde X^\lambda$ intersects the open set $\widetilde X_{\widetilde Y}$. Since $\widetilde Y\to Y$ is an isomorphism, the set $\widetilde X^\lambda$ is mapped birationally onto $X^\lambda$. This shows part (\ref{thm:loclambda:Lequiv}).

The equality $\Xi_{B^\lambda}(X^\lambda)=\Cone(\lambda)^\perp$ is now a consequence of \cite[Theorem~4.6]{Kn14}. It also follows that $N_{B^\lambda}(X^\lambda)$ is identified with the quotient of $N_B(X)$ by its subspace $\Span_\QQ\Cone(\lambda)$. The assertion on $\rho_{X^\lambda}(D)$ stems from the fact that $D^*=\overline{\pi_\lambda^{-1}(D)}\cap P/H$.

To show that $\Sigma(X^\lambda)=\Sigma(X)\cap \Cone(\lambda)^\perp$, it is enough to show that $V(X^\lambda)\subseteq N_{B^\lambda}(X^\lambda)$ is the union of the projections of the maximal cones of $\Fan_L(X)$ contained in $V(X)$ and containing $\Cone(\lambda)$.

A maximal cone of $\Fan_L(X)$ contained in $V(X)$ is of the form $\Cone_Z$ for $Z$ a closed $P$-orbit of $X$. Then $Z$ is also a closed $L$-orbit, and, as above, we have that $Z$ has an $L$-stable neighborhood $X_Z$ equivariantly isomorphic to an $L$-stable neighborhood $\widetilde X_{\widetilde Z}$ of $\widetilde Z$, where $\widetilde Z$ is a closed $L$-orbit of $\widetilde X$ mapped isomorphically onto $Z$.

If $\Cone_Y\supseteq \Cone(\lambda)$ then $\widetilde X^\lambda\cap  \widetilde Z\neq \varnothing$, and $\widetilde X_{\widetilde Z}$ (resp.\ $X_Z$) contains the open set $\widetilde X_{\widetilde Y}$ (resp.\ $X_Y$) considered before. This implies that $\widetilde X^\lambda$ and $X^\lambda$ are isomorphic in a neighborhood of $\widetilde X^\lambda\cap \widetilde Z = X^\lambda\cap Z$. Thanks to \cite[Theorem~4.6]{Kn14} the projection of $\Cone_Z$ in $N_{B^\lambda}(X^\lambda)$ is the maximal cone $\Cone_{X^\lambda\cap Z}$ associated with $X^\lambda\cap Z$, which is a closed $L^\lambda$-orbit of $X^\lambda$. Since $P^u$ acts trivially on $Z$, we have that $(P^\lambda)^u$ acts trivially on $X^\lambda\cap Z$. Therefore the latter is a closed $P^\lambda$-orbit of $X^\lambda$ and $\Cone_{X^\lambda\cap Z}\subseteq V(X^\lambda)$. Lemma~\ref{lemma:loclambdaclosed} implies that all closed $P^\lambda$-orbits of $X^\lambda$ are obtained in this way, and the proof of $\Sigma(X^\lambda)=\Sigma(X)\cap \Cone(\lambda)^\perp$ is complete.

It remains to show that $\Sigma(X)\cap \Cone(\lambda)^\perp$ is also equal to $\Sigma(X)\cap \lambda^\perp$, and this follows from the fact that any $\sigma\in\Sigma(X)$ is $\leq0$ on $V(X)$. Then $\sigma\in \Cone(\lambda)^\perp$ if and only if $\sigma\in \lambda^\perp$.
\end{proof}

\begin{corollary}\label{cor:Sigma_L}
Recall that $S_P$ denotes the set of simple roots of $L$. If there exists a $P$-dominant one-parameter subgroup of $T$ vanishing on $S_P$ but not vanishing on any root of $P^u$, then $\Sigma_L(P/H)= \Sigma(P/H)\cap \Span_\QQ S_P$.
\end{corollary}
\begin{proof}
Let $\lambda$ be such a $P$-dominant one-parameter subgroup of $T$, and let $X$ be a $P$-toroidal complete embedding of $P/H$. Then $\Sigma(X^\lambda)= \Sigma(X)\cap\lambda^\perp$ by Theorem~\ref{thm:loclambda}. By Theorem~\ref{thm:NS}, an element of $\Sigma(P/H)$ is either in $\Span_\QQ S_P$, in which case it is in $\lambda^\perp$, or it is a root of $P^u$ up to sign, in which case it is not in $\lambda^\perp$. This yields $\Sigma(P/H)\cap\lambda^\perp=\Sigma(P/H)\cap \Span_\QQ S_P$. On the other hand $P^\lambda=L^\lambda=L$ and $B^\lambda=B$, therefore $\Sigma(X^\lambda)=\Sigma_L(X^\lambda)$.

In Theorem~\ref{thm:loclambda} we have shown that $X^\lambda$ is $L$-equivariantly birational to the localization with respect to $\lambda$ of an $L$-toroidal complete embedding of the open $L$-orbit of $P/H$. By \cite[Corollary~4.7]{Kn14} we conclude $\Sigma_L(X^\lambda)=\Sigma_L(P/H)\cap \lambda^\perp = \Sigma_L(P/H)$.
\end{proof}

\section{Examples}\label{s:examplesfinite}

\begin{example}\label{ex:SL2modT}
Consider the case $L=T$ and $\dim P^u=1$, i.e.\ $P$ is a semi-direct product $T\ltimes \CC$. Denote by $-\alpha$ the root of $P^u$, i.e.\ the character of $T$ acting on $\CC$. Notice that $\alpha$ extends to a character of $P$, trivial on $P^u$, and that $H^u\subseteq P^u$ since $L$ is a torus. We may also assume that $H=M\ltimes H^u$, where $M\subseteq T$. Here the quotient $P/H= T\times^M P^u/H^u$ is spherical under the action of $L(=T)$ if and only if either $H^u=P^u$, in which case $P/H$ is the torus $T/M$, or $M^\circ$ acts non-trivially on $P^u/H^u$. The last condition amounts to requiring that no multiple of $\alpha$ is trivial on $M$. If $H^u=P^u$ then $\Div(P/H)^B=\varnothing=\Sigma(P/H)$, and $\Xi_B(P/H) = \Chars(T/M)$.

Let us discuss the case where $M^\circ$ acts nontrivially on $P^u/H^u$, which implies that $H^u$ is trivial and $H=M$. The vector bundle $T\times^H P^u$ is $T$-equivariantly trivial, in other words $P/H\cong T/H \times P^u$, where $T$ acts diagonally: by left translation on the first factor, and by conjugation on the second. This implies that  $\Xi_B(P/H)=\Chars(T/H)\oplus \ZZ\alpha$.

Since $P^u$ has no non-constant invertible regular functions, the action of $P^u$ on $P/H$ is the left translation on the factor $P^u$. So $\Div(P/H)^B$ has one element $D$, not stable under $P^u=U_{-\alpha}$, and $\alpha$ is an element of $\Xi_B(P/H)$ where $\rho(D)$ takes value $1$.

Moreover, the variety $X=\PP^n\times \PP^1$, where $n=\dim T/H$, is a $P$-toroidal compactification of $P/H$: to check this, let us define explicitly the action. Let $\gamma_1,\ldots,\gamma_n$ be a basis of $\Chars(T/H)$. Then an element $(t,u)\in T\ltimes \CC$ acts as follows:
\[
(t,u)\cdot([x_0,\ldots,x_n],[y_0,y_1]) = ([x_0,\gamma_1(tH)x_1,\ldots,\gamma_d(tH)x_n],[y_0,\alpha(t^{-1})(y_1+uy_0)]).
\]
The $B$-stable (i.e.\ $T$-stable) prime divisors are $D_0=\PP^n\times \{[0,1]\}$, $D_1=\PP^n\times \{[1,0]\}$, and $E_i\times \PP^1$ with $i\in\{0,\ldots,n\}$, where $E_i = \{ x_i=0\}$. The only $B$-stable prime divisor that is not $P$-stable is $D_1$, which doesn't contain any $P$-orbit, so $X$ is $P$-toroidal. The rational function $y_1/y_0$ on $X$ is $B$-semiinvariant with weight $\alpha$, has a simple zero on $D_1$, a simple pole on $D_0$, and no zeros nor poles on the other prime divisors. Then $V(\PP^n\times \PP^1)$ is defined by the equation $\langle -, \alpha\rangle \leq 0$. We conclude that $\Sigma(P/H)=\{\alpha\}$.
\end{example}

\begin{example}\label{ex:Cn}
We want to discuss the fundamental example where $P/H\cong \CC^n$ with $n\geq 1$ and $P$ is the semidirect product of $\GL(n)$ with the group of translations of $\CC^n$. The special case where $n=1$ was essentially analyzed already in Example~\ref{ex:SL2modT}.

It is convenient to work even more in general, because the computations carried out here and in the next similar examples will be used in the proof of Theorem~\ref{thm:values}. Our assumptions are the following: we require that $L\subseteq H$, with $(L,L)\cong \SL(n)$ or $\Sp(n)$ (the latter with $n$ even), and that $H^u$ is a normal subgroup of $P^u$ with $P/H\cong P^u/H^u$ abelian and isomorphic to the $(L,L)$-module $M=\CC^n$ or $M=(\CC^n)^*$.

In other words $P/H$ is the affine space $M$, where $L$ acts linearly (and contains $\SL(M)$ or $\Sp(M)$), and $P^u$ acts as the group of translations of $M$.

The rank of $P/H$ as a spherical $L$-variety is $1$, and $\Xi_B(P/H)$ is generated by the $B$-eigenvalue $\omega$ of a $B$-eigenvector $f\in M^*$. There is exactly one $B$-stable prime divisor of $P/H$, namely $\Divf(f)$. It is $L$-stable if $\dim M=1$ and an $L$-color if $\dim M>1$. We claim that $\Sigma(P/H)=\{\omega\}$. Indeed, the homogeneous space $P/H$ embeds equivariantly in $\PP^n=\PP(M\oplus\CC)$ where $P$ acts trivially on the summand $\CC$. Then $\PP^n$ is a complete $P$-toroidal embedding of $P/H$. Its unique $P$-stable prime divisor is $\PP^n\smallsetminus M$ where $f$ has a pole, whence the claim.
\end{example}

\begin{example}\label{ex:toric}
Let us consider a variation of Example~\ref{ex:Cn}, where $P/H\cong \CC^m$, but now $L$ is equal to $(\Gm)^m$ acting on $\CC^n$ via linearly independent characters denoted $-\alpha_1,\ldots,-\alpha_n$, and $P^u$ is the group of translations of $\CC^m$. Here $\Xi(P/H)=\Chars(L)\cong \ZZ^m$, the colors of $P/H$ are the coordinate hyperplanes of $\CC^m$, and a $P$-toroidal compactification of $P/H$ is $\PP^1\times\ldots\times \PP^1$. It follows that $\Sigma(P/H)=\{\alpha_1,\ldots,\alpha_m\}$.

Notice that $P/H$ and its $P$-equivariant embeddings are toric varieties, which raises the question of the relationship between our invariants and the standard theory of toric varieties and their automorphisms. Here we do not develop this theme further. 
\end{example}

\begin{example}\label{ex:CpCn1}
For later convenience we discuss another variation of Example~\ref{ex:Cn}: we assume $n>1$ and we ``enlarge'' $P^u/H^u$, assuming that $P^u/H^u\cong M\oplus \CC$ as $L$-modules. So $L$ acts on the summand $\CC$ via a character, that is non-trivial since $P/H$ is $L$-spherical. We maintain the hypothesis that $H^u$ is a normal subgroup of $P^u$ and that $P^u/H^u$ is abelian, so it is a vector group and the isomorphism with $M\oplus \CC$ is also as (additive) groups. Denote again by $\omega$ and $f$ resp.\ the highest weight and a highest weight vector of $M^*$, and by $\beta$ the highest weight of the dual of $\CC$, so that $\Xi_B(P/H)=\Span_\ZZ\{\omega,\beta\}$. In this case $P/H$ has exactly one $L$-stable prime divisor $E=M\oplus\{0\}$, and exactly one $L$-color $F=M'\oplus\CC$, where $M'=\Divf(f)$. The variety $\PP^n\times \PP^1$ is a complete $P$-toroidal embedding of $P/H$, and we conclude that $\Sigma(P/H)=\{\omega,\beta\}$. 
\end{example}

\begin{example}\label{ex:CpCn}
Maintain all the assumptions of Example~\ref{ex:CpCn1} except for the fact that $H^u$ is now allowed to be non-normal in $P^u$. The descriptions of $\Div(P/H)^B$ and of $\Xi_B(P/H)$ remain the same. Moreover, the summand $\CC$ is $T$-stable so it is the image in $P^u/H^u$ of a one-dimensional $T$-stable subgroup $U'_{-\beta}$ of $U_{-\beta}\subseteq P^u$. Then $U'_{-\beta}$ moves $E$, and Corollary~\ref{cor:PumovesD} implies that $\beta$ is non-positive on $\rho(D)$ for $D$ any $P$-stable prime divisor in any $P$-toroidal embedding of $P/H$. It follows that $\beta$ is a linear combination of spherical roots of $P/H$ with non-negative rational coefficients.

The normalizer $K$ of $H^u$ in $P^u$ is strictly bigger than $H^u$ and stable under conjugation by $L$, so $K/H^u$ in $P^u/H^u$ is $M$, $\CC$, or $M\oplus \CC$. Using the fact that the Lie algebra of $P^u$ is nilpotent and isomorphic as a vector space to the sum of $M\oplus \CC$ and the Lie algebra of $H^u$, and that $\dim M>1$, it is easy to deduce that $K/H^u$ contains the summand $\CC$. The normalizer of $K$ in $P^u$ is also $L$-stable, therefore equal to $P^u$.

To sum up, the summand $\CC$ corresponds to a subgroup of $P^u$ containing and normalizing $H^u$ and normal in $P$, let us denote it $R$. The natural map $P/H\to P/RL$ is $P$-equivariant, so the spherical root of $P/RL$, which is $\omega$ by Example~\ref{ex:Cn}, is also a linear combination of spherical roots of $P/H$ with non-negative rational coefficients by Proposition~\ref{prop:imageV}. Moreover, the condition $\langle-,\omega\rangle = 0$ defines a face of $V(P/H)$ of dimension $0$ or $1$, of dimension $1$ if and only if $\omega\in\Sigma(P/H)$.

On the other hand $X=M\times \PP^1$ is a $P$-toroidal embedding of $P/H$. We can fix a normal $P$-equivariant completion $\overline X$ of $X$, and thanks to Lemma~\ref{lemma:toroidalexist} there exists a complete $P$-toroidal embedding $Y$ of $P/H$ mapping surjectively to $\overline X$. The valuation of the $P$-stable invariant divisor $D=M\times \{\infty\}$ of $X$ is $0$ on $\omega$ and $<0$ on $\beta$, so the same holds for a $P$-stable prime divisor $E$ of $Y$ mapping dominantly to $D$. This implies that $\langle -,\omega\rangle = 0$ defines a face of $V(P/H)$ of dimension $>0$, in other words $\omega\in\Sigma(P/H)$.
\end{example}

\section{Kac-Moody groups and parabolic subgroups}\label{s:KM}

We denote by $\kmG$ a minimal Kac-Moody group with set of simple roots $S$, as constructed in \cite[Section~7.4]{Ku02} (denoted $\kmG^\text{min}$ in loc.cit.) following Kac-Peterson (see e.g.\ \cite{PK83}). We denote by $T$ the standard maximal torus of $\kmG$. If $\alpha$ is a real root, we denote by $\kmU_\alpha$ the one-parameter root subgroup of $\kmG$ associated with $\alpha$, by $\kmU$ (resp.\ $\kmU_-$) the subgroup generated by $\kmU_\alpha$ for all positive (resp.\ negative) real roots $\alpha$. We denote by $\kmB$ (resp.\ $\kmB_-$) the standard Borel subgroup containing $T$ and $\kmU$ (resp.\ $\kmU_-$). We denote by $W$ the Weyl group of $\kmG$.

A subset of simple roots is called {\em of finite type} if the corresponding Dynkin subdiagram has connected components only of finite type. A standard parabolic subgroup of $\kmG$ containing $\kmB$ or $\kmB_-$ is called {\em of finite type} if the corresponding set of simple roots is of finite type.

For any standard parabolic subgroup $\kmP\subseteq \kmG$ of finite type and of negative sign, i.e.\ containing $\kmB_-$, we denote the corresponding subset of simple roots by $S_\kmP$. We denote by $L$ the standard Levi subgroup of $\kmP$ containing $T$, and we recall that $\kmP= L\ltimes \kmP^u$ for a subgroup $\kmP^u$ of $\kmU_-$ stable under conjugation by $L$.

We also consider the ``completed'' Kac-Moody group as constructed in \cite[Section~7.1]{Ku02}, and we denote it here by $\hat\kmG$. Correspondingly, we denote by $\hat\kmB$ the standard Borel subgroup of $\hat\kmG$, so $\kmB=\hat\kmB\cap\kmG$. We denote by $\hat\kmP^+$ the standard parabolic subgroup of $\hat\kmG$ associated with $S_\kmP$ and containing $\hat\kmB$. Denote by $\tilde\omega\colon \kmG\to\kmG$ an involution of $\kmG$ as in \cite[7.4.E.1]{Ku02} sending $\kmU_\alpha$ onto $\kmU_{-\alpha}$ for all real roots $\alpha$. We recall that $\hat\kmP^+$ is equipped with the structure of a pro-group, and we fix a subgroup $N\subseteq \hat\kmP^+$ in its defining family. Then $\tilde\omega(\kmP)$ is mapped surjectively to $\hat\kmP^+/N$ (this can be shown as in the proof of \cite[Corollary~7.3.8]{Ku02}).

Therefore we can identify $\hat\kmP^+/N$ with the quotient $P=\kmP/\tilde\omega^{-1}(N\cap\kmG)$, and equip $P$ with the corresponding structure of (finite-dimensional) algebraic group. We identify the subgroups $L$ and $T$ of $\kmP$ with their (isomorphic) images in $P$, we denote $B$ the Borel subgroup of $L$ containing $T$ and associated with the set of simple roots $S_\kmP$ (i.e.\ $B=L\cap\kmB$), and as usual we denote $B_-$ the opposite Borel subgroup with respect to $T$.

\section{Spherical subgroups of Kac-Moody groups}\label{s:KMsph}

We fix from now on a parabolic subgroup $\kmP$ of finite type of $\kmG$ containing $\kmB_-$ and a finite-dimensional quotient $P$ of $\kmP$ constructed as in the previous section. We denote by $H$ a closed subgroup of $P$ and by $\kmH$ its inverse image in $\kmP$; we identify the quotient $\kmP/\kmH$ with $P/H$.

In this section we introduce in Definition~\ref{def:invariants} below the combinatorial objects that will be used in Section~\ref{s:hsd} to generalize the notion of homogeneous spherical datum, and discuss some first properties.

\begin{remark}\label{rem:finitecase}
Let us first make some remarks on the invariants defined in Section~\ref{s:invariantsfinite} in the finite-dimensional case, so suppose that $\kmG$ is finite-dimensional. Then we may assume $\kmH=H$ and $\kmP=P$. Notice that $B$ is not necessarily equal to $\kmB$, because $B$ is a Borel subgroup of the Levi subgroup $L$ of $\kmP$, and $\kmB$ is a Borel of $\kmG$. By Lemma~\ref{lemma:fiberoverflag}, the homogeneous space $P/H$ is a spherical $L$-variety if and only if $\kmG/\kmH$ is a spherical $\kmG$-variety. We assume this holds, and the same lemma assures that the lattice of $\kmG/\kmH$ is equal to $\Xi_B(P/H)$. Let us make some further remarks about the invariants of $\kmG/\kmH$ as a spherical $\kmG$-variety; these facts are known, and also follow from Theorem~\ref{thm:indep} below applied to $\kmQ=\kmG$. Consider the natural map $\pi\colon\kmG/\kmH\to\kmG/\kmP$. Given a $\kmG$-color $D$ of $\kmG/\kmH$, either it is the inverse image of a $\kmG$-color of $\kmG/\kmP$, or $D$ meets the fiber $\kmP/\kmH=\pi^{-1}(e\kmP)$, in which case the intersection $D\cap \kmP/\kmH$ is a $B$-stable prime divisor of $\kmP/\kmH$, and all such divisors of $\kmP/\kmH$ are obtained in this way. Notice also that the $\kmG$-colors of $\kmG/\kmP$ are the Schubert divisors, which are in bijection with $S\smallsetminus S_\kmP$ via the map that associates a simple root $\alpha$ to the Schubert divisor of $\kmG/\kmP$ moved by $\alpha$. In addition, we claim the following.
\begin{enumerate}
\item\label{enum:colors} The valuation of $D=\pi^{-1}(D')$, where $D'$ is a $\kmG$-color of $\kmG/\kmP$, is equal to $\alpha^\vee|_{\Xi_B(P/H)}$, if $\alpha\in S$ moves $D$ and only $D$. Otherwise, any $\alpha\in S$ moving $D$ is also a spherical root of the $\kmG$-action on $\kmG/\kmH$, and the other color $E\subset\kmG/\kmH$ moved by $\alpha$ intersects $\kmP/\kmH$.
\item The spherical roots of the $\kmG$-action on $\kmG/\kmH$ are equal to the spherical roots of the $\kmP$-action on $\kmP/\kmH$.
\end{enumerate}
These claims follow again from Theorem~\ref{thm:indep} (see Remark~\ref{rem:finitedim}). A last remark about a color $D$ as in claim (\ref{enum:colors}) above: we recall that if $D$ is moved by $\alpha\in S$ that is a spherical root and moves another color $E$, then we have $\rho(D)+\rho(E)=\alpha^\vee|_{\Xi_B(P/H)}$ by \cite[Proposition~3.4]{Lu97}.
\end{remark}

We get back to the general case of $\kmG$ being possibly infinite-dimensional.

\begin{definition}\label{def:invariants}
If $P/H$ is a spherical $L$-variety then $\kmH$ is called a {\em standard spherical subgroup of finite type} of $\kmG$. We define:
\begin{enumerate}
\item the lattice
\[
\Xi(\kmG/\kmH) = \Xi_B(P/H),
\]
whose rank is by definition the {\em rank} of $\kmG/\kmH$ and is denoted by $\Rank \kmG/\kmH$;
\item the rational vector space
\[
N(\kmG/\kmH) = \Hom_\ZZ(\Xi(\kmG/\kmH),\QQ);
\]
\item the finite set
\[
\Col(\kmG/\kmH)= \Div(P/H)^B \cup (S\smallsetminus S_\kmP)
\]
(considered as an abstract disjoint union), whose elements are the {\em colors} of $\kmG/\kmH$.
\item\label{def:invariants:moves} We say that a simple root $\alpha$ of $\kmG$ {\em moves} a color $D$ of $\kmG/\kmH$ if $D$ is not stable under the action of $U_{-\alpha}\subseteq P$ in the case where $D\in \Div(P/H)^B$, and if $D=\alpha$ in the case where $D\in S\smallsetminus S_\kmP$.
\item The map $\rho$ is defined on $\Div(P/H)^B$ with values in $N(\kmG/\kmH)$, we extend it to the set $\Col(\kmG/\kmH)$. Let $D\in \Col(\kmG/\kmH)\smallsetminus \Div(P/H)^B$, thus $D=\alpha\in S\smallsetminus S_\kmP$. By Lemma~\ref{lemma:demazure} all elements of $\Div(P/H)^B$ are $U_{-\alpha}$-stable except possibly for one element $E$, since in this case $U_{-\alpha}$ is one-dimensional and commutes with $B^u$. We define
\[
\rho(D) = \begin{cases} \alpha^\vee|_{\Xi(\kmG/\kmH)} - \rho(E) & \text{if $E$ exists,} \\ \alpha^\vee|_{\Xi(\kmG/\kmH)} & \text{otherwise.}\end{cases}
\] 
We denote this extended map also by $\rho_{\kmG/\kmH}$.
\item We define the {\em valuation cone} $V(\kmG/\kmH)=V(P/H)$ and the set of {\em spherical roots} $\Sigma(\kmG/\kmH)=\Sigma(P/H)$.
\item If $H$ contains a maximal unipotent subgroup of $P$, then $\kmG/\kmH$ is called {\em horospherical}.
\end{enumerate}
\end{definition}

Unless otherwise stated, from now on we assume that $\kmH$ is a standard spherical subgroup of finite type.

The above data associated with $\kmG/\kmH$ do not depend on the element $N$ of the defining family of $\hat\kmP^+$ used in the definition of $\kmH$. Indeed, let $N'$ be another element in this family, giving similarly as above the finite-dimensional quotient $P'$ of $\kmP$, with a subgroup $H'\subseteq P'$ such that $\kmH$ is its inverse image in $\kmP$. We may assume $N'\subset N$, then the group $P$ is a quotient of $P'$ and the image of $H'$ in the quotient is $H$, so that $P'/H'\cong P/H$ as $P'$-varieties.

The above data actually do not even depend on $\kmP$. This is not clear a priori, and will be shown in Section~\ref{s:indep}.

Conjugating $\kmH$ with an element $p\in \kmP$ does not change the above objects, since $P/H$ and $P/({}^{\overline p}H)$, where $\overline p$ is the class of $p$ in $P$, are $P$-equivariantly isomorphic. Conjugating $\kmH$ in $\kmG$ also behaves well but again requires a detailed treatment, see Section~\ref{s:conjugation}. 

We end this section rewriting in the setting of Kac-Moody groups two results obtained previously.

\begin{proposition}\label{prop:SigmainNS}
Any spherical root of $\kmG/\kmH$ is a linear combination of simple roots of $\kmG$ with non-negative rational coefficients. 
\end{proposition}
\begin{proof}
This is Theorem~\ref{thm:NS}.
\end{proof}

\begin{proposition}
We have $\Sigma(\kmG/\kmH)=\varnothing$ if and only if $\kmG/\kmH$ is horospherical.
\end{proposition}
\begin{proof}
This is Proposition~\ref{prop:horospherical}.
\end{proof}

\begin{example}\label{ex:rank0}
The case where $\kmH$ is a parabolic subgroup of $\kmG$ (containing $\kmB_-$ and contained in $\kmP$) can be regarded as the most basic one. Here $\Sigma(\kmG/\kmH)=\varnothing$ and $\Delta(\kmG/\kmH)$ can be identified with $S\smallsetminus S_\kmH$ thanks to the natural identification between $\Div(P/H)^B$ and $S_\kmP\smallsetminus S_\kmH$. Conversely, suppose that $\Sigma(\kmG/\kmH)=\varnothing$ for a standard spherical subgroup $\kmH\subseteq \kmP$ of finite type of $\kmG$. Then the spherical $L$-variety $P/H$ has rank $0$, i.e.\ it is a complete homogeneous space for $L$. In this case, up to conjugating $H$ in $P$, we have that $H$ contains $B_-P^u$, so $\kmH$ is a parabolic subgroup of $\kmG$ containing $\kmB_-$ and contained in $\kmP$.
\end{example}

\section{Types of simple roots}

In this section we generalize standard results on spherical varieties (see e.g.\ \cite{Kn14}) to the Kac-Moody case.

\begin{proposition}\label{prop:atmostone}
A simple root $\alpha\in S$ moves at most two colors of $\kmG/\kmH$, and it moves two colors if and only if $\alpha\in \Sigma(\kmG/\kmH)$. In the latter case $\langle \rho(D),\alpha\rangle =1$ for $D$ equal to either color moved by $\alpha$, and $\langle\rho(E),\alpha\rangle\leq 0$ for any color $E\in\Div(P/H)^B$ not moved by $\alpha$.
\end{proposition}
\begin{proof}
Let $\alpha\in S\smallsetminus S_\kmP$. Then we have already observed that at most one element of $\Div(P/H)^B$ is not stable under $U_{-\alpha}$. The fact that $\alpha$ moves at most two colors follows. If instead $\alpha$ is a simple root of $L$, then it moves no $L$-stable prime divisor of $P/H$ and it moves not more than two $L$-colors of $P/H$ by \cite[Section~2.7]{Lu97}.

We show the remaining assertions. Assume $\alpha\in S_\kmP$: in this case $\alpha$ moves two colors if and only if $\alpha\in\Sigma_L(P/H)$; this is equivalent to $\alpha$ being in $\Sigma(\kmG/\kmH)$ thanks to Corollary~\ref{cor:Sigma_L}. The value on $\alpha$ of any $L$-color of $P/H$ moved by $\alpha$ is computed in \cite[Section~3.2]{Lu97}. This concludes the proof for $\alpha\in S_\kmP$, except for the case where $E$ is not an $L$-color, i.e.\ if $E$ is $L$-stable. In this case $\rho(E)$ is in $V_L(P/H)$, and $\alpha\in\Sigma_L(P/H)$ assures $\langle\rho(E),\alpha\rangle \leq 0$.

It remains to consider $\alpha\in S\smallsetminus S_\kmP$: then $\alpha$ moves two colors if and only if it moves an element $D\in\Div(P/H)^{(B)}$.

Suppose that $\alpha\in S\smallsetminus S_\kmP$ moves indeed an element $D\in \Div(P/H)^{(B)}$, and let $X$ be a $P$-toroidal complete embedding of $P/H$. Notice that $D$ is a $B$-stable prime divisor that is not stable under the action of $U_{-\alpha}\subseteq P^u$, and it is stable under the action of the commutator $(U_{-\alpha},B^u)$ because $B^u$ commutes with $U_{-\alpha}$ (since $\alpha\notin S_\kmP$).

Lemma~\ref{lemma:demazure} implies that $\alpha\in\Xi_B(X)$, and that $\langle\rho(E),\alpha\rangle \leq 0$ for all $B$-stable prime divisors $E$ of $X$ different from $D$. This applies in particular if $E$ is $P$-stable, then $\langle V(X),\alpha\rangle \leq 0$, and $\alpha$ is then a linear combination with non-negative coefficients of spherical roots of $\kmG/\kmH$. From Proposition~\ref{prop:SigmainNS} we deduce that $\alpha\in\Sigma(\kmG/\kmH)$.

Suppose now that $\alpha\in S\smallsetminus S_\kmP$ belongs to $\Sigma(\kmG/\kmH)$, and choose a $P$-toroidal complete embedding $X$ of $P/H$. Since $S$ is linearly independent, we can choose a one parameter subgroup $\lambda$ of $T$ such that $\langle \lambda, \alpha\rangle =0$ and $\langle\lambda,\beta\rangle >0$ for all $\beta\in S\smallsetminus \{\alpha\}$, and consider the localization $X^\lambda$. In this case $P^\lambda$ has one-dimensional unipotent radical, equal to $U_{-\alpha}$, and $L^\lambda=T$. Moreover, the localization $X^\lambda$ is a $P^\lambda$-toroidal complete embedding of an $L^\lambda$-spherical homogeneous space $P^\lambda/K$. By Theorem~\ref{thm:loclambda} we have $\Sigma(X^\lambda)=\{\alpha\}$, hence $P^\lambda/K$ occurs in Example~\ref{ex:SL2modT} as the case that has a $B^\lambda$-stable prime divisor not stable under $U_{-\alpha}$. We conclude thanks to Proposition~\ref{prop:loclambda} that $X$ has a $B$-stable prime divisor not stable under $U_{-\alpha}$, i.e.\ $\alpha$ moves an element of $\Div(P/H)^B$.

Finally, suppose that $\alpha$ is in $S\smallsetminus S_\kmP$ and is a spherical root, denote by $D$ the color in $\Div(P/H)^B$ moved by $\alpha$ and by $E$ the other color moved by $\alpha$ (so formally $E=\alpha$). The equality $\langle \rho(D),\alpha\rangle =1$ follows from Lemma~\ref{lemma:demazure}, because again $D\in \Div(P/H)^B$ is not stable under the action of $U_{-\alpha}$ which commutes with $B^u$. This yields $\langle \rho(E),\alpha\rangle = \langle\alpha^\vee,\alpha\rangle - \langle \rho(D),\alpha\rangle = 2-1 = 1$.
\end{proof}

The last statement of the above proposition holds in greater generality, namely for all $E\in\Delta(\kmG/\kmH)$ not moved by $\alpha$. We will prove this fact later in Corollary~\ref{cor:values}.

\begin{definition}
We denote $S^a(\kmG/\kmH)=S\cap \Sigma(\kmG/\kmH)$ and $S^{2a}(\kmG/\kmH)=S\cap \frac12 \Sigma(\kmG/\kmH)$, we denote by $S^b(\kmG/\kmH)$ the set of simple roots of $\kmG$ that move exactly one color of $\kmG/\kmH$ and are not in $S^{2a}(\kmG/\kmH)$, and we denote by $S^p(\kmG/\kmH)$ the set of simple roots of $\kmG$ that move no color of $\kmG/\kmH$. The {\em type} of a simple root $\alpha$ is $(a)$, $(2a)$, $(b)$ or $(p)$ according to the set where it belongs. If $\alpha\in S^a(\kmG/\kmH)$ then we denote by $D_\alpha^+$ and $D_\alpha^-$ the two colors moved by $\alpha$ (with $D_\alpha^+\in\Div(P/H)^{(B)}$). If $\alpha$ moves only one color of $\kmG/\kmH$, then we denote this color by $D_\alpha$.
\end{definition}

\begin{definition}
We denote by $\A(\kmG/\kmH)$ the set of all colors of $\kmG/\kmH$ moved by simple roots that are also spherical roots. For such a root $\alpha$, we denote by $\A(\kmG/\kmH,\alpha)$ or simply $\A(\alpha)$ the set $\{D_\alpha^+,D_\alpha^-\}$ of the two colors moved by $\alpha$.
\end{definition}

\begin{proposition}\label{prop:typeofS}
The set $S$ is the disjoint union
\[
S = S^a(\kmG/\kmH)\cup S^{2a}(\kmG/\kmH) \cup S^b(\kmG/\kmH)\cup S^p(\kmG/\kmH).
\]
Any simple root in $S^{2a}(\kmG/\kmH)$ moves exactly one color, and $S^{2a}(\kmG/\kmH)\cup S^p(\kmG/\kmH)\subseteq S_\kmP$. For any $\alpha\in S$, according to its type we have the following formulae for the colors moved by $\alpha$:
\[
\begin{array}{c|l}
\text{type} & \\
\hline
(a) & \rho(D_\alpha^+) + \rho(D_\alpha^-) = \alpha^\vee|_{\Xi(\kmG/\kmH)} \\
(2a) & \rho(D_\alpha) = \frac12\alpha^\vee|_{\Xi(\kmG/\kmH)} \\
(b) & \rho(D_\alpha) = \alpha^\vee|_{\Xi(\kmG/\kmH)}
\end{array}
\]
\end{proposition}
\begin{proof}
The four subsets of $S$ are disjoint by definition, except for $S^{2a}(\kmG/\kmH)$ that may intersect $S^p(\kmG/\kmH)$. Let $\alpha\in S^{2a}(\kmG/\kmH)$. If $\alpha\in S_\kmP$ then $\alpha\in S_\kmP\cap \frac12\Sigma_L(P/H)$, so $\alpha$ moves exactly one color of the spherical $L$-variety $P/H$ (see e.g.\ \cite[Section~2.7]{Lu97}).

Suppose now that $\alpha\in S\smallsetminus S_\kmP$, and let $X$ be a complete $P$-toroidal embedding of $P/H$. As in the proof of Proposition~\ref{prop:atmostone}, we consider a localization of $X$ with respect to a one-parameter subgroup $\lambda$ of $T$ such that $\langle\lambda,\alpha\rangle = 0$ and $\langle\lambda,\beta\rangle >0$ for all $\beta\in S\smallsetminus \{\alpha\}$. Then $X^\lambda$ occurs in Example~\ref{ex:SL2modT}, but no variety in this example has spherical root $2\alpha$: contradiction.

It follows that the disjoint union $S^{2a}(\kmG/\kmH)\cup S^b(\kmG/\kmH)$ is the set of simple roots that move exactly one color each, and the union of the four subsets of the statement of the proposition is $S$. Moreover we have showed that $S^{2a}(\kmG/\kmH)\subseteq S_\kmP$, and since any root in $S\smallsetminus S_\kmP$ moves at least one color by definition, we have also $S^p(\kmG/\kmH)\subseteq S_\kmP$.

It remains the formulae for the colors moved by a simple root $\alpha$. If $\alpha\in S_\kmP$ then they are proved in \cite[Proposition~3.4]{Lu97}. Otherwise $\alpha$ is in $S\smallsetminus S_\kmP$ and it is either in $S^{a}(\kmG/\kmH)$ or $S^b(\kmG/\kmH)$ thanks to the first part of the proof. In this case the corresponding formula holds by definition.
\end{proof}

\section{Independence from $\kmP$}\label{s:indep}

Suppose that $\kmH$ is contained in another parabolic subgroup $\kmQ$ of finite type containing $\kmB_-$. Then we can apply to $\kmQ$ the same construction of Section~\ref{s:KM}, based on the pro-group structure of $\hat\kmQ^+$. Any element in the defining family of $\hat\kmP^+$ as a pro-group contains some element of the defining family of $\hat\kmQ^+$, which shows that $\kmH$ is also the inverse image in $\kmQ$ of an algebraic subgroup of $Q$, the latter being a finite-dimensional quotient of $\kmQ$. In other words $\kmQ$ can play the same role of $\kmP$ in our constructions.

We show in this section that choosing $\kmQ$ instead of $\kmP$ essentially doesn't change the objects introduced in Definition~\ref{def:invariants}. In this section we use the notations $S^p(\kmG/\kmH)_\kmP$, $\Sigma(\kmG/\kmH)_\kmP$, $\rho_{\kmG/\kmH,\kmP}$, etc., to underline the possible dependence on $\kmP$.

\begin{theorem}\label{thm:indep}
Let $\kmQ$ be a parabolic subgroup of $\kmG$ of finite type containing $\kmB_-$ and $\kmH$. Then
\[
\begin{array}{rcl}
S^p(\kmG/\kmH)_\kmP &=& S^p(\kmG/\kmH)_\kmQ\\
\Xi(\kmG/\kmH)_\kmP &=& \Xi(\kmG/\kmH)_\kmQ\\
\Sigma(\kmG/\kmH)_\kmP &=&\Sigma(\kmG/\kmH)_\kmQ,
\end{array}
\]
and there is a bijection
\[
\vartheta\colon \Delta(\kmG/\kmH)_\kmP\to\Delta(\kmG/\kmH)_\kmQ
\]
such that $\vartheta(\A(\kmG/\kmH)_\kmP)=\vartheta(\A(\kmG/\kmH)_\kmQ)$, such that $\rho_{\kmG/\kmH,\kmP} = \rho_{\kmG/\kmH,\kmQ}\circ\vartheta$, and such that for all $D\in \Delta(\kmG/\kmH)_\kmP$ and all $\alpha\in S$ the color $D$ is moved by $\alpha$ if and only if $\vartheta(D)$ is.
\end{theorem}
\begin{proof}
The intersection $\kmQ\cap \kmP$ is also a parabolic subgroup of $\kmG$ of finite type containing $\kmB_-$ and $\kmH$, therefore we may suppose that $\kmQ$ is contained in $\kmP$, and also that the finite-dimensional quotient $Q$ of $\kmQ$ as above is the projection of $\kmQ$ in $P$. More precisely, the projection $Q$ is a parabolic subgroup of $P$, and $P/Q$ is a spherical $L$-variety. But $Q^u$ contains $P^u$, so $P/Q$ is isomorphic to $(P/P^u)/(Q/P^u)$. The quotient $P/P^u$ is isomorphic to $L$ and $Q/P^u$ is a parabolic subgroup, so the rank of $P/Q$ as a spherical $L$-variety is $0$ and $P/Q\cong L/R$ where $R$ is a parabolic subgroup of $L$.

Denote by $L_Q$ the Levi subgroup of $Q$ containing $T$, and by $B_Q$ the Borel subgroup $B\cap L_Q$ of $L_Q$. Consider the natural map $\varphi\colon P/H\to P/Q$. We have $\Xi_B(P/H)=\Xi_{B_Q}(Q/H)$, because $Q/H$ is the fiber of $\varphi$ over $eQ$ and $BQ$ is open in $P$. Thanks to Lemma~\ref{lemma:extend} we can choose a complete $P$-toroidal embedding $X$ of $P/H$ such that $\varphi$ extends to a $P$-equivariant map $X\to P/Q$. Denote this map also by $\varphi$, and denote by $X_Q$ the fiber $\varphi^{-1}(eQ)$ in $X$.

Since $BQ$ is open in $P$ we have that intersecting with $\varphi^{-1}(eQ)$ induces a bijection between $\Div(X_Q)^{B_Q}$ and the subset of $\Div(X)^B$ of those elements mapped dominantly to $P/Q$. On the other hand any $P$-orbit of $X$ intersects $X_Q$ in a $Q$-orbit, and if $Y$ is a closed $Q$-orbit of $X_Q$ then $Q^u\supseteq P^u$ acts trivially on $Y$ and $BY$ is a closed $P$-orbit of $X$. It follows that $X_Q$ is a complete $Q$-toroidal embedding of $Q/H$, and that $\Sigma(P/H)=\Sigma(Q/H)$.

This yields the equalities $\Xi(\kmG/\kmH)_\kmP = \Xi(\kmG/\kmH)_\kmQ$ and $\Sigma(\kmG/\kmH)_\kmP =\Sigma(\kmG/\kmH)_\kmQ$.

Let us now consider the set of colors. Denote by $\Delta_1$ the set of the elements of $\Div(P/H)^B$ mapped dominantly to $P/Q$, and by $\Delta_2$ the complement $\Div(P/H)^B\smallsetminus \Delta_1$. Intersecting with $\varphi^{-1}(eQ)$ induces a bijection between $\Delta_1$ and $\Div(Q/H)^{B_Q}$. The elements of $\Delta_2$ are precisely the inverse images via $\varphi$ of the elements of $\Div(P/Q)^{B}$. Since $P/Q$ is a complete $L$-homogeneous space, the set $\Div(P/Q)^{B}$ is in bijection with $S_\kmP\smallsetminus S_\kmQ$, via the map $\tau$ that associates to $E\in\Div(P/Q)^{B}$ the unique simple root $\alpha$ that moves $E$.

Notice also that $\Delta(\kmG/\kmH)_\kmP$ is the disjoint union $\Delta_1\cup\Delta_2\cup (S\smallsetminus S_\kmP)$, so we can define the map $\vartheta$ as follows:
\begin{enumerate}
\item $\vartheta(D) = D\cap \varphi^{-1}(eQ)$ if $D\in \Delta_1$,
\item $\vartheta(D)=\tau(E)$ if $D\in \Delta_2$ and $D=\varphi^{-1}(E)$ for $E\in \Div(P/Q)^{B}$;
\item $\vartheta(D)=D$ if $D\in S\smallsetminus S_\kmP$.
\end{enumerate}
This defines a bijection $\theta\colon \Delta(\kmG/\kmH)_\kmP\to \Delta(\kmG/\kmH)_\kmQ$ such that $\vartheta(\Delta_1) = \Div(Q/H)^{B_Q}$ and $\vartheta(\Delta_2\cup (S\smallsetminus S_\kmP))=S\smallsetminus S_\kmQ$, and $\vartheta$ satisfies the required property that $D\in \Delta(\kmG/\kmH)_\kmP$ is moved by $\alpha\in S$ if and only if $\vartheta(D)$ is.

It remains to check that $\rho_{\kmG/\kmH,\kmP} = \rho_{\kmG/\kmH,\kmQ}\circ\vartheta$, so let $D\in \Delta(\kmG/\kmH)_\kmP$. If $D\in\Delta_1$ it holds $\rho_{P/H}(D) = \rho_{Q/H}(D\cap \varphi^{-1}(eQ))$, hence $\rho_{\kmG/\kmH,\kmP}(D) = \rho_{\kmG/\kmH,\kmQ}(\vartheta(D))$. Otherwise $\vartheta(D)=\alpha\in S\smallsetminus S_\kmQ$.

By Proposition~\ref{prop:typeofS}, the element $\vartheta(D)$ falls in one of the following cases:
\begin{enumerate}
\item\label{enum:typeb} $\rho_{\kmG/\kmH,\kmQ}(\vartheta(D))=\alpha^\vee|_{\Xi(\kmG/\kmH)_\kmQ}$, if $\alpha\notin \Sigma(\kmG/\kmH)_\kmQ$ and $2\alpha\notin \Sigma(\kmG/\kmH)_\kmQ$, or
\item\label{enum:type2a} $\rho_{\kmG/\kmH,\kmQ}(\vartheta(D))=\frac12\alpha^\vee|_{\Xi(\kmG/\kmH)_\kmQ}$, if $2\alpha\in \Sigma(\kmG/\kmH)_\kmQ$, or
\item\label{enum:typea} $\rho_{\kmG/\kmH,\kmQ}(\vartheta(D))=\alpha^\vee|_{\Xi(\kmG/\kmH)_\kmQ}-\rho_{\kmG/\kmH,\kmQ}(E)$, if $\alpha\in \Sigma(\kmG/\kmH)_\kmQ$ and $E\in \Div(Q/H)^{B_Q}$ is moved by $\alpha$.
\end{enumerate}

Since $\Sigma(\kmG/\kmH)_\kmP=\Sigma(\kmG/\kmH)_\kmQ$, we have that $\rho_{\kmG/\kmH,\kmP}(D) = \rho_{\kmG/\kmH,\kmQ}(\vartheta(D))$ if cases (\ref{enum:typeb}) or (\ref{enum:type2a}) occur.

Suppose now that case (\ref{enum:typea}) occurs. Then $\alpha$ moves two elements of $\Delta(\kmG/\kmH)_\kmP$: one is $D$, and the other is $E'$ such that $\vartheta(E')=E$. Therefore $E'\in \Delta_1$, whence $\rho_{\kmG/\kmH,\kmP}(E') = \rho_{\kmG/\kmH,\kmQ}(E)$. Proposition~\ref{prop:typeofS} yields also $\rho_{\kmG/\kmH,\kmP}(D)=\alpha^\vee|_{\Xi(\kmG/\kmH)_\kmP}-\rho_{\kmG/\kmH,\kmP}(E')$, which implies the required $\rho_{\kmG/\kmH,\kmP}(D) = \rho_{\kmG/\kmH,\kmQ}(\vartheta(D))$. The proof of the existence of $\vartheta$ is complete.

The equality $S^p(\kmG/\kmH)_\kmP = S^p(\kmG/\kmH)_\kmQ$ follows, because $S^p(\kmG/\kmH)_\kmP$ is the set of simple roots moving no element of $\Delta(\kmG/\kmH)_\kmP$, and the same, mutatis mutandis, holds for $S^p(\kmG/\kmH)_\kmQ$.
\end{proof}

\begin{remark}\label{rem:finitedim}
If $\kmG$ is finite-dimensional we may set $\kmQ=\kmG$. Then the objects defined in Definition~\ref{def:invariants} with respect to $\kmQ$ are the usual invariants of $\kmG/\kmH$ as a spherical $\kmG$-variety.
\end{remark}

\section{Localization at simple roots}

In this and in the next section we generalize to the Kac-Moody case the procedure of {\em localization} at simple roots and at spherical roots introduced by Luna in \cite{Lu01} (see also \cite{Kn14}).

\begin{definition}
Let $S'\subseteq S$. A one-parameter subgroup $\lambda$ of $T$ is {\em adapted to $S'$} if $\langle \lambda, \alpha\rangle =0$ for all $\alpha\in S'$ and $\langle\lambda,\beta\rangle >0$ for all $\beta\in S\smallsetminus S'$. In this case, if $X$ is a complete $P$-toroidal embedding of $P/H$, then $X^\lambda$ is called {\em a localization of $X$ at $S'$}.
\end{definition}

For any $S'\subseteq S$ we denote by $\kmG^{S'}$ the minimal Kac-Moody group with maximal torus $T$ and set of simple roots $S'$; it is naturally a subgroup of $\kmG$. As for $\kmG$, denote by $\kmB^{S'}$ and $\kmB_-^{S'}$ the standard Borel subgroups of $\kmG^{S'}$ containing $T$.

\begin{lemma}\label{lemma:locS}
Let $S'\subseteq S$ and $\lambda$ be a one parameter subgroup of $T$ adapted to $S'$. Then $P^\lambda$ is a finite-dimensional quotient of $\kmP^{S'}$, where the latter is the parabolic subgroup of $\kmG^{S'}$ associated with $S'\cap S_\kmP$ and containing $\kmB_-^{S'}$.
\end{lemma}
\begin{proof}
This follows from Lemma~\ref{lemma:locP}.
\end{proof}

\begin{theorem}\label{thm:locS}
In the hypotheses of Lemma~\ref{lemma:locS}, let $X$ be a complete $P$-toroidal embedding of $P/H$. Denoting by $P^\lambda/K$ the open $P^\lambda$-orbit of $X^\lambda$ and by $\kmK$ the inverse image of $K$ in $\kmP^{S'}$, the subgroup $\kmK$ of $\kmG^{S'}$ is standard spherical of finite type, and we have:
\[
\begin{array}{lcl}
S^p(\kmG^{S'}/\kmK) &=& S^p(\kmG/\kmH)\cap S'\\[5pt]
\Sigma(\kmG^{S'}/\kmK) &=& \Sigma(\kmG/\kmH)\cap \Span_\QQ S'\\[5pt]
\Xi(\kmG^{S'}/\kmK) &=& \Xi(\kmG/\kmH)\cap \lambda^\perp.
\end{array}
\]
Moreover, the set $\Delta(\kmG^{S'}/\kmK)$ can be identified with the subset of $\Delta(\kmG/\kmH)$ of elements moved by some $\alpha\in S'$, in such a way that the identification respects the property of being moved by a simple root in $S'$ and is compatible with the map $\rho$, in the sense that
\[
\rho_{\kmG^{S'}/\kmK}(D) = \rho_{\kmG/\kmH}(D)|_{\Xi(\kmG^{S'}/\kmK)}
\]
for any $D\in \Delta(\kmG^{S'}/\kmK)$.
\end{theorem}
\begin{proof}
Everything follows from Proposition~\ref{prop:loclambda} and Theorem~\ref{thm:loclambda}.
\end{proof}

\begin{remark}\label{rem:finitetype}
If $S'$ in the above theorem is of finite type, then $\kmG^{S'}$ and $\kmK$ are finite-dimensional groups, therefore the set of spherical roots $\Sigma(\kmG/\kmH)\cap \Span_\QQ S'$ is the set of spherical roots of a finite-dimensional spherical variety.
\end{remark}

\begin{corollary}\label{cor:multipleofsimple}
If $\sigma\in\Sigma(\kmG/\kmH)$ is a multiple of a simple root $\alpha$, then $\sigma=\alpha$ or $\sigma=2\alpha$.
\end{corollary}
\begin{proof}
Let $G$ be a connected reductive group and $X$ a spherical $G$-homogeneous space. Then $X$ admits a geometric quotient $X/Z(G)^\circ$, which is a spherical $(G,G)$-homogeneous space. It has the same spherical roots as $X$, thanks to \cite[Theorem~5.4]{Kn91}. The corollary follows now from the well-known classification of spherical $\SL(2)$-varieties, thanks to Remark~\ref{rem:finitetype} applied to the subset of simple roots $S'=\{\alpha\}$.
\end{proof}

\section{Localization at spherical roots}

\begin{theorem}\label{thm:locSigma}
Let $Y$ be an $L$-orbit of a $P$-toroidal embedding $X$ of $P/H$. Suppose that $X'=\overline Y$ is $P$-stable and let $P/K$ be the open $P$-orbit of $X'$. Then the inverse image $\kmK$ of $K$ in $P$ is a standard spherical subgroup of finite type of $\kmG$ contained in $\kmP$, and
\[
\begin{array}{lcl}
S^p(\kmG/\kmK) &=& S^p(\kmG/\kmH)\\[5pt]
\Sigma(\kmG/\kmK) &=& \Sigma(\kmG/\kmH)\cap \Cone_Y^\perp\\[5pt]
\Xi(\kmG/\kmK) &=& \Xi(\kmG/\kmH)\cap \Cone_Y^\perp.
\end{array}
\]
Moreover, the set $\A(\kmG/\kmK)$ can be identified with the subset
\[
\bigcup_{\alpha\in \Sigma(\kmG/\kmK)} \left\{ D_\alpha^+,D_\alpha^- \right\}
\]
of $\A(\kmG/\kmH)$ in such a way that the identification respects the property of being moved by a simple root in $S$ and is compatible with the map $\rho$, in the sense that
\[
\rho_{\kmG/\kmK}(D) = \rho_{\kmG/\kmH}(D)|_{\Xi(\kmG/\kmK)}
\]
for any $D\in \A(\kmG/\kmK)$.
\end{theorem}
\begin{proof}
The homogeneous space $P/K$ is $L$-spherical, therefore the inverse image $\kmK$ of $K$ in $\kmP$ is a spherical subgroup of finite type of $\kmG$. Thanks to Proposition~\ref{prop:locSigma1}, the variety $X'$ is a $P$-toroidal embedding of $P/K$ with lattice and spherical roots as desired. 

Let us prove the assertion on $S^p$. First observe that $S^p(\kmG/\kmK)$ is equal to the set of simple roots of $L$ moving no $L$-color of $X'$. Then by Corollary~\ref{cor:Ltoroidaln} the open $L$-orbit of $X$ has in $X$ an $L$-stable neighborhood $Z$ that is an $L$-toroidal embedding and contains the open $L$-orbit of $X'$. At this point the assertion on $S^p$ follows from \cite[Proposition~6.1]{Kn14}.

The last assertion of the theorem follows if we prove an analogue of \cite[Lemma~6.2]{Kn14}, i.e.:
\begin{enumerate}
\item\label{proof:locsigmaSa} we have $S^a(\kmG/\kmH)\cap \Cone_Y^\perp=S^a(\kmG/\kmK)$;
\item\label{proof:locsigmaDcapXp} if $D\in \Div(P/H)^B$ is a color of $\kmG/\kmH$ moved by $\alpha\in S^a(\kmG/\kmH)\cap \Cone_Y^\perp$, then $\overline D\cap P/K$ contains a color $E$ of $\kmG/\kmK$ moved by $\alpha$ so that $\rho_{\kmG/\kmK}(E)$ is the restriction of $\rho_{\kmG/\kmH}(D)$ to $\Xi(\kmG/\kmK)$.
\end{enumerate}
Part (\ref{proof:locsigmaSa}) follows from $\Sigma(\kmG/\kmK) = \Sigma(\kmG/\kmH)\cap \Cone_Y^\perp$, let us prove part (\ref{proof:locsigmaDcapXp}). If $\alpha\in S_\kmP$ then $D$ is an $L$-color of $X$, and in this case we conclude using \cite[Lemma~6.2]{Kn14} applied to an $L$-toroidal neighborhood of $Y$ in $X$ thanks to Corollary~\ref{cor:Ltoroidaln}. We may suppose $\alpha\in S\smallsetminus S_\kmP$, so $\alpha$ moves in $X$ exactly one $B$-stable prime divisor, which is the closure of an element $D\in\Div(P/H)^B$, and exactly one element $E\in \Div(P/K)^B$. The two functionals associated with $D$ and $E$ are both $1$ on $\alpha(\in\Xi(\kmG/\kmK)\subseteq \Xi(\kmG/\kmH))$, in particular they are non-zero.

By Lemma~\ref{lemma:Bstable} there is an element $\widetilde D\in \Div(X)^B$ that contains $E$ but not $X'$. Since $E$ is moved by $\alpha$ but $X'$ is $P$-stable, we deduce that $\widetilde D$ is moved by $\alpha$, so $\widetilde D=\overline D$. In other words $E$ is an irreducible component of $\overline D\cap P/K$. Now, if  a $B$-semiinvariant rational function $f$ on $X'$ has no pole on $\overline D$ once extended to $X$, then $f$ has no pole on $E$. This implies that $\rho_{\kmG/\kmK}(E)$ is a positive rational multiple of $\rho_{\kmG/\kmH}(D)|_{\Xi(\kmG/\kmK)}$. Since these two functionals are both $1$ on $\alpha$, they are equal.
\end{proof}

\begin{definition}
In the assumptions of Theorem~\ref{thm:locSigma}, the variety $X'$ is called {\em a localization of $X$ at $\Sigma'=\Sigma(\kmG/\kmK)$}.
\end{definition}

\section{Other relations between spherical roots and colors}\label{s:other}

The goal of this section is to prove a crucial property of colors moved by simple roots that are spherical roots, given in Theorem~\ref{thm:values} below. In the finite-dimensional case, it was proved by Luna in \cite{Lu01} in characteristic $0$ by reducing to the case of spherical varieties of rank $2$, and in arbitrary characteristic with other methods by Knop in \cite{Kn14}.

\begin{proposition}\label{prop:commoncolor}
Two different simple roots $\alpha$ and $\beta$ move the same color of $\kmG/\kmH$ if and only if one of the two following mutually exclusive situations occur.
\begin{enumerate}
\item\label{prop:commoncolor:a} Both $\alpha$ and $\beta$ are spherical roots, and they move the same color.
\item\label{prop:commoncolor:apa} They satisfy $\langle\alpha^\vee, \beta\rangle = 0$, and $\alpha+\beta$ or $\frac12(\alpha+\beta)$ is in $\Sigma(\kmG/\kmH)$. Also, both $\alpha$ and $\beta$ are not in $\Sigma(\kmG/\kmH)$.
\end{enumerate}
Moreover, if condition (\ref{prop:commoncolor:apa}) holds, then both $\alpha$ and $\beta$ are in $S^b(\kmG/\kmH)$.
\end{proposition}
\begin{proof}
We prove the ``only if'' part: supposing that condition (\ref{prop:commoncolor:a}) is not satisfied, we show that (\ref{prop:commoncolor:apa}) holds.

If $\alpha$ and $\beta$ move the same color $D$ then $D\in\Div(P/H)^B$. Suppose that $\alpha\notin\Sigma(\kmG/\kmH)$, then it is either in $S^{2a}(\kmG/\kmH)$ or in $S^{b}(\kmG/\kmH)$. In either case the root $\alpha$ is in $S_\kmP$. Then $D$ is not $L$-stable so it is an $L$-color of $P/H$, and $\rho(D)$ is a positive rational multiple of $\alpha^\vee|_{\Xi(\kmG/\kmH)}$ by Proposition~\ref{prop:typeofS}. Suppose now $\beta\in\Sigma(\kmG/\kmH)$: then $\langle\rho(D),\beta\rangle$ is equal to $1$ by Proposition~\ref{prop:atmostone} and is equal to a positive multiple of $\langle\alpha^\vee,\beta\rangle$, which is impossible. Hence $\beta\notin\Sigma(\kmG/\kmH)$, and for the same reasons as $\alpha$ we conclude $\beta\in S_\kmP$. Then the rest follows from \cite[Proposition~3.2]{Lu01} applied to the $L$-spherical variety $P/H$.

For sake of completeness, we recall that the argument in loc.cit.\ is essentially the following. First, we may apply Remark~\ref{rem:finitetype} to the subset of simple roots $S'=\{\alpha,\beta\}$ of $S_\kmP$. Then one reduces the problem, in a way similar to the proof of Corollary~\ref{cor:multipleofsimple}, to the case of a homogeneous space that is spherical under the action of a semisimple group of rank $2$, and such homogeneous spaces are well-known. For more details see also \cite[Proof of Proposition~5.4]{Kn14}.

The ``if'' part amounts to the fact that if condition (\ref{prop:commoncolor:apa}) is met, then $\alpha$ and $\beta$ move the same color. Here $\langle\alpha^\vee,\beta\rangle=0$, so $S'=\{\alpha,\beta\}$ is of finite type, and we conclude as before applying Remark~\ref{rem:finitetype} together with \cite[Proposition~3.2]{Lu01}.

The last assertion and the fact that (\ref{prop:commoncolor:a}) and (\ref{prop:commoncolor:apa}) are mutually exclusive follow in the same way.
\end{proof}

\begin{definition}
A set of spherical roots $\Sigma'\subseteq \Sigma(\kmG/\kmH)$ is called {\em a set of neighbors} if there exists $v\in V(\kmG/\kmH)$ such that
\[
\Sigma'=\{ \sigma\in\Sigma(\kmG/\kmH) \;|\; v(\sigma)=0 \}.
\] 
Two spherical roots $\sigma,\tau\in\Sigma(\kmG/\kmH)$ are {\em neighbors} if $\{\sigma,\tau\}$ is a set of neighbors.
\end{definition}

Notice that any two spherical roots that are multiples of simple roots are neighbors: this is shown exactly as in \cite[Lemma~6.4]{Kn14}, thanks to Proposition~\ref{prop:SigmainNS}.

\begin{remark}\label{rem:locSigmaexist}
Let $X$ be a complete $P$-toroidal embedding of $P/H$. We remark that localizations at $\Sigma'$ of $X$ exist for any set of neighbors $\Sigma'\subseteq \Sigma(\kmG/\kmH)$. Indeed, the set $F=V(X)\cap\{\sigma=0\text{ for all } \sigma\in \Sigma' \}$ is a face of $V(X)$. Then it is enough to consider $X'=\overline Y$ where $Y$ is an $L$-orbit such that $\Cone_Y$ is contained in $F$ and has the same dimension of $F$. Notice that the rank of such localization is $\Rank X - \dim F$.
\end{remark}

We prepare for the proof of the main result of this section, Theorem~\ref{thm:values}, with three lemmas.

\begin{lemma}\label{lemma:rank1}
Let $G$ be a connected reductive group and $M$ be a spherical $G$-module of rank $1$. If $M$ has dimension $1$ then it has exactly one $G$-stable prime divisor and no $G$-color, otherwise it has exactly one $G$-color and no $G$-stable prime divisor. 
\end{lemma}
\begin{proof}
Let us denote by $B_G$ a Borel subgroup of $G$. If $M$ has dimension $1$ then $\{0\}$ is the only $B_G$-stable prime divisor of $M$ and it is $G$-stable. If $M$ has dimension greater than $1$ but only rank $1$, then $(G,G)$ acts non-trivially on it. On the other hand $M$ is factorial, so any $B_G$-stable prime divisor $D$ has a global equation. Considering once again that $M$ has rank $1$, there can only be one $B_G$-stable prime divisor, which is a $G$-color and precisely the zero set of a highest weight vector of $M^*$.
\end{proof}

\begin{lemma}\label{lemma:modules}
Let $M$ be a spherical module under the action of a connected reductive group $G\subseteq\GL(M)$.
\begin{enumerate}
\item\label{lemma:modules:nosphroots} If $\Sigma_G(M)=\varnothing$, then $(G,G)$ is a product of special linear groups and of symplectic groups, and $M$ is the direct sum of the corresponding defining representations or their duals.
\item If $M$ has rank $2$ and $|\Sigma_G(M)|=1$, then $M$ is an irreducible $G$-module. In addition, suppose that $M$ has one $G$-stable prime divisor $E$, and denote by $\sigma$ the element of $\Sigma_G(M)$. Then $\langle\rho(E),\sigma\rangle=-1$, and $\sigma$ is a linear combination with positive coefficients of all simple roots of $G$.
\end{enumerate} 
\end{lemma}
\begin{proof}
We use the classification of spherical modules, and refer for this to \cite{Kn98}. First, let $n$ be maximal such that there exist decompositions $G=G_1\times\ldots\times G_n$ and $M=M_1\oplus\ldots\oplus M_n$ in such a way that for all $i$ the factor $G_i$ is a non-trivial group and acts trivially on $M_j$ for all $j\neq i$. We remark that the factors $G_i$ need not be simple, and the summands $M_i$ need not be irreducible $G_i$-modules.

For all $i\in\{1,\ldots,n\}$ the $G_i$-module $M_i$ is spherical, and it appears in the list of modules of \cite[Section~5]{Kn98} up to a certain equivalence. Precisely, denote by $\overline{G_i}$ the product, inside $\GL(M_i)$, of $G_i$ and a maximal torus of $\GL(M_i)^G_i$. Then, by \cite[Theorem~5.2]{Kn98}, the group $\overline{G_i}$ and the module $M_i$ appear in \cite[Section~5]{Kn98} up to conjugating $\overline{G_i}$ by an element of $\GL(M_i)$.

The ranks of $M_1,\ldots,M_n$ add up to the rank of $M$, and $\Sigma_G(M)=\Sigma_{G_1}(M_1)\cup\ldots\cup \Sigma_{G_n}(M_n)$. At this point the lemma follows by inspection on the list of \cite[Section~5]{Kn98}.
\end{proof}

\begin{lemma}\label{lemma:affine}
Suppose that $H$ is not contained in any proper parabolic subgroup of $P$, that $P/H$ has rank $2$, exactly one $L$-stable prime divisor $E$ and exactly one $L$-color $F$. If $|\Sigma_L(P/H)|=0$ then $H$ contains a Levi subgroup of $P$ and $P/H$ is $L$-equivariantly isomorphic to a spherical $L$-module of the form $M\oplus \CC$ as in Example~\ref{ex:CpCn} with $\dim M>1$. If $|\Sigma_L(P/H)|=1$, then one of the following is true:
\begin{enumerate}
\item\label{lemma:affine:zero} the element of $\Sigma_L(P/H)$ is $0$ on $\rho(E)$, or
\item\label{lemma:affine:irred} the group $H$ contains a Levi subgroup of $P$ and $P/H$ is $L$-equivariantly isomorphic to an irreducible spherical $L$-module of dimension $>1$.
\end{enumerate} 
\end{lemma}
\begin{proof}
We may assume that $H$ has a Levi subgroup $K$ contained in $L$. We apply Lemma~\ref{lemma:PmodH} and obtain that $H^u\subseteq P^u$ and $P/H$ is $L$-equivariantly isomorphic to the smooth affine variety $L\times^{K}P^u/H^u$. The fiber $P^u/H^u$ over $eK$ of the corresponding map $P/H\to L/K$ is a $K$-module, spherical under the action of $K^\circ$.

We proceed to prove the lemma under the assumption that $|\Sigma_L(P/H)|=0$. This implies that $\Sigma_L(L/K)=\varnothing$, i.e.\ $K$ contains a maximal unipotent subgroup of $L$. But $K$ is reductive, so it must contain $(L,L)$ and hence the quotient $L/K$ is a torus. 

This implies that we may extend the action of $K$ on $P^u/H^u$ to an action of $L$, so the $L$-equivariant vector bundle $L\times^{K}P^u/H^u$ is trivial. In other words $P/H$ is $L$-equivariantly isomorphic to $L/K \times P^u/H^u$. Then the ranks of $L/K$ and of $P^u/H^u$ as spherical $L$-varieties add up to the rank of $P/H$, the spherical module $P^u/H^u$ has no spherical root for the action of $L$, it has exactly one $L$-stable prime divisor and one $L$-color.

If $L/K$ has rank $2$ or $1$ then $P^u/H^u$ has rank $0$ (i.e.\ is trivial) or $1$. The first possibility is excluded because the module $\{0\}$ has no $L$-color nor any $L$-stable prime divisor, and the possibility of $P^u/H^u$ having rank $1$ is excluded by Lemma~\ref{lemma:rank1}. Therefore $L/K$ has rank $0$, so $L=K$ and $P/H\cong P^u/H^u$.

Lemma~\ref{lemma:modules} implies that  $P^u/H^u$ under the action of $(L,L)$ is a sum of standard representations (or their duals) of special linear groups and of symplectic groups. Each such summand has rank $1$ as a spherical $L$-variety; if it has dimension $1$ then it has one $L$-stable prime divisor and no $L$-color, otherwise it has one $L$-color and no $L$-stable prime divisors. It follows that $P^u/H^u\cong M\oplus\CC$ with $\dim M>1$.

Let us examine now the case $|\Sigma_L(P/H)|=1$, and denote by $\sigma$ the element of $\Sigma_L(P/H)$. We discuss separately the cases where $\Sigma_L(L/K)$ is empty and where it is not.

If $\Sigma_L(L/K)$ is empty then $L/K$ is again a torus, the action of $K$ on $P^u/H^u$ extends to an action of $L$, and $P^u/H^u$ has again exactly one $L$-stable prime divisor and one $L$-color, but this time $|\Sigma_L(P^u/H^u)|=|\Sigma_L(P/H)|=1$. If $L/K$ has rank $2$ or $1$ then $P^u/H^u$ has rank $0$ or $1$, impossible by the same argument as above. So $L=K$ and $P/H\cong P^u/H^u$. In this case the fact that $\Sigma_L(P^u/H^u)\neq \varnothing$ implies that $(L,L)$ acts non-trivially on $P^u/H^u$, which has then dimension $>1$. Since $P^u/H^u$ has rank $2$ and $|\Sigma_L(P^u/H^u)|=1$, Lemma~\ref{lemma:modules} implies that $P^u/H^u$ is irreducible. In this case the statement (\ref{lemma:affine:irred}) of the lemma holds.

Finally, suppose that $\Sigma_L(L/K)$ is not empty. Then the colored subspace associated with the $L$-equivariant map $P/H\to L/K$ has first entry (the vector subspace) equal to $\sigma^\perp$. Moreover, $L/K$ has no $L$-stable prime divisor, so the pull-back of any $B$-semiinvariant function from $L/K$ to $P/H$ has neither a zero nor a pole on $E$. In other words $\langle\rho(E),\sigma\rangle =0$, and the statement (\ref{lemma:affine:zero}) of the lemma holds.
\end{proof}

\begin{theorem}\label{thm:values}
Let $\alpha\in S^a(\kmG/\kmH)$ and $D\in \A(\alpha)$. If $\sigma\in \Sigma(\kmG/\kmH)$ is a neighbor of $\alpha$ and $\langle\rho(D),\sigma\rangle > 0$ then $\sigma\in S^a(\kmG/\kmH)$ and $D\in\A(\sigma)$.
\end{theorem}

The rather long proof of the theorem is based on a reduction to the case where $\kmG/\kmH$ has rank $2$. We discuss this case first, in the two following lemmas.

\begin{lemma}\label{lemma:values1}
In the hypotheses of Theorem~\ref{thm:values}, assume in addition that $\kmG/\kmH$ has rank $2$, that $D\in\Div(P/H)^B$, and that $H$ is not contained in any proper parabolic subgroup of $P$. Then Theorem~\ref{thm:values} holds.
\end{lemma}
\begin{proof}
The proof of this lemma is similar to the one of \cite[Proposition~6.5]{Kn14}. We may assume that $\sigma\neq\alpha$, which implies that $\Sigma(\kmG/\kmH)=\{\alpha,\sigma\}$ because $\kmG/\kmH$ has rank $2$.

If $\sigma$ is a simple root then it moves $D$, because otherwise we would have $\langle\rho(D),\sigma\rangle \leq 0$ thanks to Proposition~\ref{prop:atmostone}, contradicting our assumptions. Being $\sigma$ a simple root that moves $D$, Theorem~\ref{thm:values} holds.

Therefore we assume that $\sigma$ is not a simple root.

Denote by $E$ the other color moved by $\alpha$ and suppose that $D$ or $E$ is also moved by another simple root $\beta$. By Proposition~\ref{prop:commoncolor} the simple root $\beta$ is a spherical root, so $\beta=\sigma$, contradicting the assumption that $\sigma$ is not a simple root.

Therefore we assume also that $\alpha$ is the only simple root moving $D$ or $E$. Our goal is to show that our assumptions now lead to a contradiction.

At this point $\alpha$ is the only simple root of $\kmG$ moving two colors. Since $\rho(D)$ is positive both on $\alpha$ and on $\sigma$, the convex cone generated by $\rho(D)$ together with $V(P/H)$ is the whole $N(P/H)$. In other words, thanks to Remark~\ref{rem:Pequiv}, the vector space $N(P/H)$ is the first component of a $P$-equivariant colored subspace, equal to either $(N(P/H),\varnothing)$ or $(N(P/H),\{D\})$ depending on whether $D$ is also $L$-invariant or not. By Theorem~\ref{thm:morphisms} this colored subspace is associated with a $P$-equivariant morphism $\varphi\colon P/H\to P/Q$. Moreover the $L$-spherical variety $P/Q$ has rank $0$, i.e.\ it is a complete $L$-homogeneous space and $Q$ is a parabolic subgroup of $P$.

We have assumed that $H$ is not contained in any proper parabolic subgroup of $P$, so $Q=P$ and the map $\varphi$ is the map to a single point. As a consequence, all $L$-colors of $P/H$ are mapped dominantly via $\varphi$.

Suppose that $D$ is not $L$-stable, i.e.\ it is an $L$-color. Then $\varphi$ has colored subspace $(N(P/H),\{D\})$, which says that $D$ is the unique $L$-color of $P/H$. On the other hand $\alpha$ is the unique simple root moving it, therefore $\alpha\in S_\kmP$. Then $\alpha\in\Sigma_L(P/H)$, which implies that $\alpha$ moves two $L$-colors of $P/H$: contradiction.

Finally, it remains to discuss the case where $D$ is $L$-stable. Then $P/H$ has no $L$-colors, so $S^p(\kmG/\kmH)=S_\kmP = S^p(\kmG/\kmP)$.

Consider a complete $P$-toroidal embedding $X$ of $P/H$, let $X'\subset X$ be a localization of $X$ at $\{\sigma\}$, with open $P$-orbit $P/K$, and let $\kmK$ be the inverse image of $K$ in $\kmP$. By Theorem~\ref{thm:locSigma}, we have $S^p(\kmG/\kmK)=S^p(\kmG/\kmH)=S_\kmP$. This equality together with $S^a(\kmG/\kmK)=\varnothing$, which is a consequence of $\Sigma(\kmG/\kmK)=\{\sigma\}$, implies that $\Div(P/K)^B=\varnothing$.

Corollary~\ref{cor:PumovesD} yields that $K\supseteq P^u$, so $\kmG/\kmK$ is horospherical and cannot have spherical roots.  We obtain the desired contradiction, and the proof is complete.

\end{proof}

\begin{lemma}\label{lemma:values2}
In the hypotheses of Theorem~\ref{thm:values}, assume in addition that $\kmG/\kmH$ has rank $2$, that $D\notin\Div(P/H)^B$, and that $H$ is not contained in any proper parabolic subgroup of $P$. Then Theorem~\ref{thm:values} holds.
\end{lemma}
\begin{proof}
As in the proof of Lemma~\ref{lemma:values1}, we may assume that $\sigma\neq\alpha$, so $\Sigma(\kmG/\kmH)=\{\alpha,\sigma\}$.

Assume that another simple root $\beta\neq\alpha$ moves $D$; then $\beta$ is a spherical root by Proposition~\ref{prop:commoncolor}, hence $\beta=\sigma$ and Theorem~\ref{thm:values} holds. If another simple root $\beta\neq\alpha$ moves the other color $E$ moved by $\alpha$, then again $\beta=\sigma$ and $\langle \rho(E),\sigma\rangle = 1$. This yields
\begin{equation}\label{eq:alphaveesigma}
\langle\alpha^\vee, \beta\rangle= \langle\alpha^\vee, \sigma\rangle = \langle \rho(D),\sigma\rangle + \langle\rho(E),\sigma\rangle >0
\end{equation}
which contradicts $\beta\neq\alpha$.

Hence we may assume that $\alpha$ is the only simple root moving $D$ or $E$.

If $\langle\rho(E),\sigma\rangle > 0$, then $\sigma$ is a simple root by Lemma~\ref{lemma:values1} applied to the color $E$, and $\sigma$ moves $E$ contradicting our assumptions.

Therefore we may also suppose that $\langle\rho(E),\sigma\rangle \leq 0$.

Our goal now is to show that all these assumptions lead to a contradiction, which will eventually emerge from a detailed analysis of $H^u$.

Since $D\notin\Div(P/H)^B$ we have $E\in\Div(P/H)^B$. No simple root different from $\alpha$ moves $E$, hence we even have $E\in \Div(P/H)^L$.

Suppose that $\sigma$ is a simple root in $S\smallsetminus S_\kmP$. Then $\Sigma_L(P/H)=\varnothing$, and by Proposition~\ref{prop:commoncolor} the colors moved by $\sigma$ are not moved by any other simple roots (any such simple root would be a spherical root, so equal to $\alpha$). In particular $\sigma$ moves an element $E'\in \Div(P/H)^L$, and we even have $\{E,E'\}=\Div(P/H)^L$ since no simple root other than $\alpha$ and $\sigma$ is a spherical root.

Reasoning as in the proof of Lemma~\ref{lemma:affine} we conclude that $P/H$ is the product of a torus $L/K$ and a spherical $L$-module $M$ (where $K$ is a reductive subgroup of $L$). In addition, the module $M$ has rank $\leq2$, we have $\Sigma_{L}(M)=\varnothing$, and $M$ has exactly two $L$-stable prime divisors. By Lemmas~\ref{lemma:modules} and~\ref{lemma:rank1}, this is only possible if $L=K$, the action of $(L,L)$ on $M$ is trivial and $M$ is the direct sum  $M_1\oplus M_2$ of two one-dimensional $L$-modules.

From Example~\ref{ex:toric} we have that the $L$-weights of $M_1^*$ and $M_2^*$ are the spherical roots $\alpha$ and $\sigma$, and $\langle\rho(E),\sigma\rangle = 0$ since the simple root $\sigma$ moves $E'$ and not $E$. This contradicts the fact that
\[
\langle\rho(E),\sigma\rangle = \underbrace{\langle\alpha^\vee,\sigma\rangle}_{\leq0} - \underbrace{\langle \rho(D),\sigma\rangle}_{>0} <0,
\]
which holds because here $\sigma$ is a simple root different from $\alpha$.

This settles the case $\sigma\in S\smallsetminus S_\kmP$, so assume now the contrary: $\sigma\notin S\smallsetminus S_\kmP$. Here $E$ is the unique $L$-stable prime divisor of $P/H$. It is also the unique element of $\Div(P/H)^B$ such that $\langle\rho(E),\alpha \rangle >0$ (this value is not only positive, but actually $1$). The variety $P/H$ admits a $P$-equivariant morphism to a point, let us denote it by $\psi\colon P/H\to P/P$, with associated colored subspace equal to $(N(X),\Delta_L(P/H))$. This implies that the whole space $N(X)$ is generated as a convex cone by $V(X)$ together with $\rho(E)$ and $\rho(\Delta_L(P/H))$, by the definition of colored subspace recalled in Section~\ref{s:mor}.

Consider now the element
\[
\gamma = \sigma - \langle\rho(E),\sigma\rangle\alpha.
\]
Since $\langle\rho(E),\sigma\rangle\leq 0$, the above linear combination of $\sigma$ and $\alpha$ has non-negative coefficients, hence $\gamma$ is $\leq0$ on $V(X)$. On $\rho(E)$ it is zero, indeed:
\[
\langle\rho(E),\gamma\rangle = \langle\rho(E),\sigma\rangle - \langle\rho(E),\sigma\rangle\langle\rho(E),\alpha\rangle = 0,
\]
so there must be an element $F\in\Delta_L(P/H)$ where $\gamma$ is positive. Then it is elementary to show that $N(X)$ is also generated as a convex cone by $V(X)$ together with $\rho(E)$ and the only $\rho(F)$. In other words also $(N(X),\{F\})$ is a $P$-equivariant colored subspace. Denote by $\psi'\colon P/H\to P/Q$ the associated morphism: then $P/Q$ is an $L$-spherical variety of rank $0$, i.e.\ a complete $L$-homogeneous space. So $Q$ is a parabolic subgroup of $P$ containing $H$, whence $Q=P$, $\psi'=\psi$ and the colored subspaces of $\psi'$ and $\psi$ are equal, which yields $\Delta_L(P/H)=\{F\}$.

To summarize, the variety $P/H$ has exactly one $B$-stable prime divisor different from $E$, namely $F$.

Let us assume in addition that $\Sigma_L(P/H)=\varnothing$. By Lemma~\ref{lemma:affine} the group $H$ contains a Levi subgroup of $P$ (we may assume it is $L$) and the homogeneous space $P/H\cong P^u/H^u$ occurs in Example~\ref{ex:CpCn} as the case of the spherical $L$-module $M\oplus\CC$ with $M$ irreducible of dimension $>1$. Thanks to Example~\ref{ex:CpCn}, one of the spherical roots of $P/H$ is the highest weight of $M^*$, and it takes value $0$ on $\rho(E)$ and $1$ on $\rho(F)$. Only $\sigma$ can be this spherical root, in particular
\begin{equation}\label{eqn:positive}
\langle\alpha^\vee,\sigma\rangle = \langle \rho(D),\sigma\rangle + \underbrace{\langle\rho(E),\sigma\rangle}_{=0} >0.
\end{equation}

Example~\ref{ex:CpCn} also shows that there is a normal subgroup $R$ of $P$, contained in $P^u$ and containing $H^u$, such that $P^u/R \cong M$. Then there must exist $\gamma\in S\smallsetminus S_\kmP$ such that $U_{-\gamma}$ is not in $R$, otherwise $R$ would be equal to $P^u$. Then a non-trivial element of the projection of $U_{-\gamma}$ in $P^u/R$ is a highest weight vector of $P^u/R$. In other words $-\gamma$ is the highest weight of $M$, and the lowest weight of $M$ is an element of $-\gamma + \Span_{\ZZ_{\leq 0}}S_\kmP$. So the highest weight $\sigma$ of $M^*$ is in $\gamma+\Span_{\ZZ_{\geq 0}}S_\kmP$. It follows $\alpha=\gamma$, otherwise $\langle\alpha^\vee,\sigma\rangle \leq 0$, contradicting (\ref{eqn:positive}). Therefore $-\alpha$ and $-\sigma$ are resp.\ the highest and the lowest weights of $M$.

Now consider the summand $\CC$ of $P^u/H^u$. It is the image of $U_{-\beta}\subseteq P^u$ for some positive root $\beta\neq\alpha$, which is then in the lattice of $P/H$ and takes value $1$ on $\rho(E)$ and $0$ on $\rho(F)$. These values, together with $\langle\rho(F),\alpha\rangle\leq0$ and $\alpha\neq\beta$, imply that
\[
\alpha = \beta+\underbrace{\langle \rho(F),\alpha\rangle}_{<0} \sigma.
\]
Let $w_0^L$ be the longest element of the Weyl group of $L$. Using the fact that $-w_0^L(-\alpha)=\sigma$ and the fact that $-w_0^L(\beta)=-\beta$ (because $(L,L)$ acts trivially on $\CC$) it is elementary to deduce that $\langle \rho(F),\alpha\rangle$ is actually $-1$, i.e.\
\begin{equation}\label{eqn:roots}
\alpha +\sigma = \beta.
\end{equation}
After Lemma~\ref{lemma:affine} the module $M$ is isomorphic to $\CC^n$ under the action of $(L,L)\cong \SL(n)$ or $\Sp(n)$ (with $n$ even). From (\ref{eqn:roots}) we see that the sum of the highest and the lowest weight of $M$ is trivial when restricted to $T\cap (L,L)$. This implies that $n$ is even and $(L,L)\cong \Sp(n)$.

Because of $\Sigma_L(P/H)=\varnothing$ and Theorem~\ref{thm:NS}, the spherical root $\sigma$ is a root of $\kmG$. Consider the $\alpha$-string of roots $\{\sigma+k\alpha\;|\; -p\leq k\leq q\}$ through $\sigma$, where $p,q$ are integers such that $p-q=\langle\alpha^\vee,\sigma\rangle$. From (\ref{eqn:positive}) and the fact that $\alpha+\sigma$ is a root we deduce $p>q\geq 1$, i.e.\ that $\sigma-\alpha$ and $\sigma-2\alpha$ are (positive) roots too. Now $\sigma-\alpha$ is actually a root of $L$, this is shown easily computing the highest weight of $\CC^n$ minus its lowest weight under the action of any reductive group whose commutator is $\Sp(n)$. Then $\alpha$ doesn't appear in the expression of $\sigma-\alpha$ as a linear combination of simple roots, and so $\sigma-2\alpha$ cannot be a root. We obtain the desired contradiction, and the case $\Sigma_L(P/H)=\varnothing$ is completed.

We examine the only other possibility for $\Sigma_L(P/H)$, namely $\Sigma_L(P/H)=\{\sigma\}$. In this case $\langle\alpha^\vee,\sigma\rangle \leq 0$ (because $\sigma\in\Span_{\QQ_{\geq0}}S_\kmP$), so $\langle\rho(D),\sigma\rangle >0$ implies $\langle\rho(E),\sigma\rangle <0$. This excludes the possibility that $\langle\rho(E),\sigma\rangle =0$, so by Lemma~\ref{lemma:affine} the group $H$ contains a Levi subgroup of $P$ (we may suppose it is $L$), and $P/H\cong P^u/H^u$ is $L$-equivariantly isomorphic to an irreducible spherical $L$-module $M$ of rank $2$ and dimension $>1$. Then $H^u$ is a normal subgroup of $P^u$, otherwise its normalizer would correspond to a non-trivial and proper $L$-submodule of $M$. As above, there exists a simple root $\beta\in S\smallsetminus S_\kmP$ such that $U_{-\beta}$ is not contained in $H^u$, and $-\beta$ is the highest weight of $M$. The $T$-weights of $\CC[M]$ are now in the set $\Span_\ZZ\{\beta, S_\kmP\}$, so the latter contains also $\Xi_B(P/H)$. But $\Span_\ZZ\{\beta, S_\kmP\}$ doesn't contain $\alpha$ unless $\beta=\alpha$.

In other words we have that $-\alpha$ is the highest weight of $M$. Since $M$ has dimension $>1$, there is a simple root $\alpha'\in S_\kmP$ such that $\langle(\alpha')^\vee,-\alpha\rangle \neq 0$, which implies
\begin{equation}\label{eqn:negative}
\langle\alpha^\vee,\alpha'\rangle < 0.
\end{equation}
Notice that $\alpha'$ is a simple root of $\widetilde L$, where $\widetilde L$ is the product of all simple normal subgroups of $L$ acting non-trivially on $M$.

We refer once again to Lemma~\ref{lemma:modules}, which yields $\langle\rho(E),\sigma\rangle=-1$. The same lemma also implies that $\sigma$ is a linear combination with positive coefficients of all simple roots of $\widetilde L$. From this and (\ref{eqn:negative}) we deduce that $\langle\alpha^\vee,\sigma\rangle <0$, whence
\[
\langle\rho(D),\sigma\rangle = \underbrace{\langle\alpha^\vee,\sigma\rangle}_{<0} - \underbrace{\langle\rho(E),\sigma\rangle}_{=-1} \leq 0
\]
which contradicts our assumption $\langle\rho(D),\sigma\rangle >0$. The case $\Sigma_L(P/H)=\{\sigma\}$ is completed, and with it the proof of the lemma.
\end{proof}

\begin{proof}[Proof of Theorem~\ref{thm:values}]
Let $X$ be a complete $P$-toroidal embedding of $P/H$. Replacing if necessary $X$ with a suitable localization at $\{ \sigma, \alpha\}$, we may assume that $\Sigma(\kmG/\kmH)=\{\sigma,\alpha\}$ and that $X$ has rank $2$. In addition, by making $\kmP$ smaller and conjugating $H$ in $P$ if necessary, we may assume that no proper parabolic subgroup of $P$ contains $H$.

At this point, if $D\in \Div(P/H)^B$ then the theorem follows from Lemma~\ref{lemma:values1}, otherwise from Lemma~\ref{lemma:values2}.
\end{proof}

\begin{corollary}\label{cor:values}
The last statement of Proposition~\ref{prop:atmostone} holds for any $E\in\Delta(\kmG/\kmH)$ not moved by $\alpha$.
\end{corollary}
\begin{proof}
This stems from the fact that two spherical roots that are simple roots are neighbors.
\end{proof}

\begin{corollary}\label{cor:valuesnotneighbors}
Theorem~\ref{thm:values} holds even if $\sigma$ is not assumed to be a neighbor of $\alpha$.
\end{corollary}
\begin{proof}
Since $V(\kmG/\kmH)$ has maximal dimension in $N(\kmG/\kmH)$, the cone $C=\Span_{\RR_{\geq 0}}\Sigma$ (inside $\Xi_\RR=\RR\otimes_\ZZ\Xi(\kmG/\kmH)$) is strictly convex. Consider now the vector subspace $M=\Span_\RR C$ of $\Xi_\RR$, and extend $\RR$-linearly the functional $\rho(D)$ to $M$. Since $C$ is strictly convex, there exists an affine hyperplane $M'$ of $M$ such that $A=M'\cap C$ is a convex polytope (of maximal dimension in $M'$) whose vertices are the spherical roots of $\kmG/\kmH$, up to positive multiples. Denote by $A_+$ the union of all proper faces of $A$ where $\rho(D)$ is non-negative. Since $\alpha$ and $\sigma$ are both positive on $\rho(D)$, the polytope $A_+$ is not empty.

The main observation is that $A_+$ is connected. This holds if $\rho(D)$ is non-negative on all vertices of $A$, because then $A_+$ is the whole boundary $\partial A$, and it is also true if $\rho(D)$ is negative on some vertex of $A$, because in this case $A_+$ is a deformation retract of the intersection $\partial A\cap \{\langle\rho(D),-\rangle \geq 0\}$, which is itself homeomorphic to a disc of dimension equal to the dimension of $M'$ minus $1$. Take now the two vertices of $A$ corresponding resp.\ to $\sigma$ and $\alpha$. They are contained in $A_+$, hence they are connected by a chain of vertices of $A_+$ such that any two successive elements of the chain are adjacent vertices of $A$, i.e.\ the corresponding spherical roots are neighbors. The corollary is now a direct consequence of Theorem~\ref{thm:values}.
\end{proof}

\section{Spherical systems and Luna's axioms}\label{s:hsd}

We come to the main definition of the paper. It generalizes to the Kac-Moody case the one given in \cite[Section~2.1]{Lu01}. 

\begin{definition}
The {\em homogeneous spherical datum} of $\kmG/\kmH$ is defined as the quintuple $(S^p(\kmG/\kmH), \Sigma(\kmG/\kmH), \A(\kmG/\kmH), \Xi(\kmG/\kmH), \rho_{\kmG/\kmH}|_{\A(\kmG/\kmH)})$. We also define the {\em spherical system} of $\kmG/\kmH$ as the quadruple $(S^p(\kmG/\kmH), \Sigma(\kmG/\kmH), \A(\kmG/\kmH), \rho'_{\kmG/\kmH}\colon \A(\kmG/\kmH)\to\Hom_\ZZ(\Span_\ZZ\Sigma(\kmG/\kmH),\QQ))$, where $\rho'_{\kmG/\kmH}$ is induced by $\rho_{\kmG/\kmH}$.
\end{definition}

The homogeneous spherical datum of $\kmG/\kmH$ satisfies the same combinatorial properties observed in the finite-dimensional case in \cite{Lu01}. Following loc.cit., we first collect these properties in the separate, combinatorial Definition~\ref{def:hsd} below, and then we prove them for the homogeneous spherical datum of $\kmG/\kmH$ in Theorem~\ref{thm:consistent}.

\begin{definition}
An element $\sigma\in\Chars(T)$ is a {\em spherical root for $\kmG$} if there exists a standard spherical subgroup of finite type $\kmK$ of $\kmG$ with $\sigma\in\Sigma(\kmG/\kmK)$. The set of spherical roots for $\kmG$ is denoted by $\Sigma(\kmG)$. A spherical root $\sigma$ for $\kmG$ is {\em compatible} with a subset $S^p\subseteq S$ if there exists a standard spherical subgroup of finite type $\kmK$ of $\kmG$ with $S^p(\kmG/\kmK)=S^p$ and $\sigma\in\Sigma(\kmG/\kmK)$.
\end{definition}

Thanks to Remark~\ref{rem:locSigmaexist} and Theorem~\ref{thm:locSigma}, a spherical root $\sigma$ is compatible with $S^p$ if and only if there exists a spherical subgroup of finite type $\kmK$ of $\kmG$ with $S^p(\kmG/\kmK)=S^p$ and $\{\sigma\}=\Sigma(\kmG/\kmK)$.

\begin{definition}\label{def:hsd}
An {\em (abstract) homogeneous spherical datum for} $\kmG$ is a quintuple $(S^p, \Sigma, \A, \Xi, \rho)$, where $S^p$ is a subset of $S$, $\Sigma$ is finite subset of primitive elements of $\Xi$, which is a subgroup of $\Chars(T)$, $\A$ is a finite set, $\rho$ is a map $\rho\colon \A\to\Hom_\ZZ(\Xi,\ZZ)$, such that the following properties are satisfied:
\begin{itemize}
\item[(A1)] For all $\delta\in \A$ and $\sigma\in \Sigma$ we have $\delta(\sigma)\leq1$, and $\delta(\sigma)=1$ implies $\gamma\in\Sigma\cap S$.
\item[(A2)] For all $\alpha\in\Sigma\cap S$, the set $\A(\alpha) = \{\delta\in \A\;|\; \delta(\alpha)=1\}$ contains exactly two elements $\delta_\alpha^+$, $\delta_\alpha^-$, and they satisfy $\rho(\delta_\alpha^+)+\rho(\delta_\alpha^-) = \alpha^\vee|_{\Xi}$.
\item[(A3)] The set $\A$ is the union of all $\A(\alpha)$ for $\alpha\in\Sigma\cap S$.
\item[($\Sigma$1)] If $2\alpha\in\Sigma\cap 2S$ then $\langle \alpha^\vee, \Xi\rangle \subseteq 2\ZZ$.
\item[($\Sigma$2)] If $\alpha$ and $\beta$ are orthogonal simple roots such that $\alpha+\beta$ or $\frac12(\alpha+\beta)$ is in $\Sigma$, then $\alpha^\vee|_\Xi=\beta^\vee|_\Xi$.
\item[(S)] Any $\sigma\in \Sigma$ is compatible with $S^p$. 
\end{itemize}
An {\em (abstract) spherical system for} $\kmG$ is a quadruple $(S^p, \Sigma, \A, \rho')$, where $S^p$ is a subset of $S$, $\Sigma$ is finite subset of $\Chars(T)$, $\A$ is a finite set equipped with a map $\rho'\colon \A\to\Hom_\ZZ(\Span_\ZZ\Sigma,\ZZ)$, such that $(S^p, \Sigma,\A,\Span_\ZZ\Sigma, \rho')$ is an abstract homogeneous spherical datum.
\end{definition}

While being obviously satisfied in the finite-dimensional case, the property that $\kmH\subseteq\kmP$ for some parabolic subgroup $\kmP\subseteq\kmG$ of finite type lies at the foundation of our theory. It has a combinatorial counterpart which we give in the following.

\begin{definition}\label{def:ftsys}
An abstract homogeneous spherical datum or an abstract spherical system for $\kmG$ is of {\em finite type} if there exist a subset $A_1$ of $\A$, a subset $S_2$ of $S\smallsetminus (\Sigma\cup S^p)$, and an element
\[
\eta\in \Span_{\QQ_{>0}} \left(\rho(A_1)\cup\left\{\alpha^\vee|_\Xi\;\middle\vert\; \alpha\in S_2\right\}\right)
\]
such that $\langle \eta, \sigma \rangle >0$ for all $\sigma\in \Sigma$, and such that $S_1\cup S_2\cup S^p$ is a subset of $S$ of finite type, where
\begin{equation}\label{eqn:S1}
S_1=\left\{ \alpha\in \Sigma\cap S\;\middle\vert\; \A(\alpha)\subseteq A_1 \right\}.
\end{equation}
\end{definition}

Finally, in the following theorem we collect our main results and check that the objects associated with $\kmH$ indeed satisfy Luna's axioms.

\begin{theorem}\label{thm:consistent}
The homogeneous spherical datum (resp.\ spherical system) of $\kmG/\kmH$ is an abstract homogeneous spherical datum (resp.\ spherical system) of finite type for $\kmG$.
\end{theorem}
\begin{proof}
The assertion on the spherical system follows from the one on the homogeneous spherical datum, because the required conditions on values of functionals on $\Span_\ZZ\Sigma(\kmG/\kmH)$ are implied by the same conditions holding on the bigger lattice $\Xi$.

We prove the assertion on the homogeneous spherical datum. Axioms (A1) and (A2) follow from Propositions~\ref{prop:atmostone} and~\ref{prop:typeofS} and Corollaries~\ref{cor:values} and~\ref{cor:valuesnotneighbors}, axiom (A3) holds by definition of $\A(\kmG/\kmH)$, axiom ($\Sigma$1) stems from Proposition~\ref{prop:typeofS}, axiom ($\Sigma$2) from Propositions~\ref{prop:typeofS} and~\ref{prop:commoncolor}, and axiom (S) is obvious.

To show that the homogeneous spherical datum of $\kmG/\kmH$ is of finite type, consider the $P$-equivariant morphism $\varphi\colon P/H\to P/P$. The associated $P$-equivariant colored subspace $(\Cone_\varphi, \Colemb_\varphi)$ satisfies $\Cone_\varphi = N(P/H)$, so the latter is generated as a convex cone by $\rho(\Div(P/H)^B)\cup V(P/H)$. In other words the cone generated by $\rho(\Div(P/H)^B)$ meets the relative interior of $-V(P/H)$. Choose $\eta$ in the intersection. Now define $A_1 = \Div(P/H)^B\cap \A(\kmG/\kmH)$ and $S_2$ to be the set of simple roots moving some element of $\Div(P/H)^B\smallsetminus \A(\kmG/\kmH)$.

For each $\alpha\in S_2$, the color $D_\alpha\in \Div(P/H)^B\smallsetminus \A(\kmG/\kmH)$ moved by $\alpha$ satisfies $\rho(D_\alpha)\in \QQ_{>0}\alpha^\vee|_\Xi$. Then $\eta$ satisfies the requirements of the theorem, up to replacing if necessary $A_1$ and $S_2$ with smaller subsets (i.e.\ those of the elements actually involved in the expression of $\eta$ as a linear combination of elements of $\rho(\Div(P/H)^B)$). The set $S_1$ defined as in (\ref{eqn:S1}) is a subset of $S_\kmP$, because $A_1\subseteq \Div(P/H)^B$ and if $\A(\alpha)$ contains only elements in $\Div(P/H)^B$ then $\alpha$ is in $S_\kmP$. Also $S_2$ is a subset of $S_\kmP$: indeed, if a simple root $\alpha\notin S_\kmP$ moves an element of $\Div(P/H)^B$ then it is a spherical root and the colors it moves are both in $\A(\kmG/\kmH)$. We conclude that $S_1\cup S_2\cup S^p$ is a subset of $S_\kmP$, therefore it is of finite type and the proof is complete.
\end{proof}

\section{Invariance under conjugation}\label{s:conjugation}

In this section, for any element $w\in W$ we denote by $\dot{w}$ a fixed representative of $w$ in $\kmG$.

\begin{lemma}\label{lemma:conj}
Let $R$ be a very reductive subgroup of $L$ and $w\in W$, and suppose that ${}^{\dot{w}}R$ is a very reductive subgroup of the standard Levi subgroup $M$ of $\kmQ$, for some parabolic subgroup $\kmQ\subseteq\kmG$ of finite type containing $\kmB_-$. Then $M={}^{\dot{w}} L$.
\end{lemma}
\begin{proof}
The group ${}^{\dot{w}} L$ is reductive and contains $T$, and its subgroup $N={}^{\dot{w}} L\cap \kmQ$ is parabolic. Since $N$ contains the very reductive subgroup ${}^{\dot{w}}R$ of ${}^{\dot{w}} L$, it follows that $N={}^{\dot{w}} L$. In other words ${}^{\dot{w}} L\subseteq\kmQ$, so ${}^{\dot{w}} L\subseteq M$. 

Hence ${}^{\dot{w}}R \subseteq {}^{\dot{w}}L \subseteq M$. Since ${}^{\dot{w}}R$ is very reductive in $M$ and ${}^{\dot{w}}L$ is a reductive subgroup of $M$, we have ${}^{\dot{w}} L = M$.
\end{proof}

In the next theorem we show that the homogeneous spherical datum is invariant under conjugation of $\kmH$ in $\kmG$. Here, for brevity, equality of two homogeneous spherical data is meant as in Theorem~\ref{thm:indep}. 

\begin{theorem}\label{thm:conjugation}
Let $g\in\kmG$ and suppose that ${}^g\kmH$ is a standard spherical subgroup of finite type of $\kmG$, thus contained in some parabolic $\kmQ$ of finite type containing $\kmB_-$. Then the homogeneous spherical datum of $\kmG/{}^g\kmH$ (defined with respect to $\kmQ$) is equal to the homogeneous spherical datum of $\kmG/\kmH$ (defined with respect to $\kmP$).
\end{theorem}
\begin{proof}
Denote by $S_\kmQ$ the set of simple roots associated with $\kmQ$, and set $\kmK={}^g\kmH$. Up to conjugating $\kmH$ inside $\kmP$ and $\kmK$ inside $\kmQ$, and up to replacing $\kmP$ and $\kmQ$ with smaller parabolic subgroups (thanks to Theorem~\ref{thm:indep}), we can assume that $H$ has a Levi subgroup $L_H$ that is a very reductive subgroup of $L$, and that ${}^gL_H$ is a very reductive subgroup of the standard Levi subgroup $M$ of $\kmQ$. By the Bruhat decomposition, and since both $\kmP$ and $\kmQ$ contain $\kmB_-$, we can also assume that $g={\dot{w}}$ for some $w\in W$. Lemma~\ref{lemma:conj} implies that ${}^gL=M$, and up to changing the element $w$ we can also assume that ${}^{\dot{w}}B$ is the Borel subgroup $B_M=M\cap\kmB$ of $M$. This is equivalent to $w(S_\kmP)=S_\kmQ$.

Write
\begin{equation}\label{eqn:length}
w = s_1\cdots s_m
\end{equation}
where $s_k$ is the simple reflection corresponding to the simple root $\alpha_k$ for all $k\in\{1,\ldots,m\}$, and $m$ is the length of $w$. We proceed by induction on $m$, the base of the induction being the obvious case $m=0$. We underline that the condition $w(S_\kmP)=S_\kmQ$ is now part of our assumptions and must be taken into consideration while proving the induction step.

Since the expression (\ref{eqn:length}) is reduced, the root $w(\alpha_m)$ is negative. Then $\alpha_m\in S\smallsetminus S_\kmP$, otherwise $w(\alpha_m)$ would be positive because of $w(S_\kmP)=S_\kmQ$. Notice that this equality also implies that $w(\alpha_m)$ is not a root of $M$.

Let now $w_0^M$ be the longest element of the Weyl group of $M$. We have $w_0^Mw(S_\kmP)=-S_\kmQ$, and $w_0^Mw(\alpha_m)$ is negative because $w_0^M$ permutes the roots of $\kmQ$ that are not roots of $M$. It follows that the set of simple roots $S_0=S_\kmP\cup\{\alpha_m\}$ is of finite type, otherwise $w_0^Mw$ would sends infinitely many positive roots to negative ones. We denote by $\kmP_0$ the corresponding parabolic subgroup containing $\kmB_-$ of $\kmG$, and by $L_0$ the standard Levi subgroup of $\kmP_0$.

Let $w_0^L$ and $w_0^{L_0}$ be the longest elements of the Weyl groups of resp.\ $L$ and $L_0$, and consider the subgroup $\kmH_1={}^{\dot{w}_0^{L_0}\dot{w}_0^L}\kmH$. It is conjugated to $\kmK$ by a representative of the element $ww_0^Lw_0^{L_0}\in W$, and $\kmG/\kmH_1$ has the same homogeneous spherical datum of $\kmG/\kmH$ by Theorem~\ref{thm:indep} since $\kmH$ and $\kmH_1$ are conjugated in $\kmP_0$.

Our goal is to apply the induction hypothesis to the groups $\kmH_1$ and $\kmK$ and the element $ww_0^Lw_0^{L_0}\in W$, so we must check all required conditions that we had set up:
\begin{enumerate}
\item\label{proof:conjugation:contained} $\kmH_1$ is contained in a negative parabolic subgroup $\kmP_1$ of $\kmG$ of finite type (denote by $H_1\subseteq P_1$ corresponding finite dimensional quotients as usual),
\item\label{proof:conjugation:Levi} $H_1$ has a Levi subgroup that is a very reductive subgroup of the Levi subgroup $L_1$ of $P_1$ containing $T$,
\item \label{proof:conjugation:roots}$ww_0^Lw_0^{L_0}$ sends $S_{\kmP_1}$ onto $S_\kmQ$,
\item \label{proof:conjugation:length}$ww_0^Lw_0^{L_0}$ is shorter than $w$.
\end{enumerate}

The set $S_\kmP$ is a subset of the simple roots $S_0$ of $L_0$, hence $S_1=-w_0^{L_0}(S_\kmP)$ is also a subset of $S_0$. Denote by $\kmP_1$ the parabolic subgroup of finite type of $\kmG$ (contained in $\kmP_0$) associated with the set of simple roots $S_1$ and containing $\kmB_-$. We have that $L_1={}^{\dot{w}_0^{L_0}}L$ is the standard Levi subgroup of $\kmP_1$.

Set $\widetilde \kmQ=\dot{w}^{-1}\kmQ$. The intersection $\kmP\cap\widetilde\kmQ$ contains $\kmH$ and $L$, and is contained in $\kmP_0$. The intersection $\kmP\cap\widetilde\kmQ\cap L_0$ has Levi subgroup $L$, and its unipotent radical is either trivial, or has roots that are linear combinations of simple roots with non-positive coefficients and involving $\alpha_m$. Suppose the unipotent radical of $\kmP\cap\widetilde\kmQ\cap L_0$ is not trivial, and let $\alpha$ be one if its roots. Then $w(\alpha)$ is a root of $\kmQ$, not a root of $M$, so $w(\alpha)$ is a linear combination with non-positive coefficients of $S_\kmQ=w(S_\kmP)$ and $w(\alpha_m)$ (the latter with negative coefficient). This contradicts the fact that $w(\alpha_m)$ is itself negative and not in the $\ZZ$-span of $S_\kmQ$.

The consequence is that $\kmP\cap\widetilde\kmQ\cap L_0=L$, which is then also the projection of $\kmP\cap\widetilde\kmQ$ in $L_0$ along the quotient by $\kmP_0^u$. This implies that the projection of $\kmH\subset \kmP_0$ in $L_0$ is contained in (and a very reductive subgroup of) $L$. In turn, this implies that the projection of $\kmH_1\subset \kmP_0$ in $L_0$ is contained in (and a very reductive subgroup of) $L_1$.

Conditions (\ref{proof:conjugation:contained}) and (\ref{proof:conjugation:Levi}) follow, and condition (\ref{proof:conjugation:roots}) is clear by construction.

Finally, we count how many negative roots change sign under the action of $w$ and of $ww_0^Lw_0^{L_0}$. The set of negative roots of $\kmG$ is equal to two disjoint unions $R_1\cup R_2 \cup R_3=R_1'\cup R_2' \cup R_3$, where $R_1$ (resp.\ $R_1'$) is the set of negative roots of $L_1$ (resp.\ $L$), $R_2$ (resp.\ $R_2'$) is the set of negative roots of $L_0$ that are not roots of $L_1$ (resp.\ $L$), and $R_3$ is the set of negative roots that are not roots of $L_0$.

We have $R_1'=w_0^Lw_0^{L_0}(R_1)$, and $w(R_1')=ww_0^Lw_0^{L_0}(R_1)$ is the set of negative roots of $M$. Since $w_0^Lw_0^{L_0}$ permutes the set $R_3$, the element $ww_0^Lw_0^{L_0}$ changes the sign of as many roots of $R_3$ as $w$ does.

It remains to compare the behaviour of the two non-empty sets $R_2$ and $R_2'$. The elements of $w_0^{L_0}(R_2)$ are the positive roots of $L_0$ not of $L$, and these are permuted by $w_0^L$, so $w_0^Lw_0^{L_0}(R_2)=w_0^{L_0}(R_2)=-R_2'$. Since $w(\alpha_m)$ is negative and not in the $\ZZ$-span of $S_\kmQ$, while $w(S_\kmP)=S_\kmQ$, we have that $ww_0^Lw_0^{L_0}(R_2)$ is a set of negative roots, and $w(R_2')$ is a set of positive roots.

Claim (\ref{proof:conjugation:length}) follows, and the proof is complete.
\end{proof}

Theorem~\ref{thm:conjugation} enables us to give the following.

\begin{definition}\label{def:spherical}
Let $\kmH$ be a subgroup of $\kmG$. If it is conjugated to a standard spherical subgroup of finite type, then $\kmH$ is called a {\em spherical subgroup of finite type} of $\kmG$. In this case, the {\em homogeneous spherical datum} of $\kmG/\kmH$ is defined as the homogeneous spherical datum of $\kmG/\kmK$, where $\kmK$ is any standard spherical subgroup of $\kmG$ of finite type conjugated to $\kmK$.
\end{definition}

\section{Further developments}\label{s:conj}

It is natural to pose the question whether homogeneous spherical data are as relevant here as in the classical theory of spherical varieties. 

First of all, in this setting they are not entirely combinatorial objects, because axiom (S) involves the notion of a spherical root compatible with $S^p$. To make this into a purely combinatorial condition one could classify all spherical subgroups of finite type with only one spherical root. It has been done in the finite-dimensional setting (see e.g.\ \cite{Br89} and references therein), and completing this task for $\kmG$ infinite-dimensional seems to be attainable (see Example~\ref{ex:new}). It will be included in a forthcoming paper.

Then we come to the question whether the classification of spherical homogeneous spaces can be extended to our situation.

\begin{conjecture}
Mapping $\kmH\subseteq\kmG$ to the homogeneous spherical datum of $\kmG/\kmH$ induces a bijection between the set of conjugacy classes of spherical subgroups of finite type of $\kmG$ and homogeneous spherical data of finite type of $\kmG$.
\end{conjecture}

Evidence supporting this generalization is still limited to a direct verification for a few cases of $\kmG$, such as $\kmG$ of affine type with Cartan matrix of size $2$ or $3$.

\section{Examples}\label{s:examples}

\begin{example}\label{ex:verysolv}
Let $\kmP=\kmB_-$ and $\kmH=\kmT(\kmU_-,\kmU_-)$. Then $P/H$ is isomorphic to an affine space, where $L=T$ acts with characters equal to the opposites of the simple roots of $\kmG$, and $P^u$ acts as the full group of translations. In other words it occurs in Example~\ref{ex:toric}. Therefore $\Sigma(\kmG/\kmH)=S$, the lattice $\Xi(\kmG/\kmH)$ is equal to the root lattice, the set $\Div(P/H)^B$ contains $n=|S|$ elements $D^+_1,\ldots,D^+_n$, each moved by exactly one simple root $\alpha_i$, and
\[
\langle \rho(D^+_i),\alpha_j\rangle = \delta_{ij}.
\]
The whole set $\Delta(\kmG/\kmH)$ is equal to $\{D^+_1,D^-_1,\ldots,D^+_n,D^-_n\}$, with
\[
\langle \rho(D^-_i),\alpha_j\rangle = \langle\alpha_i^\vee,\alpha_j\rangle - \delta_{ij}.
\]
The corresponding {\em Luna diagram} for $\kmG$ e.g.\ of type $\mathsf A^{(1)}_1$ is the following:
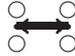
\[
\begin{picture}(2400,1800)(-300,-600)
\put(0,0){\usebox{\leftrightbiedge}}
\multiput(0,0)(1800,0){2}{\usebox{\aone}}
\end{picture}
\]
We refer to \cite[Section~2]{BP15} for details on Luna diagrams.
\end{example}

\begin{example}
With $\kmG$ of type $\mathsf A^{(1)}_1$, let us call $\alpha_0$ and $\alpha_1$ the two simple roots. We modify a bit Example~\ref{ex:verysolv} and consider $\kmH$ to be the subgroup generated by $\kmT$ and $\kmU_{-\alpha}$ for $\alpha$ any positive real root different from $\alpha_0$. In other words $H$ contains $T$, and contains $U_{-\beta}$ if and only if $\beta$ is a positive root different from $\alpha_0$ and from $\delta=\alpha_0+\alpha_1$. We have again $\Sigma(\kmG/\kmH)=S=\{\alpha_0,\alpha_1\}$ and $\Delta(\kmG/\kmH)=\{D^+_0,D^-_0,D^+_1,D^-_1\}$, but the pairing between colors and spherical roots is the following:
\[
\begin{array}{l|cc}
\langle\rho(-),-\rangle & \alpha_0 & \alpha_1 \\
\hline
D_0^{+} &1 &-1 \\
D_0^{-} &1 &-1 \\
D_1^{+} &0 &1 \\
D_1^{-} &-2 &1
\end{array}
\]
and the Luna diagram is
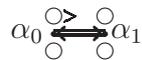
\[
\begin{picture}(2400,1800)(-300,-600)
\put(0,0){\usebox{\leftrightbiedge}}
\multiput(0,0)(1800,0){2}{\usebox{\aone}}
\put(0,600){\usebox{\toe}}
\put(-1500,-150){$\alpha_0$}
\put(2000,-150){$\alpha_1$}
\end{picture}
\]
This $\kmH$ is also contained in $\kmP$ where $S_\kmP = \{ \alpha_0 \}$. The Levi subgroup $L$ has type $\mathsf A_1$, and $H\cap L$ is a maximal torus of $L$. If we call $N$ the normalizer of $H\cap L$ in $L$, then $H^u$ is stable under conjugation by $N$, so $K=NH^u$ is a disconnected subgroup of $P$ with $K^\circ = H$. Call $\kmK$ the inverse image of $K$ in $\kmP$; then $\Sigma(\kmG/\kmK)=\{2\alpha_0,\alpha_1\}$ (with $\Sigma_L(P/K)=\{2\alpha_0\}$) and $\Delta(\kmG/\kmK) = \{D_0,D^+_1,D^-_1\}$ with pairing
\[
\begin{array}{l|cc}
\langle\rho(-),-\rangle& 2\alpha_0 & \alpha_1 \\
\hline
D_0 &2 &-1  \\
D_1^{+} &0 &1 \\
D_1^{-} &-4 &1
\end{array}
\]
and Luna diagram
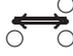
\[
\begin{picture}(2400,1800)(-300,-600)
\put(0,0){\usebox{\leftrightbiedge}}
\put(1800,0){\usebox{\aone}}
\put(0,0){\usebox{\aprime}}
\end{picture}
\]
\end{example}

\begin{example}\label{ex:conj}
We give an example of the situation of Theorem~\ref{thm:conjugation}. Let $\kmG$ be of type $\mathsf A^{(1)}_2$ with simple roots $\alpha_0$, $\alpha_1$ and $\alpha_1$. Define $\Sigma=\{\alpha_0, \alpha_1+\alpha_2\}$ with $\langle \rho(D^+_0),\alpha_1+\alpha_2\rangle = 0$; the corresponding Luna diagram is the following:
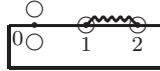
\[
\begin{picture}(5400,1800)(-300,-1200)
\multiput(0,0)(1800,0){2}{\usebox{\edge}}
\put(0,0){\usebox{\aone}}
\put(1800,0){\usebox{\atwo}}
\put(0,0){\thicklines\line(-1,0){900}}
\put(3600,0){\thicklines\line(1,0){900}}
\put(-900,0){\thicklines\line(0,-1){1500}}
\put(4500,0){\thicklines\line(0,-1){1500}}
\put(-900,-1500){\thicklines\line(1,0){5400}}
\put(-800,-700){$\scriptstyle 0$}
\put(1600,-900){$\scriptstyle 1$}
\put(3400,-900){$\scriptstyle 2$}
\end{picture}
\]
Setting $\Xi=\Span_\ZZ\Sigma$, these are the invariants of $\kmG/\kmH$ for a subgroup $\kmH$ in $\kmP$, where $S_\kmP=\{\alpha_1\}$ and $\kmH$ is generated by $T$ together with $\kmU_{\alpha_1}$ and $\kmU_\beta$ for $\beta$ any negative real root different from $-\alpha_2$ and $-\alpha_1-\alpha_2$. If we define other subgroups $\kmK$ and $\kmQ$ of $\kmG$ similarly as above, but exchanging the roles of $\alpha_1$ and $\alpha_2$, we obtain that $\kmH$ and $\kmK$ are conjugated in $\kmG$, and they have the same above invariants.
\end{example}

\begin{example}\label{ex:new}
We give an example of a spherical root that is ``new'', i.e.\ not appearing in the classical theory of finite-dimensional spherical varieties. For this, we choose $\kmH$ in an infinite-dimensional $\kmG$ so that $\Sigma_L(P/H)=\varnothing\neq \Sigma(P/H)$ and $\Rank\kmG/\kmH=1$. These properties are achieved imposing $H\supseteq L$, so that $P/H=P^u/H^u$, and choosing $H^u$ so that $P^u/H^u$ is an irreducible spherical $L$-module of rank $1$. Notice that this is only possible if $L$ has semisimple type $\mathsf A_n$ or $\mathsf C_n$ for some $n$, thanks to Lemma~\ref{lemma:modules}.

For example, let us consider $\kmG$ of type $\mathsf G^{(1)}_2$, with simple roots $\alpha_0,\alpha_1,\alpha_2$ numbered as follows:
\[
\begin{picture}(4200,2000)(-300,-1000)
\put(0,0){\usebox{\edge}}
\put(1800,0){\usebox{\righttriedge}}
\put(-200,-900){$\scriptstyle 0$}
\put(1600,-900){$\scriptstyle 1$}
\put(3400,-900){$\scriptstyle 2$}
\end{picture}
\]
Set $S_\kmP=\{\alpha_0,\alpha_1\}$, consider the corresponding parabolic subgroup $\kmP$, its Levi subgroup $L$, and set $\kmH=L\kmH^u$ where $\kmH^u$ is the group generated by $\kmU_\alpha$ where $\alpha$ is any positive real root with $S_\kmP$-height%
\footnote{Recall that, for $S'\subseteq S$, the $S'$-height of a positive root is defined as the usual height but where only simple roots {\em not in} $S'$ count.}
at least $2$. Then $P/H=P^u/H^u$ is isomorphic to $\CC^3$ and equal to the defining representation of $\SL(3)$ if one chooses an appropriate isomorphism $(L,L)\cong \SL(3)$. The group $P^u$ acts on $P/H$ as the full group of translations. As in Example~\ref{ex:Cn}, we have $\Sigma(P/H)=\{\omega\}$ where $\omega$ is the highest weight of $(\CC^3)^*$. The highest weight of $P^u/H^u$ as an $L$-module is $-\alpha_2$, therefore
\[
\omega=-s_0s_1s_0(\alpha_2) = \alpha_0+\alpha_1+\alpha_2
\]
is the spherical root of $\kmG/\kmH$. 
\end{example}

\begin{example}\label{ex:veryred}
Not all abstract spherical systems are of finite type. For example, the following Luna diagram
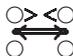
\[
\begin{picture}(2400,1800)(-300,-600)
\put(0,0){\usebox{\leftrightbiedge}}
\multiput(0,0)(1800,0){2}{\usebox{\aone}}
\put(0,600){\usebox{\toe}}
\put(1800,600){\usebox{\tow}}
\end{picture}
\]
corresponds to the spherical system $(S^p,\Sigma,\A)$ given by $S^p=\varnothing$, $\Sigma=S=\{\alpha_0,\alpha_1\}$, $\A=\{D^+_0,D^-_0,D^+_1,D^-_1\}$ with the values
\[
\begin{array}{l|cc}
\langle\rho(-),-\rangle& \alpha_0 & \alpha_1 \\
\hline
D_0^{+} &1 &-1 \\
D_0^{-} &1 &-1 \\
D_1^{+} &-1 &1 \\
D_1^{-} &-1 &1
\end{array}
\]
For this system the sets $A_1$, $S_2$ and the linear combination $\eta$ needed in Definition~\ref{def:ftsys} don't exist, because no linear combination (with non-negative coefficients) of the rows of the above table has only positive entries.
\end{example}

\end{document}